\documentclass[11pt]{amsart}
\usepackage{amssymb}
\usepackage{amsmath}
\usepackage{amsthm}
\usepackage{color}

\makeatletter
\def\LaTeX{\leavevmode L\raise.42ex
    \hbox{\kern-.3em\size{\sf@size}{0pt}\selectfont A}\kern-.15em\TeX}
\makeatother

\newcommand{\BibTeX}{{\rm B\kern-.05em{\sc
          i\kern-.025emb}\kern-.08em\TeX}}

\makeatletter
\def\@currentlabel{2.1}\label{e:dispaa}
\def\@currentlabel{2.21}\label{e:dispau}
\def\@currentlabel{2.22}\label{e:dispav}
\def\@currentlabel{2.23}\label{e:dispaw}
\def\@currentlabel{2.24}\label{e:dispax}
\def\theequation{\thesection.\@arabic\c@equation}
\makeatother

\newcounter{mnotecount}[section]

\newcommand{\rmnote}[1]{}

\renewcommand{\theequation}{\arabic{section}.\arabic{equation}}
\newtheorem{thm}{Theorem}[section]
\newtheorem{lem}[thm]{Lemma}
\newtheorem{cor}[thm]{Corollary}
\newtheorem{prop}[thm]{Proposition}
\theoremstyle{definition}
\newtheorem{defn}{Definition}[section]
\newtheorem{rem}{Remark}[section]
\newcommand{\B}{\mathbb B}
\newcommand{\M}{\mathbb{M}}
\newcommand{\R}{\mathbb{R}}

\allowdisplaybreaks

\begin{document}

\title[Forward self-similar solutions ]{Forward self-similar solutions of the fractional Navier-Stokes Equations}

\author[B. Lai]{Baishun Lai }

\address{ Institute of Contemporary Mathematics, Henan University, Kaifeng 475004, P.R. China}

\email{laibaishun@henu.edu.cn}

\author[C. Miao]{Changxing Miao}

\address{Institute of Applied Physics and Computational Mathematics, P.O. Box 8009, Beijing 100088, P.R. China }

\email{miao\_{}changxing@iapcm.ac.cn}

\author[X. Zheng]{Xiaoxin Zheng}

\address{ School of Mathematics and Systems Science, Beihang University, Beijing 100191, P.R. China}

\email{xiaoxinzheng@buaa.edu.cn}

 \date{\today}
 \subjclass[2000]{35Q30; 35B40;  76D05.}
 \keywords{Self-similar solution; nonlocal effect; blowup argument; the weighted estimate}
 \maketitle
\begin{abstract}
 We  study  forward self-similar solutions to the 3-D Navier-Stokes equations with the fractional diffusion $(-\Delta)^{\alpha}.$  First, we construct a global-time forward self-similar solutions to the fractional Navier-Stokes equations with $5/6<\alpha\leq1$  for arbitrarily large self-similar initial data by making use of the
  so called blow-up argument. Moreover, we prove that this solution is smooth in $\mathbb R^3\times (0,+\infty)$. In particular, when $\alpha=1$, we prove that
the solution constructed by  Korobkov-Tsai \cite[Anal. PDE 9 (2016), 1811-1827]{KT}  satisfies  the decay estimate by establishing  regularity of
solution for the corresponding elliptic system, which implies this solution has the  same properties as a solution which was constructed in
\cite[Jia and \v{S}ver\'{a}k, Invent. Math. 196 (2014), 233-265]{JS}.
\end{abstract}

\setcounter{equation}{0}
 \setcounter{equation}{0}
\section{Introduction}
In this paper, we consider the generalized incompressible Navier-Stokes equations with fractional Laplace operator:
\begin{equation}\label{E1.1}
\left\{
\begin{aligned}
&u_{t}+(-\Delta)^{\alpha} u+u\cdot\nabla u+\nabla p=0, \quad (t,x)\in(0,+\infty)\times \R^{3}\\
& \mbox{div}\, u=0,   \qquad\; (t,x)\in(0,+\infty)\times \R^{3} \\
\end{aligned} \right.
\end{equation}
with initial condition
\begin{align}\label{E1.1-}
u(x,0)=u_{0}(x), \; \; \;\; x\in \R^{3}.
\end{align}
Here the column vector $u=(u_{1}, u_{2}, u_{3})^{t}$ denotes the velocity field, the scalar function $p$ stands for the pressure which can be recovered at least formally from $u$ via Calder\'{o}n-Zygund operators,   and the fractional Laplacian $(-\Delta)^{\alpha}$ with $0<\alpha<1$ is a nonlocal operator (also called L\'{e}vy operator) defined by
$$
(-\Delta)^{\alpha} u(x)\triangleq c_{n,\alpha} {\rm p.v.}\int_{\R^{n}}(u(x)-u(z))\frac{{\rm d}z}{|x-z|^{n+2\alpha}},\ \ u\in \mathcal{S}(\R^{n}),
$$
 where $\mathcal{S}(\R^{n})$  is the Schwartz class of smooth real or complex-valued rapidly decreasing functions,
 {\rm p.v.} stands for the Cauchy principle value, i.e.
$$
(-\Delta)^{\alpha} u(x)=c_{n,\alpha}\lim_{\epsilon\to0}\int_{\{z\in\R^{n}: |y|>\epsilon\}}(u(x)-u(z))\frac{{\rm d}y}{|x-y|^{n+2\alpha}}
$$
with
$$c_{n,\alpha}\triangleq\frac{\alpha(1-\alpha)4^{\alpha}\gamma(\frac{n}{2}+\alpha)}{\gamma(2-\alpha)\pi^{\frac{n}{2}}}, \;\;\; \gamma(r)=\int_{0}^{+\infty}s^{r-1}e^{-s} \; {\rm d}s, \;\; r>0. $$
 Alternatively,  the fractional operator $(-\Delta)^{\alpha}$ can be equivalently defined by the Fourier transform:
$$
\widehat{(-\Delta)^{\alpha}u}(\xi)=\big(|\xi|^{2\alpha}\big)\widehat{u}(\xi).
$$
Recently, there is an increasing interest for studying the fractional Navier-Stokes equations \eqref{E1.1}, since they naturally appear in hydrodynamics, statistical mechanics, physiology, certain combustion models, and so on \cite{R,W}. As a simplified model of Eqs. \eqref{E1.1}, the following fractional Burgers equation has been studied by many mathematicians and physicists such as Biler-Funaki-Woyczynski \cite{Biler} and Kiselev-Nazarov-Schterenberg \cite{Kiselev},
$$
u_{t}+(-\Delta)^{\alpha} u+uu_{x}=0,\;\;\; t>0, \;\; u(x,0)=\varphi(x),
$$
where $u(x,t): \R\times (0,+\infty)\to\R$.  For the fractional Navier-Stokes equations \eqref{E1.1}, the existence and uniqueness of solutions have been established by Wu \cite{Wu} in the framework of Besov spaces,  and by Zhang \cite{Zhang} via probabilistic approach.
More recently, Tang-Yu \cite{Tang} established the partial regularity of the suitable weak solution of problem \eqref{E1.1} in the $L^2$-framework, which can be viewed as a generalization of the CKN regularity criterion \cite{Caffarelli} for the following classical Navier-Stokes system:
\begin{equation}\label{E1.2}
\left\{
\begin{aligned}
&u_{t}-\Delta u+u\cdot\nabla u+\nabla p=0, \quad (t,x)\in(0,+\infty)\times \R^{3}\\
& \mbox{div}\, u=0,   \qquad\; (t,x)\in(0,+\infty)\times \R^{3} \\
\end{aligned} \right.
\end{equation}
\begin{align}\label{E1.2-}
u(x,0)=u_{0}(x) \ \ \ \ \text{in}\ \ \R^{3}.
\end{align}
Similar to  the above classical Navier-Stokes equations \eqref{E1.2}-\eqref{E1.2-}, the fractional Navier-Stokes system \eqref{E1.1}-\eqref{E1.1-}
also enjoys the scaling property. Specifically, if $(u,p)$ is the solution of equations \eqref{E1.1}-\eqref{E1.1-}, then, for all $\lambda>0$, $(u_\lambda,p_\lambda)$ is also the solution of equations \eqref{E1.1} corresponding to the initial data $u_{0\lambda},$ where
\begin{align*}
  u_{\lambda }(x,t)\triangleq\lambda^{2\alpha-1}u(\lambda x,\lambda^{2\alpha}t), \quad
  p_{\lambda }(x,t)\triangleq\lambda^{4\alpha-2}p(\lambda x,\lambda^{2\alpha}t)
  \end{align*}
  and
\[ u_{0 \lambda }(x)\triangleq\lambda^{2\alpha-1}u_{0}(\lambda x).\]
According to this scaling property, we want to investigate a solution which is invariant under this scaling.
 We call such solutions \emph{as self-similar solutions} which include two types: one is {\em a  forward self-similar solution}, another is {\em a backward self-similar solution}.
 A {\em  forward self-similar solution} is a solution on $\R^{3}\times(0,+\infty) $  such that for every $\lambda>0,$
$$u(x,t)=u_\lambda(x,t)\quad\text{and} \quad p(x,t)= p_\lambda(x,t).$$
  A {\em  backward self-similar solution} is a solution on $\R^{3}\times(-\infty,0) $ such that  for every $\lambda>0,$
$$u(x,t)=u_\lambda(x,t)\quad\text{and} \quad p(x,t)= p_\lambda(x,t).$$
A  question on the existence of the self-similar blow-up solution of \eqref{E1.2} was initially stated  by Leray~\cite{Ler}.
According to the above definition, it is easy to verify that the trivial solution is a trivial backward self-similar solution. However, the nontrivial self-similar blow-up solution with finite energy
does not exist, which was firstly proved by Ne\v{c}as-R{\aa u}\v{z}i\v{c}ka-\v{S}ver\'{a}k in \cite{Ne}.
Later, Tsai~\cite{Tsai} further proved the nonexistence of the backward self-similar solutions with local finite energy.
%  Also, an alternative proof of
%{\color{blue} Ne\v{c}as-R{\aa u}\v{z}i\v{c}ka-\v{S}ver\'{a}k's work in \cite{Ne}} was shown in Escuriaza-Seregin-\v{S}ver\'{a}k \cite{ESS} which solved a famous problem concerning on regularity of $L_{3,\infty}$-solutions.

In contrast with the case of backward self-similar solutions,  several results of nontrivial forward self-similar solutions were established in the past years. In  \cite{C1,C2}, Cannone-Meyer-Planchon firstly proved the  existence and uniqueness of the small forward self-similar solutions in the framework of  homogeneous Besov spaces, see also for examples Barraza \cite{B} in Lorentz space $L^{(3,\infty)}(\mathbb R^3)$, and Koch and Tataru \cite{KT} in $\text{BMO}^{-1}(\R^3)$.

For  large scale-invariant initial values, it is well-known that the perturbation argument  such as the contraction mapping no longer works,
and one attempts to seek  other methods  to establish existence of self-similar solution.
Recently, Jia and \v{S}ver\'{a}k \cite{JS} constructed large  self-similar solutions
by developing so called local-in-space regularity estimates near the initial time and  applying the {\em Leray-Schauder fixed-point theorem}.
Later on ,
Korobkov-Tsai \cite{KT} proposed an alternative method of constructing self-similar solutions without pointwise bound via so called blow-up argument.
For the sake of convenience, we first recall the framework which was developed in \cite{JS} and \cite{KT}.

 Formally, one
can reduce the study of problem \eqref{E1.2} into that of the corresponding integral equation. More precisely,
seeking a self-similar solution $u(x,t)$ of \eqref{E1.2} is equivalent to find a self-similar solution of
$$
u=e^{t\Delta}u_{0}-\int_{0}^{t}e^{(t-s)\Delta}\mathbb{P}\big(\nabla\cdot(u\otimes u)\big)\,{\rm d}s,
$$
where $\mathbb{P}={\rm Id}-\nabla(\Delta)^{-1}{\rm div}$ is called the Leray-Hopf projection onto the divergence-free vector fields.
 Using the homogeneity of $u_{0}$ and the self-similarity of $u=t^{-1/2}U(t^{-1/2}x)$, the above problem is equivalent to  find a solution $U$ of
 the equation
\begin{equation}\label{E1.3}
U=V_{0}-\mathcal{T}(U)
\end{equation}
with
$$
V_{0}=\int_{\R^{3}}G_{1}(x-z)u_{0}(z)\,{\rm d}z
$$
and
$$
\mathcal{T}(U)=\sum_{j=1}^{3}\int_{0}^{1}\int_{\R^{3}}\frac{1}{(1-\tau)^{2}}
\Big(\partial_{j}\mathcal{O}\Big(\frac{x-z}{(1-\tau)^{\frac{1}{2}}}\Big)\Big)::\tau^{-1} U_{j}(\frac{z}{\tau^{\frac{1}{2}}})U(\frac{z}{\tau^{\frac{1}{2}}})\,{\rm d}z{\rm d}\tau.
$$
Here $G_{1}(x)$ is the profile of the heat kernel at $t=1$ and  $\mathcal{O}=(\mathcal{O}_{j,k})$ is Oseen's kernel with
$$\mathcal{O}_{j,k}(x)=\delta_{jk}G_1(x)+\Gamma*\partial_j\partial_kG_1, \;\; \text{where}\; \Gamma(x)\;\text{  is Newton potential}.$$

To solve equation \eqref{E1.3}, it suffices to verify  that $\mathcal{T}$ satisfies all the requirements of the Leray-Schauder principle in some selected Banach space $X$:

(i)\; $\mathcal{T}: X\buildrel{}\over{-\!\!\!\!\rightarrow}X$ is a continuous and compact operator.

(ii)\; There exists a constant $C$ such that, for every $\lambda\in[0,1],$ $$U=\lambda(U_{0}-\mathcal{T}(U))\;\Longrightarrow\; \|U\|_{X}\leq C.$$

For (i), the key point is to find a suitable functional-analytic setup to obtain compactness  of operator.
For (ii), the main difficulty of this step is to  establish an a-priori estimate of solutions, which help us  to
apply the continuation method to solve \eqref{E1.3}. In order to overcome these difficulties, Jia and \v Sver\'ak \cite{JS} developed   the so-called
local-in-space regularity estimates near the initial time $t=0$  to establish the H\"{o}lder estimate for local-Leray solutions constructed by Lemari\'{e}-Rieusset \cite{LR}. This estimate  enables them to obtain regularity  of self-similar solutions outside the ball, and then they got
 a better decay estimate of such solution for large $|x|$, which ensures the operator $\mathcal{T}$ to be compact in a suitable setting.

  Our goal in this paper is to apply the {\em Leray-Schauder fixed-point theorem} to  construct a forward self-similar solution $u(x,t)$ of equation
  \eqref{E1.1}  which takes  the form
$$u(x,t)=t^{\frac{1-2\alpha}{2a}}u\Big(\frac{x}{t^{\frac{1}{2\alpha}}}, 1\Big)\triangleq t^{\frac{1-2\alpha}{2a}}U
\Big(\frac{x}{t^{\frac{1}{2\alpha}}}\Big)\; \; \; \mbox{with}\; \; U(x)=u(x,1),$$
when the  corresponding initial value $u_{0}(x)$ satisfies  following scaling:
$$u_{0}(x)=\lambda^{2\alpha-1}u_{0}(\lambda x)\ \ \;\;\;\text{for all}\,\,\, \lambda>0.$$
Setting $U_{0}=e^{-(-\Delta)^{\alpha}}u_{0}$, we easily find that the difference  $V\triangleq U-U_{0}$ solves
the following  fractional elliptic equation
\begin{equation}\label{E1.4a}
(-\Delta)^{\alpha}V+\nabla P= \frac{2\alpha-1}{2\alpha}V(x)+\frac{1}{2\alpha}x\cdot \nabla V-U_{0}\cdot\nabla U_{0}
-(U_{0}+V)\cdot\nabla V-V\cdot\nabla U_{0}.
\end{equation}
  According to the Leray-Schauder fixed point theorem, the main task is now to establish  regularity estimates for  solutions of problem \eqref{E1.4a}.
  To do this, we need to overcome two difficulties. First,  the argument of Tang-Yu \cite{Tang} seems to be infeasible for Lemari\'{e}-Rieusset's solution \cite{LR}  in the framework of uniformly locally square space $L_{\rm uloc}^{2}(\R^3)$. This show  that the methods used in \cite{JS} does not directly apply to  our problem.
  a second class of difficulties concerns  the fractional diffusion operator which is a nonlocal operator. To overcome both difficulties, we will adopt the  following regularization of  the fractional Navier-Stokes equations~\eqref{E1.1} by adding an artificial diffusion $\epsilon\Delta V $
\begin{equation}\label{E1.4}
-\epsilon\Delta V+(-\Delta)^{\alpha}V+\nabla P=\lambda\Big(\frac{2\alpha-1}{2\alpha}V(x)+\frac{1}{2\alpha}x\cdot \nabla V-U_{0}\cdot\nabla U_{0}
-(U_{0}+V)\cdot\nabla V-V\cdot\nabla U_{0}\Big)
\end{equation}
By the blow-up argument used in Korobkov-Tsai~\cite{KT},  we firstly show that  $V(x)$ of \eqref{E1.4} satisfies the following a priori estimate
$$
\int_{\R^{3}}\Big(\epsilon|\nabla V|^{2}+|(-\Delta)^{\frac{\alpha}{2}}V|^{2}+\frac{5-4\alpha}{4\alpha}|V|^{2}\Big)\,{\rm d}x\leq C(U_{0}).
$$
This a priori estimate helps us to  prove  that the equation \eqref{E1.4} possess at least one  distributional weak solution.
 However, this uniform estimate is not sufficient to obtain the natural pointwise estimate of a self-similar solution for $\alpha=1$ which was established in \cite{JS}. In fact, since $|x|$ is not bounded, it is difficult  to get higher  regularity by the classical elliptic theory,  directly.
Thus, we develop a new technique  to obtain  higher  regularity of a solution $V$ of equation \eqref{E1.4a}.
First, we choose an appropriate test function $\varphi$ in the weighted-$H^{1}(\R^{3})$ space, and then we derive the following key estimate
$$
|x|V(x)\in H^{1}(\R^{3}).
$$
Based on this regularity, we  show following behavior of $V$ for large $|x|$,
 $$|V(x)|\leq \frac{ C}{1+|x|}\qquad\text{for all}\,\,\,x\in\R^3.$$
With this decay estimate in hand, we eventually get by the property of the fundamental solution that
$$
|V(x)|\leq \frac{C}{(1+|x|)^{3}}\log(1+|x|)\qquad\text{for all}\,\,\,x\in\R^3.
$$
Now we state our main result as follows:
\begin{thm}\label{T1.1}
Assume $\frac{5}{6}<\alpha\leq1$. Let  $u_{0}=\frac{\sigma(x)}{|x|^{2\alpha-1}} \text{ with }\sigma(x)=\sigma(x/|x|)\in L^{\infty}(S^{2})$, which satisfies $\mathrm{div}\, u_{0}=0$ in $\mathbb R^3\setminus \{0\}$. Then problem \eqref{E1.1}-\eqref{E1.1-}  admits at least one forward self-similar
solution $u\in \mathrm{BC}_{\rm w}\big ([0,+\infty),\,L^{(\frac{3}{2\alpha-1},\infty)}(\R^3)\big)$  such that
\begin{itemize}
\item  for each $p\in \big[2,\frac{6}{3-2\alpha}\big]$, $\big\|u(t)-e^{-t(-\Delta)^{\alpha}}u_{0}\big\|_{L^{p}(\R^{3})}\leq Ct^{\frac{1}{2\alpha}(1+\frac{3}{p})-1};$
\item  $u(x,t)$ is smooth in $\R^{3}\times (0,+\infty);$
\item for $\alpha=1$, then we have the following pointwise estimates
\begin{equation}\label{E1.6}
|u(x,t)|\leq\frac{C}{|x|+\sqrt{t}}\quad\text{and}\quad \left|u(x,t)-e^{t\Delta}u_{0}\right|\leq \frac{Ct}{|x|^{3}+t^{\frac32}}\log\Big(1+\frac{|x| }{\sqrt{t}}\Big)
\end{equation}
for all $(x,t)\in \R^{3}\times(0,+\infty)$.
\end{itemize}
\end{thm}

\begin{rem}

When $\alpha=1$, Jia-Sverak prove the existence of
forward self-similar solution and the associate wise-point estimates \eqref{E1.6} in \cite{JS}.
Korobkov-Tsai~\cite{KT}
give another proof the existence of forward self-similar solution was shown in \cite{KT} via
the blowup argument. But they did not show that the solution has the decay estimate.
 In Theorem \ref{T1.1}, we obtain this estimate by developing  new weighted-$H^1(\R^3)$ space
of weak solution $V(x)$ to system (1.6). This answers the problem posed in Korobkov-Tsai~\cite{KT}.
In other words, we give an alternative construction method of the existence of forward self-similar
solution.
\end{rem}
\begin{rem}
According to the scaling analysis, it is well-known that system (1.1) has the same
scaling with the following nonlinear equation
\begin{equation}\label{eq.sim}
u_t+(-\Delta)^\alpha u=|u|^{\frac{4\alpha-1}{2\alpha-1}-1}u\qquad \text{in}\,\,\,\,\R^3\times(0,+\infty).
\end{equation}
Now we consider the stationary solution $U$ of problem \eqref{eq.sim}, which solves
\begin{equation*}
(-\Delta)^\alpha U=|U|^{\frac{4\alpha-1}{2\alpha-1}-1}U\qquad \text{in}\,\,\,\,\R^3.
\end{equation*}
Denote kinetic energy $K\triangleq\|U\|_{\dot{H}^\alpha(\R^3)}$ and potential energy $P\triangleq\|U\|_{L^{\frac{2(3\alpha-1)}{2\alpha-1}}(\R^3)}.$ Thanks to the
embedding theorem, we find that the kinetic energy can not control the potential energy if $\alpha<\frac56$. Inspired by this analysis,
we call system \eqref{E1.4a} is supcritical if $\alpha<\frac56$. Because  super-criticality usually implies that the kinetic
energy can not control the nonlinearity, we give a roughly explanation on condition $\alpha\in(5/6,1]$ in Theorem \ref{T1.1}.
\end{rem}

\begin{rem}
In fact, we also proved that system \eqref{E1.1}-\eqref{E1.1-}  admits at least one forward self-similar
solution of the following form $u=u_{\rm L}+v$  such that
 $ u_{\rm L}\in \mathrm{BC}_{\rm w}\big ([0,+\infty),\,L^{\frac{3}{2\alpha-1}}(\R^3)\big) $  for $\frac58<\alpha\leq\frac56$  and $$\big\|v(t)\big\|_{L^{p}(\R^{3})}\leq Ct^{\frac{1}{2\alpha}(1+\frac{3}{p})-1}\,\,\qquad\text{for each}\quad p\in \big[2,\frac{6}{3-2\alpha}\big].$$
 Since $\frac58<\alpha\leq\frac56$, we see that $\frac{6}{3-2\alpha}\leq \frac{3}{2\alpha-1} $ and $\frac{1}{2\alpha}(1+\frac{3}{p})-1>0$. This implies $v\in L^\infty_{\rm loc}([0,+\infty);L^p(\R^3))$ for $p\in \big[2,\frac{6}{3-2\alpha}\big]$.  Moreover, the solution $u$ satisfies problem \eqref{E1.1}-\eqref{E1.1-}  in the sense of distribution.
\end{rem}
\medskip

\noindent\textbf{Notation:}  Let  $\langle x\rangle =(1+|x|^2)^{\frac12}$. We denoted by $\M^{3}$ the space of all real $3\times3$ matrices. Adopting summation over
repeated Latin indices, running from $1$ to $3$, we denote
\[A:B=A_{ij}B_{ij},\ \ |A|=\sqrt{A:A},\ \ \  A=(A_{ij}),\ \ B=(B_{ij})\in \M^{3};\]
\[u\otimes v= (u_{i}v_{j})\in \M^{3},\,\, Au=\big(A_{ij}u_{j}\big)\in\R^{3}, \quad u,v\in\R^{3}\text{ and } A\in \M^{3}.\]
We define
\[\|f\|_{p,\,\lambda}\triangleq\sup_{x\in\mathbb{R}^3}
\Big(\int_{|x-y|<\lambda}|f(y)|^{p}\text{d} y\Big)^{\frac{1}{p}},\]
\[L^p_{\rm ul}(\R^3)\triangleq\big\{f\in L^1_{\rm loc}(\R^3);\,\|f\|_{p,\,1}<\infty\big\},\quad
\|f\|_{L^p_{\rm ul}(\R^3)}\triangleq\|f\|_{p,\,1}.\]
The rest of the paper is structured as follows: Section 2 contains preliminaries which consist of some basic functional spaces, notations and some standard facts on non-local heat operator. In Section 3, we study the existence and regularity of solution  of the corresponding elliptic problem by establishing some a priori estimates, which is the core of our paper.
In Section 4, we give the proof of the Theorem \ref{T1.1}.

\setcounter{equation}{0}
 \setcounter{equation}{0}
\section{Preliminaries}
\subsection{Functional spaces, Littlewood-Paley theory and several useful lemmas}
In this subsection, we firstly review the statement of functional spaces, see for example \cite{Ga}.
Let us begin by  defining the space weak-$L^{p}(\R^{3})$, denoted by $L^{(p,\infty)}(\R^{3})$ as follows:
$$
L^{(p,\infty)}(\R^{3})\triangleq \Big\{u: m\{x\in\R^{3}: |u(x)|>s\}\leq \frac{A}{s^{p}}\ \ \mbox{with some}\ A>0\ \mbox{and all}\ s>0 \Big\}
$$
with norm
$$ \|\cdot\|_{L^{(p,\infty)}(\R^{3})}=\Big(\sup_{s>0} s^{p}\,m\{x\in\R^{3}:\, |u(x)|>s\}\Big)^{\frac{1}{p}}.$$
 Here $m$ denotes the Lebesgue measure on $\R^{3}$.

Next, we recall some basic function spaces in bounded domain.
 Let $\Omega$ be an open set in $\R^{3}$,  $D'(\Omega)=(C_c^{\infty}(\Omega))'$. Let $\alpha\in(0,1)$ and define the fractional Sobolev space $H^{\alpha}(\Omega)$ as following
$$
H^{\alpha}(\Omega)\triangleq\Big\{u\in D'(\Omega): \; u(x)\in L^{2}(\Omega) \;\;\;\text{and}\;\;\frac{|u(x)-u(y)|}{|x-y|^{\frac{3}{2}+\alpha}}\in L^{2}(\Omega\times\Omega)\Big\}
$$
with the norm
$$
\|u\|_{H^{\alpha}(\Omega)}\triangleq\bigg(\int_{\Omega}|u|^{2}\,\mathrm{d}x+\int_{\Omega}\int_{\Omega}\frac{|u(x)-u(y)|^{2}}{|x-y|^{3+2\alpha}}\,{\rm d}x{\rm d}y\bigg)^{\frac{1}{2}}.
$$
 In particular, when $\alpha=1$, the  Sobolev space $H^{1}(\Omega)$  can be defined as
$$
H^{1}(\Omega)\triangleq\Big\{u(x)\in D'(\Omega):\,\, u\in L^{2}(\Omega)\,\,\,\text{and}\,\,\,\nabla u\in L^{2}(\Omega)\Big\}
$$
with the norm
$$
\|u\|_{H^{1}(\Omega)}\triangleq\Big(\int_{\Omega}|\nabla u|^{2}+|u|^{2}\,{\rm d}x\Big)^{\frac{1}{2}}.
$$
One easily see that  $C^{\infty}(\Omega)$ is dense in $H^{\alpha}(\Omega)$. If $\Omega$ is a domain with Lipschitz boundary, then there exists a bounded linear extension operator from
$H^{\alpha}(\Omega)$ to $H^{\alpha}(\R^{3})$. Note that $H^{\alpha}(\R^{3})$ with the norm $\|\cdot\|_{H^{\alpha}(\R^{3})}$ is equivalent to the space
$$
\Big\{u\in L^{2}(\R^{3}): \quad|\xi|^{\alpha}\mathcal{F}(u)(\xi)\in L^{2}(\R^{3})\Big\}
$$
with the norm
$$
\|\cdot\|_{L^{2}(\R^{3})}+\big\||\xi|^{\alpha}\mathcal{F}(\cdot)(\xi)\big\|_{L^{2}(\R^{3})},
$$
where $\mathcal{F}$ denotes the Fourier transform. It is known that (see \cite{LM}) there exists $C>0$ depending only on  $\alpha$ such that for
$U\in H^{1}(\R_{+}^{4},\,t^{1-2\alpha}\,{\rm d}x{\rm d}t)\bigcap C(\overline{\R_{+}^{4}})$
$$ \|U(\cdot,0)\|_{H^{\alpha}(\R^{3})}\leq C\|U\|_{H^{1}(\R_{+}^{4},\,t^{1-2\alpha}\,{\rm d}x{\rm d}t)},$$
where $H^{1}(\R_{+}^{4},\,t^{1-2\alpha}\,{\rm d}x{\rm d}t)$ is defined as
$$
H^{1}\big(\R_{+}^{4},\,t^{1-2\alpha}\,{\rm d}x{\rm d}t\big)\triangleq \Big\{U:\,\, U\in L^{2}\big(\R_{+}^{4},\,t^{1-2\alpha}\,{\rm d}x{\rm d}t\big)\,\text{  and } \,\nabla U\in L^{2}\big(\R_{+}^{4},t^{1-2\alpha}\,{\rm d}x{\rm d}t\big)\Big\}
$$
with
$$\|U\|_{L^{2}\big(\R_{+}^{4},t^{1-2\alpha}\,{\rm d}x{\rm d}t\big)}\triangleq \Big(\int_{0}^{+\infty}\int_{\R^{3}}t^{1-2\alpha}U^{2}(x,t)\,{\rm d}x{\rm d}t\Big)^{\frac{1}{2}}.$$
We know that  every $U\in H^{1}\big (\R_{+}^{4},\,t^{1-2\alpha}\,{\rm d}x{\rm d}t\big)$ has well-defined trace $u\triangleq U(\cdot,0)\in H^{\alpha}(\R^{3})$
 by a standard density argument.

We define $H_{0}^{\alpha}(\Omega)$ as the completion of $C_{0}^{\infty}(\Omega)$ under
the norm $\|\cdot\|_{\dot{H}^{\alpha}(\R^{3})}$.  When $\alpha=1$, $H_{0}^{1}(\Omega)$ the completion of $C_{0}^{\infty}(\Omega)$ under
the following equivalent norm
$$
\|u\|_{H^{1}_0(\Omega)}\triangleq\bigg(\int_{\Omega}|\nabla u|^2(x)\,{\rm d}x\bigg)^{\frac{1}{2}}.
$$
 We also denote
$$
H_{0,\sigma}^{1}(\Omega)\cap H_{0,\sigma}^{\alpha}(\Omega)=\big\{f\in H_{0}^{1}(\Omega)\cap H_{0}^{\alpha}(\Omega):\,\, \mbox{div}\ f=0\big\}.
$$
Let $X$ be a Banach space, we denote by $X'$ the dual space of $X$ with respect to the norm
$$
\|f\|_{X'}\triangleq\sup_{\varphi\in X, \|\varphi\|_{X}\leq1}\Big|\int_{\Omega}f\,\varphi \,{\rm d}x\Big|,\ \ \ \mbox{for any}\ f\in X'.
$$
For $q>1$ denote by $D^{1,q}(\Omega)$ to be
$$
D^{1,q}(\Omega)\triangleq \Big\{f:\,\, f\in W_{\rm loc}^{1,q}(\Omega)\ \ \mbox{and}\ \ \|f\|_{D^{1,q}(\Omega)}=\|\nabla f\|_{L^{q}(\Omega)}<+\infty\Big\}.
$$
Further, denote by $D_{0}^{1,2}(\Omega)$ the closure of $C_{0}^{\infty}(\Omega)$ in $D^{1,2}(\Omega)$ and $$H(\Omega)\triangleq\left\{u:\, \,u\in D_{0}^{1,2}(\Omega)\ \ \mbox{and}\ \ \mbox{div}\ u=0\right\}.$$
Besides, when $\Omega$ is bounded and locally Lipschitz,
 $u\in D^{1,q}(\Omega)$ implies $u\in W^{1,q}(\Omega)$, for details one refers to \cite{Ga}. Finally, we denote the Bessel potential space by
$$
H_{p}^{\alpha}(\R^{n})=\left\{u\in L^{p}(\R^{3}):\,\, (I-\Delta)^{\frac{\alpha}{2}}u\in L^{p}(\R^{3})\right\},
$$
which is equipped with the norm
$$
\|u\|_{H_{p}^{\alpha}(\R^{3})}=\big\|(I-\Delta)^{\frac{\alpha}{2}}u\big\|_{L^{p}(\R^{3})},
$$
and the homogeneous space by
$$
\dot{H}_{p}^{\alpha}(\R^{3})=\left\{u\in D'(\R^{3}):\,\, (-\Delta)^{\frac{\alpha}{2}}u\in L^{p}(\R^{3})\right\}
$$
with the semi-norm
$$
\|u\|_{\dot{H}_{p}^{\alpha}(\R^{3})}=\big\|(-\Delta)^{\frac{\alpha}{2}}u\big\|_{L^{p}(\R^{3})}.
$$

Next, we  review the so-called Littlewood-Paley decomposition described, e.g., in \cite{BCD11}.
Suppose that  $(\chi,\varphi)$ be a couple of smooth functions with  values in $[0,1]$
such that  $\text{supp}\,\chi\subset \big\{\xi\in\mathbb{R}^{3}\big||\xi|\leq\frac{4}{3}\big\}$,
$\text{supp}\,\varphi\subset\big\{\xi\in\mathbb{R}^{3}\,\big|\,\frac{3}{4}\leq|\xi|\leq\frac{8}{3}\big\}$ and
\begin{equation*}
    \chi(\xi)+\sum_{j\in \mathbb{N}}\varphi(2^{-j}\xi)=1\qquad \forall\,\xi\in \mathbb{R}^{3}.
\end{equation*}
For any $u\in \mathcal{S}'(\mathbb{R}^{3})$, let us define
\begin{equation*}
 \Delta_{-1}u\triangleq\chi(D)u\quad\text{and}\quad   {\Delta}_{j}u\triangleq\varphi(2^{-j}D)u\qquad \forall\,j\in\mathbb{N}.
\end{equation*}
Moreover, we can define the low-frequency cut-off:
\begin{equation*}
    {S}_{j}u\triangleq\chi(2^{-j}D)u.
\end{equation*}
So, we easily find  that
\begin{equation*}
    u=\sum_{j\geq-1}{\Delta}_{j}u,\quad \text{in}\quad\mathcal{S}'(\mathbb{R}^{3})
\end{equation*}
which  corresponds to  the \emph{inhomogeneous Littlewood-Paley decomposition}.
In usual, we always use the following properties of quasi-orthogonality:
\begin{equation*}
    {\Delta}_{j}{\Delta}_{j'}u\equiv 0\quad \text{if}\quad |j-j'|\geq 2.
\end{equation*}
\begin{equation*}
{\Delta}_{j}({S}_{j'-1}u{\Delta}_{j'}v)\equiv0\quad \text{if}\quad |j-j'|\geq5.
\end{equation*}
We shall also use the \emph{homogeneous Littlewood-Paley} operators governed by
\begin{equation*}
    \dot{S}_{j}u\triangleq\chi(2^{-j}D)u \quad\text{and}\quad {\Delta}_{j}u\triangleq\varphi(2^{-j}D)u\qquad\forall\,j\in\mathbb{Z}.
\end{equation*}
We denoted by  $\mathcal{S}_{h}'(\mathbb R^3)$ the space of tempered distributions $u$ such that
\begin{equation*}
    \lim_{j\rightarrow-\infty}\dot{S}_{j}u=0\quad \text{in} \quad \mathcal{S}'(\mathbb R^3).
\end{equation*}
 The \emph{homogeneous Littlewood-Paley decomposition}  can be written as
 \begin{equation*}
    u=\sum_{j\in\mathbb Z}{\Delta}_{j}u,\quad \text{in}\quad\mathcal{S}'_h(\mathbb{R}^{3})
\end{equation*}

\begin{defn}\label{def2.2}
Assume that  $s\in \mathbb{R}$, $(p,q)\in [1,\infty]^{2}$ and $u\in \mathcal{S}'(\mathbb{R}^{3})$. Then we  define the \emph{ homogeneous Besov spaces} as
\begin{equation*}
   \dot{B}^{s}_{p,q}(\mathbb{R}^{3})\triangleq\big\{u\in\mathcal{S}'(\mathbb{R}^{3})\big|\,\,\|u\|_{\dot{B}^{s}_{p,q}(\mathbb{R}^{3})}<+\infty\big\},
\end{equation*}
where
\begin{equation*}
  \|u\|_{{B}^{s}_{p,q}(\mathbb{R}^{3})}\triangleq\begin{cases}
    \Big( \displaystyle{\sum_{j\in\mathbb Z}}\,2^{jsq}\|{\Delta}_{j}u\|_{L^{p}(\mathbb{R}^{3})}^{q}\Big)^{\frac{1}{q}}
    \quad&\text{if}\quad q<+\infty,\\
   \displaystyle{\sup_{j\in\mathbb Z}}\,2^{js}\| {\Delta}_{j}u\|_{L^{p}(\mathbb{R}^{3})}\quad&\text{if}\quad q=+\infty.
    \end{cases}
\end{equation*}
\end{defn}
Before we conclude this section, we recall a useful  Sobolev embedding theorem:
\begin{lem}[\cite{Ga}]\label{L2.8}
Let $2\leq q\leq 2_{\alpha}^{*}=\frac{6}{3-2\alpha}$ and $u\in H^{\alpha}(\Omega)$, then
$$
\|u\|_{L^{q}(\Omega)}\leq C\|u\|_{H^{\alpha}(\Omega)},\quad  \;\; \Omega\subset \mathbb R^3\;\;{or }\;\; \Omega=\mathbb R^3.
$$
If $2\leq q< 2_{\alpha}^{*}$, and $\Omega\subset \R^{n}$ is a bounded domain, then we have the following compact embedding
$$H_{0}^{\alpha}(\Omega)\hookrightarrow\hookrightarrow L^q(\Omega),$$
in other words,  for every  bounded sequence $\{u_{k}\}\subset H_{0}^{\alpha}(\Omega)$, there exists a converging subsequence, still denote by  $\{u_{k}\}$, such that
$$ \lim_{k\rightarrow+\infty} u_k(x)=u(x)\in L^{q}(\Omega).$$
\end{lem}
\begin{lem}[\cite{LGrafakos,O}]\label{GHY}

\begin{enumerate}
	\item[\rm(i)]  Let $1<p,q,r<\infty$, $0< s_1,s_2\leq\infty,$ $\frac1p+\frac1q=\frac1r+1$, and $\frac{1}{s_1}+\frac{1}{s_2}=\frac1s.$ Then there holds
               \[\|f\ast g\|_{L^{r,s}(\R^3)}\leq C(p,q,s_1,s_2)\|f\|_{L^{p,s_1}(\R^3)}\|g\|_{L^{q,s_2}(\R^3)}.\]
               \item [\rm (ii)]\;Let $0<p,q,r\leq\infty$, $0< s_1,s_2\leq\infty,$ $\frac1p+\frac1q=\frac1r$, and $\frac{1}{s_1}+\frac{1}{s_2}=\frac1s.$ Then we have the H\"older inequality for Lorentz spaces
                   \[\|fg\|_{L^{r,s}(\R^3)}\leq C(p,q,s_1,s_2)\|f\|_{L^{p,s_1}(\R^3)}\|g\|_{L^{q,s_2}(\R^3)}.\]
             \end{enumerate}

\end{lem}
\begin{lem}[\cite{Tan04}]\label{lemma-young}
\begin{enumerate}
  \item[\rm(i)]\; Let $\varphi\in \mathcal{S}(\mathbb{R}^3)$, then there holds
\begin{equation*}
\big\|\varphi\ast f\|_{L^\infty(\R^3)}\leq C\|f\|_{1,\,1},\quad \forall\, f\in L^1_{\rm ul}(\R^3),
\end{equation*}
where $C$ is a positive constant independent of $f$.
  \item [\rm (ii)]\;If $m\geq1$, then
\begin{equation*}
\|f\|_{q,\,m\lambda}\leq (C m^{3})^{\frac1q}\|f\|_{q,\,\lambda},\quad\forall\,f\in L^q_{\rm ul}(\R^3)\quad \text{and}\quad \forall\,\lambda>0.
\end{equation*}
\end{enumerate}
\end{lem}
Lastly, we show some properties of solutions to the stationary Euler system, which  are the key point of the blowup argument.
 \begin{lem}\label{L3.1}
Let $\Omega$ be a connected domain in $\R^{3}$ with  Lipschitz boundary, and the functions $v\in H(\Omega)$ and $p\in D^{1,\frac{3}{2}}(\Omega)$
satisfy the stationary Euler system
\begin{equation}\label{E3.1}
\left \{
\begin{array}{ll}
v\cdot\nabla \,v+\nabla p=0 \qquad\ \ \ &\text{in}\ \ \Omega,\\
{\rm div}\ v=0 \ \ \ &\text{in}\ \ \Omega,\\
v=0 \ \ \ &\text{on}\ \ \partial\Omega.
\end{array}
\right.
\end{equation}
Then
$$
\exists\  c\in \R\ \mbox{such that}\ p(x)\equiv c\ \ \mbox{for}\ \ \mathfrak{H}^{2}\mbox{-almost all }\ \ x\in\partial\Omega,
$$
where $\mathfrak{H}^{m}$ is denoted by the $m$-dimensional Hausdorff measure.
\end{lem}

\begin{proof}
The proof just follows the ideas developed in  \cite{Amick,ABLS}.  We present a proof in some detail of this lemma for the reader's convenience.

Let $z_{0}\in\partial\Omega$ and choose  a new orthogonal coordinate system $\tilde{x}=(\tilde{x}_{1},\tilde{x}_{2}, \tilde{x}_{3})$ centered at $z_{0}$
with the $x_{3}$-axis pointing along the inner normal to $\partial\Omega$ at $z_{0}$. Then $(v,p)$ satisfies
$$
v\cdot\widetilde{\nabla}v+\widetilde{\nabla}p=0\ \ \ \text{in }\ \Omega',
$$
where
$\widetilde{\nabla}\triangleq \big(\frac{\partial}{\partial\tilde{x}_{1}}, \frac{\partial}{\partial\tilde{x}_{2}}, \frac{\partial}{\partial\tilde{x}_{3}}\big)$
 and $\Omega'$ is the domain of $\Omega$ in the new coordinate system $\tilde{x}$.
For sufficient small $\epsilon>0$, the boundary component $\partial\Omega'$ is given locally by
$$
\tilde{x}_{3}=g(\tilde{x}_{1}, \tilde{x}_{2}),\ \  (\tilde{x}_{1},\tilde{x}_{2})\in \B_{\epsilon}(0)\subset \R^{2}, \ \ \mbox{with}\ g\in C^{0,1}\big(\B_{\epsilon}(0)\big).
$$
Let $\delta>0$ be sufficiently small such that
$$
\mathcal{A}=A(\epsilon,\delta)=\big\{(\tilde{x}_{1},\tilde{x}_{2}, \tilde{x}_{3}): \,\, (\tilde{x}_{1},\tilde{x}_{2})\in \B_{\epsilon}(0),\,\, \tilde{x}_{3}
\in (g(\tilde{x}_{1},\tilde{x}_{2}), g\big(\tilde{x}_{1},\tilde{x}_{2})+\delta\big)\big\}\subset \Omega'.
$$
One easily estimates
$$
\int_{\mathcal{A}}\frac{|v\cdot\widetilde{\nabla} v|}{|\tilde{x}_{3}-g(\tilde{x}_{1},\tilde{x}_{2})|}\,\mathrm{d}\tilde{x}\leq
 \big\|\widetilde{\nabla}u\big\|_{L^{2}(\mathcal{A})}\Big\|\frac{v}{\tilde{x}_{3}-g(\tilde{x}_{1},\tilde{x}_{2})}\Big\|_{L^{2}(\mathcal{A})}
 $$
and
\begin{align*}
&\int_{\mathcal{A}}\frac{v^{2}}{|\tilde{x}_{3}-g(\tilde{x}_{1},\tilde{x}_{2})|^{2}}\,{\rm d}\tilde{x}
=\int_{\B_{\epsilon}(0)}\,{\rm d}\tilde{x}'\int_{g(\tilde{x}_{1},\tilde{x}_{2})}^{g(\tilde{x}_{1},\tilde{x}_{2})+\delta}
\frac{v^{2}}{|\tilde{x}_{3}-g(\tilde{x}_{1},\tilde{x}_{2})|^{2}}\,{\rm d}\tilde{x}_{3}\\
\leq& \int_{\B_{\epsilon}(0)}\,\mathrm{d}\tilde{x}'\int_{g(\tilde{x}_{1},\tilde{x}_{2})}^{g(\tilde{x}_{1},\tilde{x}_{2})+\delta}
\frac{1}{|\tilde{x}_{3}-g(\tilde{x}_{1},\tilde{x}_{2})|^{2}}
\Big(\int_{g(\tilde{x}_{1},\tilde{x}_{2})}^{\tilde{x}_{3}}|\partial_{\eta}v(\tilde{x}',\eta)|\,{\rm d}\eta\Big)^{2}\,{\rm d}\tilde{x}_{3}
\leq 4\int_{\mathcal{A}}|\partial_{\eta}v(\tilde{x}',\eta)|^{2}\,{\rm d}\tilde{x},
\end{align*}
where  we have used the fact $u(\tilde{x}_{1},\tilde{x}_{2},\tilde{x}_{3})=0$ at $\tilde{x}_{3}=g(\tilde{x}_{1},\tilde{x}_{2})$ and the following  Hardy inequality that
$$
\int_{a}^{b}\Big(\frac{w(x)}{x-a}\Big)^{2}\,{\rm d}x\leq 4\int_{a}^{b}\big(w'(x)\big)^{2}\,{\rm d}x
$$
for all functions $w(x)\in C^{1}\big([a,b]\big )$ which vanish at $x=a$.
\medskip

For any scalar function $\phi(\tilde{x}_1,\tilde{x}_2)\in C_0^{\infty}(\B_{\epsilon}(0))$ and  $g(\tilde{x}_1,\tilde{x}_2)\in W^{1,\infty}(\B_{\epsilon}(0))$, we have that for $i=1,2$
\begin{align*}
\int_{\mathcal{A}}\partial_{\tilde{x_i}}  \phi \,p(\tilde{x}_1,\tilde{x}_2, \tilde{x}_3)\,{\rm d}\tilde{x}_1{\rm d}\tilde{x}_2{\rm d}x'_3
=&\int_{\B_{\epsilon}(0)}\int_0^\delta\partial_{\tilde{x_i}}\phi\, p\big(\tilde{x}_1,\tilde{x}_2, g(\tilde{x}_1,\tilde{x}_2)+x'_3\big)\,{\rm d}\tilde{x}_1{\rm d}\tilde{x}_2{\rm d}x'_3\\
=&\int_{\B_{\epsilon}(0)}\int_0^\delta\phi(\tilde{x}_1,\tilde{x}_2) \Big({\partial p\over \partial \tilde{x}_i}+{\partial p\over \partial \tilde{x}_3}{\partial g\over \partial \tilde{x}_i}\Big){\rm d}\tilde{x}_1{\rm d}\tilde{x}_2{\rm d}x'_3\\
=&\int_{\mathcal{A}}\phi(\tilde{x}_1,\tilde{x}_2)\Big ({\partial p\over \partial \tilde{x}_i}+{\partial p\over \partial \tilde{x}_3}{\partial g\over \partial \tilde{x}_i}\Big)\,{\rm d}\tilde{x}.
\end{align*}
By integrating by parts and the Lebesgue theorem, we have that for $i=1,2$
\begin{equation*}
\begin{split}
 \bigg|\int_{\B_{\epsilon}(0)}\partial_{\tilde{x_i}}  \phi \,p\big(\tilde{x}_1,\tilde{x}_2, g(\tilde{x}_1, \tilde{x}_2)\big)\,{\rm d}\tilde{x}_1{\rm d}\tilde{x}_2\bigg|
 =&\bigg|\lim_{\delta\to 0}{1\over \delta}\int_{\mathcal{A}}\partial_{\tilde{x_i}}  \phi \,p\big(\tilde{x}_1,\tilde{x}_2, \tilde{x_3}\big)\,{\rm d}\tilde{x}\bigg|\\
 =&\bigg|\lim_{\delta\to 0}{1\over\delta}\int_{\mathcal{A}}\phi(\tilde{x}_1,\tilde{x}_2)\Big ({\partial p\over \partial \tilde{x}_i}+{\partial p\over \partial \tilde{x}_3}{\partial g\over \partial \tilde{x}_i}\Big)\,\mathrm{d}\tilde{x}\bigg|\\
 \leq&\|g\|_{W^{1,\infty}}\lim_{\delta\to 0}{1\over\delta}\int_{\mathcal{A}}|\phi(\tilde{x}_1,\tilde{x}_2)|\,|\widetilde{\nabla}p|\,{\rm d}\tilde{x}\\
 \leq& C\lim_{\delta\to 0}\int_{\mathcal{A}}{\big|u\cdot\widetilde{\nabla}u\big|\over |\tilde{x}_3-g(\tilde{x}_1,\tilde{x}_2)|}\,{\rm d}\tilde{x}=0.
\end{split}
\end{equation*}
This estimate implies that for $i=1,2$
\begin{equation*}
\begin{split}
\int_{\B_{\epsilon}(0)}\partial_{\tilde{x_i}}  \phi \,p\big(\tilde{x}_1,\tilde{x}_2, g(\tilde{x}_1, \tilde{x}_2)\big)\,{\rm d}\tilde{x}_1{\rm d}\tilde{x}_2=-\int_{\B_{\epsilon}(0)}\phi\partial_{\tilde{x}_i}\big(p(\tilde{x_1},\tilde{x}_2)\big)\,{\rm d}\tilde{x}_1{\rm d}\tilde{x}_2=0.
\end{split}
\end{equation*}
Thus, arbitrariness of $\phi$, enables us to conclude that $p$ is a constant on $\partial\Omega$ almost everywhere.
\end{proof}

 %%%%%%%%%%%%%%%%%%%%%%%%%%%%%%%%%%%%%%%%%%%%%%%%%%%%%%%%%%%%%%%%%%%%%%%%%%%%%%%%%%%%%%%%%%%%%%%%%%%%%%%%%%%%%%%%%%%%%
 \subsection{Solution of the linear elliptic equation and properties of solution for the linear fractional diffusion equations}
%%%%%%%%%%%%%%%%%%%%%%%%%%%%%%%%%%%%%%%%%%%%%%%%%%%%%%%%%%%%%%%%%%%%%%%%%%%%%%%%%%%%%%%%%%%%%%%%%%%%%%%%%%%%%%%%%%%%%%%%%%%%%%%%%
In this subsection, we first focus on the following linear equation
\begin{equation}\label{EL}
-\Delta U-\frac12\big(x\cdot\nabla U+U\big)=f(x)  \quad \text{in}\quad\R^3.
\end{equation}
In order to find the solution of \eqref{EL}, we define inspired by the homogeneous principle that
\begin{equation}\label{defU}
\begin{split}
U(x)\triangleq&\int_0^1\frac{1}{\big(4\pi s\big)^{\frac32 }}\int_{\R^3}e^{-\frac{|y|^2}{4(1-s)}}\frac{1}{(1-s)^{\frac32}}f\left(\frac{x-y}{\sqrt{1-s}}\right)\,\mathrm{d}y\mathrm{d}s\\
=&\int_0^1\int_{\R^3}\Phi(y,s)\frac{1}{(1-s)^{\frac32}}f\left(\frac{x-y}{\sqrt{1-s}}\right)\,\mathrm{d}y\mathrm{d}s.
\end{split}
\end{equation}
Here and what in follows, we denote
\[\Phi(x,t)=\frac{1}{(4\pi t)^{\frac32}}e^{-\frac{|x|^2}{4t}}\quad\text{for all}\quad(x,t)\in\R^3\times(0,+\infty).\]
Setting $\widetilde{f}(x,t)=\frac{1}{t^{\frac32}}f\big(\frac{x}{\sqrt{t}}\big),$
one has
\[U(x)=\int_0^1\int_{\R^3}\Phi(y,s)\widetilde{f}(x-y,1-s)\,\mathrm{d}y\mathrm{d}s.\]
Consequently this convolution should be solution of equation \eqref{EL}.
\begin{prop}\label{prop-E}
	Let $U$ be defined in \eqref{defU}, and $f\in C^2(\R^3)$ satisfying
	\begin{equation}\label{con}
	\sup_{x\in\R^3}|x|^3|f|(x)<+\infty.
	\end{equation}
	Then
	$U\in C^2(\R^3)$ solves the linear elliptic equation
	$-\Delta U-\frac12\big(x\cdot\nabla U+U\big)=f(x) $ in  $\R^3.$
\end{prop}
\begin{proof}
	According to the representation \eqref{defU}, it is easy to show that $U\in C^2(\R^3)$. So we need to show that $u(x,t)$ is the solution of equation \eqref{EL}.
We compute
\begin{equation}\label{eq-II-1}
-\Delta U=\int_0^1\int_{\R^3}\big(-\Delta_y\big)\Phi(y,s)\widetilde{f}(x-y,1-s)\,\mathrm{d}y\mathrm{d}s
\end{equation}
and
\begin{equation}\label{eq-II-2}
\begin{split}
-\frac12\left(x\cdot\nabla U+U\right)
=&\int_0^1\int_{\R^3}\Phi(y,s)\left(-\frac12x\cdot\nabla_x-\frac12\right) \widetilde{f}(x-y,1-s)\,\mathrm{d}y\mathrm{d}s\\
=&-\frac12\int_0^1\int_{\R^3}\Phi(y,s)\big((x-y)\cdot\nabla_x\big) \widetilde{f}(x-y,1-s)\,\mathrm{d}y\mathrm{d}s\\
&-\frac12\int_0^1\int_{\R^3}\Phi(y,s)\big(y\cdot\nabla_x\big) \widetilde{f}(x-y,1-s)\,\mathrm{d}y\mathrm{d}s\\
&-\frac12\int_0^1\int_{\R^3}\Phi(y,s)\widetilde{f}(x-y,1-s)\,\mathrm{d}y\mathrm{d}s.
\end{split}
\end{equation}
Since
\begin{align*}
\frac12(x-y)\cdot\nabla_x\left(\widetilde{f}(x-y,1-s)\right)
=&\frac{(x-y)}{2(1-s)^2}\cdot\nabla f\left(\frac{x-y}{\sqrt{1-s}}\right)\\
=&(1-s)\partial_s\left(\widetilde{f}(x-y,1-s)\right)-\frac32\widetilde{f}(x-y,1-s)\\
=&\partial_s\left((1-s)\widetilde{f}(x-y,1-s)\right)-\frac12\widetilde{f}(x-y,1-s),
\end{align*}
we readily have
\begin{equation}\label{eq-II-3}
-\frac12(x-y)\cdot\nabla_x\left(\widetilde{f}(x-y,1-s)\right)-\frac12\widetilde{f}(x-y,1-s)=-\partial_s\left((1-s)\widetilde{f}(x-y,1-s)\right).
\end{equation}
On the other hand, we see that
\begin{equation}\label{eq-II-4}
\begin{split}
&-\frac12\int_0^1\int_{\R^3}\Phi(y,s)\big(y\cdot\nabla_x\big) \widetilde{f}(x-y,1-s)\,\mathrm{d}y\mathrm{d}s\\
=&\frac12\int_0^1\int_{\R^3}\Phi(y,s)\big(y\cdot\nabla_y\big) \widetilde{f}(x-y,1-s)\,\mathrm{d}y\mathrm{d}s\\
=&-\frac32\int_0^1\int_{\R^3}\Phi(y,s) \widetilde{f}(x-y,1-s)\,\mathrm{d}y\mathrm{d}s
 -\frac12\int_0^1\int_{\R^3}\big(y\cdot\nabla_y\big)\Phi(y,s) \widetilde{f}(x-y,1-s)\,\mathrm{d}y\mathrm{d}s.
\end{split}
\end{equation}
A simple calculation yields
\begin{equation}\label{eq-II-5}
\frac12\big(y\cdot\nabla_y\big)\Phi(y,s)=-\frac{1}{(4\pi s)^{\frac32}}\frac{|y|^2}{4s}e^{-\frac{|y|^2}{4s}}
=-s\Delta_y\Phi(y,s)-\frac{3}{2}\Phi(y,s).
\end{equation}
Plugging \eqref{eq-II-5} into \eqref{eq-II-4} leads to
\begin{equation}\label{eq-II-6}
\begin{split}
 -\frac12\int_0^1\int_{\R^3}\Phi(y,s)\big(y\cdot\nabla_x\big) \widetilde{f}(x-y,1-s)\,\mathrm{d}y\mathrm{d}s
= \int_0^1\int_{\R^3}\big(s\Delta_y\big)\Phi(y,s) \widetilde{f}(x-y,1-s)\,\mathrm{d}y\mathrm{d}s.
\end{split}
\end{equation}
Inserting \eqref{eq-II-3} and \eqref{eq-II-6} into \eqref{eq-II-2}, we readily have
\begin{equation}\label{eq-II-7}
\begin{split}
-\frac12\big(x\cdot\nabla U+U\big)
=&-\int_0^1\int_{\R^3}\Phi(y,s) \partial_s\left((1-s)\widetilde{f}(x-y,1-s)\right)\,\mathrm{d}y\mathrm{d}s\\
&+\int_0^1\int_{\R^3}\big(s\Delta_y\big)\Phi(y,s) \widetilde{f}(x-y,1-s)\,\mathrm{d}y\mathrm{d}s.
\end{split}
\end{equation}
With \eqref{eq-II-1} and \eqref{eq-II-7} in hand,  we find that
\begin{equation}\label{eq-II-8}
\begin{split}
-\Delta U-\frac12\big(x\cdot\nabla U+U\big)
=&-\int_0^1\int_{\R^3}\Phi(y,s) \partial_s\left((1-s)\widetilde{f}(x-y,1-s)\right)\,\mathrm{d}y\mathrm{d}s\\
&-\int_0^1\int_{\R^3}\big(\Delta_y\big)\Phi(y,s)\left((1-s) \widetilde{f}(x-y,1-s)\right)\,\mathrm{d}y\mathrm{d}s.
\end{split}
\end{equation}
Integrating by parts with respect to time $t$, we have
\begin{align*}
&-\int_0^1\int_{\R^3}\Phi(y,s) \partial_s\left((1-s)\widetilde{f}(x-y,1-s)\right)\,\mathrm{d}y\mathrm{d}s\\
=&\int_0^1\int_{\R^3}\big(\partial_s\big)\Phi(y,s) \left((1-s)\widetilde{f}(x-y,1-s)\right)\,\mathrm{d}y\mathrm{d}s\\
&-\lim_{s\to1-}\int_{\R^3}\Phi(y,s)\left((1-s)\widetilde{f}(x-y,1-s)\right)\,\mathrm{d}y\\
&+\lim_{s\to0+}\int_{\R^3}\Phi(y,s)\left((1-s)\widetilde{f}(x-y,1-s)\right)\,\mathrm{d}y.
\end{align*}
Plugging this estimate in \eqref{eq-II-8} gives
\begin{equation*}\label{eq-II-9}
\begin{split}
-\Delta U-\frac12\big(x\cdot\nabla U+U\big)
=&\int_0^1\int_{\R^3}\big(\partial_s-\Delta_y\big)\Phi(y,s)\left((1-s) \widetilde{f}(x-y,1-s)\right)\,\mathrm{d}y\mathrm{d}s\\
&-\lim_{s\to1-}\int_{\R^3}\Phi(y,s)\left((1-s)\widetilde{f}(x-y,1-s)\right)\,\mathrm{d}y\\
&+\lim_{s\to0+}\int_{\R^3}\Phi(y,s)\left((1-s)\widetilde{f}(x-y,1-s)\right)\,\mathrm{d}y\\
=&-\lim_{s\to1-}\int_{\R^3}\Phi(y,s)\left((1-s)\widetilde{f}(x-y,1-s)\right)\,\mathrm{d}y\\
&+\lim_{s\to0+}\int_{\R^3}\Phi(y,s)\left((1-s)\widetilde{f}(x-y,1-s)\right)\,\mathrm{d}y.
\end{split}
\end{equation*}
Thanks to the condition \eqref{con}, we know that
\[\lim_{s\to1-}\left((1-s)\widetilde{f}(x-y,1-s)\right)=0.\]
This enables us to conclude
\[\lim_{s\to1-}\int_{\R^3}\Phi(y,s)\left((1-s)\widetilde{f}(x-y,1-s)\right)\,\mathrm{d}y=0.\]
On the other hand, we observe that
\begin{align*}
\lim_{s\to0+}\int_{\R^3}\Phi(y,s)\left((1-s)\widetilde{f}(x-y,1-s)\right)\,\mathrm{d}y
=\int_{\R^3}\delta(y)\widetilde{f}(x-y,1)\,\mathrm{d}y=f(x).
\end{align*}
Therefore we finally obtain
\[-\Delta U-\frac12\big(x\cdot\nabla U+U\big)=f(x) \qquad \text{for all}\,\,\, x\in\R^3.\]
So we finish the proof of the proposition.
\end{proof}
Next, we will investigate some properties of the linear fractional diffusion equation.
Let $0<\alpha\le1$, and $u$ be the solution  to the fractional diffusion equation
\begin{align*}
\partial_{t}u+(-\Delta)^{\alpha}u&=f(x,t)\quad \mbox{in}\ \ \R^{3}\times (0,+\infty)\\
u(x,0)&=\varphi(x)\quad\,\,\,\;\mbox{in}\ \ \R^{3}
\end{align*}
where $f\in C_{0}^{\infty}\big(\R^{3}\times [0,+\infty)\big)$ and $\varphi(x)\in C_{0}^{\infty}(\R^{3})$.
\medskip

By Duhamel formula, one writes
\begin{equation}\label{E2.1}
u(x,t)=G^\alpha_{t}\ast \varphi+\int_{0}^{t}G^\alpha_{t-s}(x-y)f(s,y)\,{\rm d}s,\quad    (t,x)\in\R^+\times\R^{3},
\end{equation}
where
$$
G^\alpha_{t}(x)=\mathcal{F}^{-1}\left(e^{-t|\xi|^{2\alpha}}\right)\ \qquad \text{for all}\,\,\, t>0
$$
where $\mathcal{F}^{-1}$ denote the inverse Fourier transform. The function $G^\alpha_{t}$  is the probability density function of a spherically symmetric
$2\alpha$-stable process whose generator corresponds to  the fractional Laplacian $(-\Delta)^{\alpha}$:
$$
\int_{\R^{3}} G^\alpha_{t}(x)\,\mathrm{d}x=1\qquad \text{for all}\,\,\, t>0.
$$

\begin{lem}[\cite{BG,MYZ}] \label{L2.1}
\begin{enumerate}
  \item [(i)] For $(x,t)\in\R^n\times(0,+\infty)$, we have
$$
G^\alpha_{t}(x)=t^{-\frac{n}{2\alpha}}G^\alpha_{1}\Big(\frac{x}{t^{\frac{1}{2\alpha}}}\Big),
$$
where $G^\alpha_{1}(x)$ is a smooth strictly positive radial function on $\R^{n}$, and
$$
G^\alpha_{1}(x)=\left((2\pi)^{\frac{n}{2}}|x|^{\frac{n}{2}-1}\right)^{-1}\int_{0}^{\infty}e^{-t^{2\alpha}}t^{\frac{n}{2}}J_{\frac{n-2}{2}}\big(|x|t\big)\,\mathrm{d}t,
$$
where $J_{\mu}$ denotes the Bessel function of first kind of order $\mu$.
\vskip0.2cm

  \item[(ii)] $\displaystyle
\lim_{|x|\to+\infty}|x|^{n+2\alpha}G^\alpha_{1}(x)=C_{\alpha,n}\sin\alpha\pi.
$
\medskip

  \item [(iii)]$
\big|\nabla^{k}G^\alpha_{t}(x)\big|\leq t\big(t^{\frac{1}{2\alpha}}+|x|\big)^{-n-2\alpha-k}.
$
\medskip

\item [(iv)] $
\big|(-\Delta)^{\alpha}G^\alpha_{t}(x)\big|\leq \big(t^{\frac{1}{2\alpha}}+|x|\big)^{-n-2\alpha}.
$
\end{enumerate}
\end{lem}
\begin{lem}\label{L2.3}
Let $\varphi\in L^{(r,\infty)}(\R^{n})$ with $1<r<+\infty$. Then we have
\begin{enumerate}
	\item [a)]for each $p\geq r,$  $
	\|G^\alpha_{t}\ast\varphi\|_{L^{p}(\R^n)}\leq C(n,p,r)t^{-\frac{n}{2\alpha}(\frac{1}{r}-\frac{1}{p})}\|\varphi\|_{L^{(r,\infty)}(\R^n)};
	$
	\smallskip
	
	\item[b)] $u(x,t)=G^\alpha_t\ast \varphi(x)\in \mathrm{BC}_{\rm w}\big([0,+\infty),\,L^{(r,\infty)}(\R^n)\big).
	$
\end{enumerate}
\end{lem}
\begin{proof} a) Since
{\color{blue}\begin{equation*}
\|G^\alpha_{t}\|_{L^{(p_1,\infty)}(\R^n)}\leq C\|G^\alpha_{t}\|_{L^{p_1}(\R^n)} =Ct^{-(1-\frac{1}{p_{1}})\frac{n}{2\alpha}}\|G^\alpha_{1}\|_{L^{p_{1}}(\R^n)},
\end{equation*}}
we obtain by the generalized Young inequality in Lemma \ref{GHY} that for all $p\in[r,+\infty),$
$$
\big\|G^\alpha _{t}\ast\varphi\big\|_{L^{(p,\infty)}(\R^n)}\leq \|G_{t}\|_{L^{(p_1,\infty)}(\R^n)}\,\|\varphi\|_{L^{(r,\infty)}(\R^n)}
\leq Ct^{-(1-\frac{1}{p_{1}})\frac{n}{2\alpha}}\|\varphi\|_{L^{(r,\infty)}(\R^n)},
$$
where $1+\frac{1}{p}=\frac{1}{p_{1}}+\frac{1}{r}$. This inequality together with the interpolation theorem yields the first desired result.

b)\
Now let $v\in L^{(r',1)}(\R^{n})$ which is  the dual space of $L^{(r,\infty)}(\R^{n})$.  We observe that for all $\varphi\in C_0^\infty(\R^n)$,
\begin{equation}\label{CC-1}
\big|\langle G^\alpha_{t}\ast\varphi-\varphi,v\rangle\big|=\big|\langle\varphi,G^\alpha_{t}\ast v-v\rangle\big|\leq \|\varphi\|_{L^{(r,\infty)}(\R^n)}\|G^\alpha_{t}\ast v-v\|_{L^{(r',1)}(\R^n)}.
\end{equation}
Since $C_{0}^{\infty}(\R^{n})$ is dense in $L^{(r',1)}(\R^{n})$, we have that $\varepsilon>0, $ there  exists a function  $\widetilde\varphi\in C_{0}^{\infty}(\R^{n})$ such that
\begin{equation}\label{CC}
\big\|v-\widetilde\varphi\big\|_{L^{(r',1)}(\R^n)}<\varepsilon.
\end{equation}
On the other hand,  we have by the fact that $\widetilde\varphi\in C_{0}^{\infty}(\R^{n})$  that  $G^\alpha_{t}\ast \widetilde\varphi\in L^\infty((0,+\infty), H^s_1(\R^n))$
for all $s\geq0.$ Since $G^\alpha_{t}\ast \widetilde\varphi$ solves
\[\partial_tu=-(-\Delta )^\alpha u,\]
we immediately get that $\|\partial_tG^\alpha_{t}\ast \widetilde\varphi\|_{H^s_1(\R^n)}\in L^\infty([0,+\infty))$ for all $s\geq0.$
This implies
that $u\in C((0,+\infty), H^s_1(\R^n))$ for all $s>0, $  and then we have $u\in C((0,+\infty), L^{(r',1)}(\R^n))$. Combining this fact with \eqref{CC} yields
$$
\big\|G^\alpha_{t}\ast v-v\big\|_{L^{(r',1)}(\R^n)}\to0\ \ \ \mbox{as}\ \ t\to0+
$$
It follows from \eqref{CC-1} that $\langle G_{t}\ast\varphi-\varphi,v\rangle \to0$ as $t\to0+$,
from which we obtain $ u(x,t)$ is weak $*$ continuous at 0 in the sense of $L^{(r,\infty)}(\R^n)$.  Similarly, we can show  that $u(x,t)$ is weak $*$ continuous for all $t>0$  in the sense of $L^{(r,\infty)}(\R^n)$.
\end{proof}

\begin{prop}\label{prop-U1}
Let 	$\varphi(x)=\frac{1}{|x|^{2\alpha-1}}$ with $\alpha \in (1/2,1]$, and $u_\alpha(x,t)=G_t^\alpha\ast \varphi(x) $. Then we have
\begin{enumerate}
	\item[\rm(i)]\; $u_\alpha \in \mathrm{BC}_{\rm w}\big([0,+\infty),\,L^{(\frac{3}{2\alpha-1},\infty)}(\R^3)\big)$.
	\item [\rm (ii)]\;  for all $s\in(\frac{3}{2\alpha-1},+\infty)$, $ \|u_\alpha(\cdot,1)\|_{L^p(\R^3)}<+\infty$, and for all $s>0$ and $p>\frac{3}{2\alpha}$,  $$\|\nabla u_\alpha(\cdot,1)\|_{H^s_p(\R^3)}<+\infty.$$
		\item [\rm (iii)]\;$\displaystyle \sup_{x\in\R^3}\langle x\rangle^{2\alpha-1+|\beta|}|D^\beta u_\alpha(x,1)|<+\infty$ for every  $\beta$.
\end{enumerate}
\end{prop}
\begin{proof}
	Since $\varphi(x)=\frac{1}{|x|^{2\alpha-1}}$ with $\alpha \in (1/2,1]$, it is easy to check that
$\varphi\in L^{(\frac{3}{2\alpha-1},\infty)}(\R^3)$.
 Thus, we can get the first two results by  Lemma \ref{L2.1} and Lemma \ref{L2.3}.
	
	For $\alpha=1,$ we see that
	\[u_1(x,1)=\frac{1}{(4\pi)^{\frac32}}\int_{\R^3}e^{-\frac{|x-y|^2}{4}}\varphi(y)\,\mathrm{d}y.\]
	We calculate
	\begin{equation*}
	\big||x|u_1(x,1)\big|\leq \frac{1}{(4\pi)^{\frac32}}\int_{\R^3}|x-y|e^{-\frac{|x-y|^2}{4}}|\varphi|(y)\,\mathrm{d}y+\frac{1}{(4\pi)^{\frac32}}\int_{\R^3}e^{-\frac{|x-y|^2}{4}}\big||y|\varphi\big|(y)\,\mathrm{d}y.
	\end{equation*}
	On one hand,
	\begin{equation*}
	\frac{1}{(4\pi)^{\frac32}}\int_{\R^3}e^{-\frac{|x-y|^2}{4}}\big||y|\varphi\big|(y)\,\mathrm{d}y
	\leq\sup_{x\in\R^3}\big||x|\varphi\big|(x)\,\frac{1}{(4\pi)^{\frac32}}\int_{\R^3}e^{-\frac{|x-y|^2}{4}}\,\mathrm{d}y\leq\sup_{x\in\R^3}\big||x|\varphi\big|(x).
	\end{equation*}
	On the other hand, we obtain by the generalized Young inequality that
	\begin{equation*}
	\frac{1}{(4\pi)^{\frac32}}\int_{\R^3}|x-y|e^{-\frac{|x-y|^2}{4}}|\varphi|(y)\,\mathrm{d}y
	\leq \big\||\cdot|\Phi(\cdot)\big\|_{L^{(2,1)}(\R^3)}\|\varphi\|_{L^{(3,1)}(\R^3)}.
	\end{equation*}
	Combining both estimates yields the third result for the case $|\beta|=0$ and $\alpha=1$. Repeating the same process, we can show the third result for each $\beta $ and $\alpha\in(1/2,1].$
\end{proof}

%%%%%%%%%%%%%%%%%%%%%%%%%%%%%%%%%%%%%%%%%%%
\setcounter{equation}{0}
\setcounter{equation}{0}
\section{Existence and regularity of  solutions to the corresponding elliptic  system }

\subsection{Existence of  solutions in $H^{\alpha}(\R^{3})$}

In this subsection, we  establish the existence of the  solution $U(x)$ of \eqref{E3.2} by the Leray-Schauder principle.
In this subsection, we always assume that $\alpha\in(5/8,1]$.
 From \eqref{E1.1},  we know that  the profile $U(x)$ of $u(x,t)$ satisfies
\begin{equation}\label{E3.2-}
\left\{ \begin{aligned}
&(-\Delta)^{\alpha}U+U\cdot\nabla U+\nabla P-\frac{2\alpha-1}{2\alpha}U(x)-\frac{1}{2\alpha}x\cdot \nabla U=0, \\
&\mbox{div}\ U=0,
\end{aligned}
\right.\qquad \text{in}\quad\R^{3}.
\end{equation}
Letting $U_{0}=G_{1}\ast u_0$, there exists a  pressure $P_0(x)$  such that
$$
(-\Delta)^{\alpha}U_{0}+\nabla P_{0}-\frac{2\alpha-1}{2\alpha}U_{0}(x)-\frac{1}{2\alpha}x\cdot \nabla U_{0}=0.
$$
We decompose $U=U_{0}+V$, then  the difference part $V(x)$ satisfies for $x\in\R^{3}$
\begin{equation}\label{E3.2}
 \left\{ \begin{aligned}
(-\Delta)^{\alpha}V-\frac{2\alpha-1}{2\alpha}V(x)-\frac{1}{2\alpha}x\cdot \nabla V+\nabla P=-U_{0}\cdot\nabla U_{0},
-(U_{0}+V)\cdot\nabla V-V\cdot\nabla U_{0}, \\
 \mbox{div}\ V=0
\end{aligned}
\right.
\end{equation}
with a suitable scalar $P$. Thus the problem to solve $U$ is equivalent to solving  \eqref{E3.2}. For this purpose,
 we introduce the following hypervisicosity perturbation of \eqref{E3.2}:
\begin{equation}\label{E3.3}
 \left\{\begin{array}{ll}
-\epsilon\Delta V+(-\Delta)^{\alpha}V+\nabla P=\lambda\Big(\frac{2\alpha-1}{2\alpha}V(x)+\frac{1}{2\alpha}x\cdot \nabla V+F(V)\Big),\\
\mbox{div}\ V=0,
\end{array}
\right.
\end{equation}
where
$$F(V)=-U_{0}\cdot\nabla U_{0}-(U_{0}+V)\nabla V-V\cdot\nabla U_{0}, \;\;\;\lambda\in[0,1].$$

 To overcome the loss of compactness of $ H_{0,\sigma}^{1}(\R^{3})\cap H_{0,\sigma}^{\alpha}(\R^{3})$,
  we will approximate $\R^{3}$ by an increasing sequence of concentric balls,
  construct solutions of \eqref{E3.3} in these balls with zero boundary condition, and take a limit of the approximate solution sequence to obtain a desired solution in $\R^{3}$ to \eqref{E3.3} at $\lambda=1$.
   Letting $\epsilon\to0$, we finally obtain the existence of solution to problem \eqref{E3.2}, and then this solution
   is converted into a self-similar solution of \eqref{E1.1}.

Now we construct a weak solution $V_{R,\epsilon}$ of \eqref{E3.3} in the following  space
$$
 \mathcal{X}_{R}\triangleq \Big\{u:\,\,  u\in H_{0,\sigma}^{1}(\B_{R})\cap H_{0,\sigma}^{\alpha}(\B_{R})\quad\text{and}\quad u\equiv0\ \ \mbox{for}\ x\in \R^{3}\setminus\B_{R}\Big\}. $$
 This  means that  for all $\varphi\in \mathcal{X}_{R}$, to look for $V_{R,\epsilon}$ satisfying
\begin{equation}\label{E3.4}
\begin{split}
&\int_{\R^{3}}(\epsilon\nabla V_{R,\epsilon}\cdot \nabla \varphi+(-\Delta)^{\frac{\alpha}{2}} V_{R,\epsilon}\cdot (-\Delta)^{\frac{\alpha}{2}} \varphi) \,{\rm d}x\\
=&\lambda\int_{\B_{R}}\Big(\frac{2\alpha-1}{2\alpha}V_{R,\epsilon}+\frac{1}{2\alpha}x\cdot\nabla V_{R,\epsilon}+F(V_{R,\epsilon})\Big)\cdot \varphi \,{\rm d}x,
\end{split}
\end{equation}
for $\lambda\in[0,1]$. Here the second term of the left hand side is defined via Fourier transform
$$
\int_{\R^{3}}(-\Delta)^{\frac{\alpha}{2}} V_{R,\epsilon}\cdot (-\Delta)^{\frac{\alpha}{2}} \varphi \,{\rm d}x=\int_{\R^{3}}|\xi|^{2\alpha}\hat{\varphi}(\xi)\hat{V}_{R,\epsilon}(\xi)\,{\rm d}\xi.
$$
Let   $u, v\in \mathcal{X}_{R}$, we  introduce inner product as follows
$$
\langle u,v\rangle_{\mathcal{\mathcal{X}_{R}}}=\int_{\R^{3}}\epsilon\nabla u\cdot \nabla v\,{\rm d}x +\int_{\R^{3}}(-\Delta)^{\frac{\alpha}{2}} u\cdot (-\Delta)^{\frac{\alpha}{2}} v\,{\rm d}x.
$$
Then equation \eqref{E3.4} can be rewritten as
$$
\langle V,\varphi\rangle_{\mathcal{\mathcal{X}_{R}}}=\lambda\int_{\B_{R}}\Big(\frac{2\alpha-1}{2\alpha}V+\frac{1}{2\alpha}x\cdot\nabla V+F(V)\Big)\cdot \varphi \,{\rm d}x, \ \ \ \forall\ \varphi\in \mathcal{\mathcal{X}_{R}}.
$$
 By the Riesz representation theorem,  for any $f\in \mathcal{\mathcal{X}'_{R}}$ there exists a unique linear mapping
$\mathbb{T}(f)\in \mathcal{\mathcal{X}_{R}}$ such that
$$
\langle \mathbb{T}(f),\varphi\rangle_{\mathcal{X}_{R}}=\int_{\B_{R}}f\cdot\varphi \,{\rm d}x,\ \ \ \forall\, \varphi\in \mathcal{X}_{R},
$$
with
$$
\|\mathbb{T}(f)\|_{\mathcal{X}_{R}}\leq \|f\|_{\mathcal{X}'_{R}}.
$$
According to \eqref{E3.4}, we define the following operator
\begin{equation}\label{E3.5}
V_{R,\epsilon}=\lambda\mathbb{T}\Big(\frac{2\alpha-1}{2\alpha}V_{R,\epsilon}+\frac{1}{2\alpha}x\cdot\nabla V_{R,\epsilon}+F(V_{R,\epsilon})\Big)\triangleq \lambda S(V_{R,\epsilon}).
\end{equation}

To prove the existence of a solution $V_{R}$ of  integral equation \eqref{E3.5} at $\lambda=1$, we first have to prove that the set
$$
\Big\{x\in \mathcal{X}_{R}:\,\, x=\lambda Sx\ \ \ \mbox{for some}\ \lambda\in [0,1]\Big\}
$$
is bounded in $X$, and then prove the operator $S$ is continuous and compact.

\medskip
{\bf Step 1}: {\em a priori bound}
\begin{lem}[a priori estimate]\label{L3.2}
Let $V_{R,\epsilon}$ be the solution of \eqref{E3.4}, we have
$$
\int_{\R^{3}}\Big(\epsilon|\nabla V_{R,\epsilon}|^{2}+|(-\Delta)^{\frac{\alpha}{2}}V_{R,\epsilon}|^{2}+\frac{5-4\alpha}{4\alpha}|V_{R,\epsilon}|^{2}\Big)\,{\rm d}x\leq C(U_{0},R,\epsilon).
$$
\end{lem}
\begin{proof}We will give a proof of Lemma \ref{L3.2} by contradiction. Now let us suppose that there exists a sequence $\lambda_{k}\in [0,1]$ and functions
$V_{k}\triangleq V^{(k)}_{R,\epsilon}\in \mathcal{X}_{R}$ such that
\begin{equation}\label{E3.6}
\left \{
\begin{array}{ll}
-\epsilon\Delta V_{k}+(-\Delta)^{\alpha}V_{k}+\nabla P_{k}=\lambda_{k}\Big(&\frac{2\alpha-1}{2\alpha}V_{k}(x)+\frac{1}{2\alpha}x\cdot \nabla V_{k}-U_{0}\cdot\nabla U_{0}\\
&-(U_{0}+V_{k})\nabla V_{k}-V_{k}\cdot\nabla U_{0}\Big)\\
\mbox{div}\ V_{k}=0,
\end{array}
\right.
\end{equation}
and
$$
L^{2}_{k}\triangleq\int_{\R^{3}}(\epsilon|\nabla V_{k}|^{2}+|(-\Delta)^{\frac{\alpha}{2}}V_{k}|^{2})\,{\rm d}x\to+\infty,\ \ \ \ \lambda_{k}\to\lambda_{0}\in[0,1].
$$
Multiplying \eqref{E3.6} by $V_{k}$ and integrating by parts in $\B_{R}$, we obtain
\begin{equation}\label{E3.7}
L_{k}^{2}+\frac{\lambda_{k}(5-4\alpha)}{4\alpha}\int_{\B_{R}}|V_{k}|^{2}\,{\rm d}x=\lambda_{k}\int_{\B_{R}}\Big(-U_{0}\cdot\nabla U_{0}
-V_{k}\cdot\nabla U_{0}\Big)\cdot V_{k}\,{\rm d}x,
\end{equation}
where we have used the fact that
$$
\int_{B_R}\big(U_{0}+V_{k})\cdot\nabla V_{k}\big)\cdot V_{k}\,{\rm d}x =0.
$$
Now we consider the normalized sequence of functions
$$
\tilde{V}_{k}\triangleq\frac{1}{L_{k}}V_{k}\quad\text{and}\quad\tilde{P}_{k}\triangleq\frac{1}{\lambda_{k}L_{k}^{2}}P_{k},
$$
such that
$$
\int_{\R^{3}}\Big(\epsilon|\nabla \tilde{V}_{k}|^{2}+|(-\Delta)^{\frac{\alpha}{2}}\tilde{V}_{k}|^{2}\Big)\,{\rm d}x=1.
$$
Therefore,  we can extract a subsequence still denoted by $\tilde{V}_{k}$ such that
$$
\tilde{V}_{k}\rightharpoonup V \ \ \ \mbox{in}\ \ H_{0}^{1}(\B_{R})\cap H_{0}^{\alpha}(\B_{R}).$$
This means
$$
\tilde{V}_{k}\rightarrow V  \ \ \ \mbox{in}\ \ L^{3}(\B_{R}).
$$
Multiplying identity \eqref{E3.7} by $\frac{1}{L_{k}^{2}} $ and taking a limit as $k\to+\infty$, we have
\begin{equation}\label{E3.8}
1+\frac{\lambda_{0}(5-4\alpha)}{4\alpha}\int_{\B_{R}}|V|^{2}\,{\rm d}x=-\lambda_{0}\int_{\B_{R}}(V\cdot\nabla U_{0})V\,{\rm d}x=\lambda_{0}\int_{\B_{R}}(V\cdot\nabla V)U_{0}\,{\rm d}x,
\end{equation}
this relation yields $\lambda_{0}>0$.

Multiplying equation \eqref{E3.6} by $\frac{1}{\lambda_{k}L_{k}^{2}}$, we have
\begin{align*}
\tilde{V}_{k}\cdot\nabla \tilde{V}_{k}+\nabla \tilde{P}_{k}={1\over L_k}&\Big({{\varepsilon}\over \lambda_k}\Delta \tilde{V}_k+{1\over \lambda_k}(-\Delta)^{\alpha}\tilde{V}_k+{2\alpha-1\over 2\alpha}\tilde{V}_k+\frac{1}{2\alpha}x\cdot\nabla\tilde{V}_k\\
&\,\,\,\,-{1\over L_k}U_0\cdot\nabla U_0-U_0\cdot\nabla\tilde{V}_k-\tilde{V}_k\cdot\nabla U_0\Big).
\end{align*}
Multiplying the above equation by $\varphi$ and integrating the resulting equality over $\mathbb{B}_R$, we can show that
$$
\int_{\B_{R}}(V\cdot\nabla V)\cdot\varphi\, \,{\rm d}x=0\qquad \text{for all}\,\, \; \varphi\in C_{0,\sigma}^{\infty}(\mathbb{B}_R).
$$
Hence, we have by the Rham Theorem (for example, see \cite{Simon}) that there exists a pressure $P\in D^{1,\frac{3}{2}}(\B_{R})\cap L^{3}(\B_{R})$ such that
$(V,P)$ solves
\begin{equation*}
\left \{
\begin{array}{ll}
V\cdot\nabla V+\nabla P=0\ \ \ \quad&\mbox{in}\ \ \B_{R},\\
\mbox{div}\ V=0 \ \ \ &\mbox{in}\ \ \B_{R},\\
V=0 \ \ \ &\mbox{in}\ \ \R^{3}\setminus\B_{R}.
\end{array}
\right.
\end{equation*}
By Lemma \ref{E3.1}, there exists a constant $c\in\R$ such that $P(x)=c$ on $\partial\B_{R}$.
This fact helps us to get
\begin{align*}
1+\frac{\lambda_{0}(5-4\alpha)}{4\alpha}\int_{\B_{R}}|V|^{2}\,{\rm d}x&=-\lambda_{0}\int_{\B_{R}}U_{0}\cdot\nabla P \,{\rm d}x\\
&=-\lambda_{0}\int_{\B_{R}}\mbox{div}(P U_{0})\,{\rm d}x\\
&=c\lambda_0\int_{\partial\B_R} U_0\cdot \overrightarrow{n}\,{\rm d}s\\
&=-c\lambda_0\int_{\B_R} \nabla \cdot U_0 \,{\rm d}x=0.
\end{align*}
This is a contradiction, and so we have completed the  proof of Lemma \ref{L3.2}.
\end{proof}

\medskip
{\bf Step 2:} Continuity and compactness

\begin{lem}\label{L3.3}
 The operator
$$
S: \mathcal{X}_{R}\ni v\to \mathbb{T}\left(\frac{2\alpha-1}{2\alpha}v+\frac{1}{2\alpha}x\cdot\nabla v+F(v)\right)\in \mathcal{X}_{R}
$$
is continuous and compact.
\end{lem}
\begin{proof}\; First we prove $S$ is continuous.  Let $v_{1}, v_{2}\in \mathcal{X}_R$, then
\begin{align*}
&\|S(v_{1})-S(v_{2})\|_{\mathcal{X}_{R}}\\
=&
\Big\|\mathbb{T}\Big(\frac{2\alpha-1}{2\alpha}v_{1}+\frac{1}{2\alpha}x\cdot\nabla v_{1}+F(v_{1})\Big)-\mathbb{T}\Big(\frac{2\alpha-1}{2\alpha}v_{2}+\frac{1}{2\alpha}x\cdot\nabla v_{2}+F(v_{2})\Big)\Big\|_{\mathcal{X}_{R}}\\
\leq &C\Big(\|v_{1}-v_{2}\|_{\mathcal{X}'_R}+\frac{1}{2\alpha}\|x\cdot\nabla v_{1}-x\cdot\nabla v_{2}\|_{\mathcal{X}'_R}+\|F(v_{1})-F(v_{2})\|_{\mathcal{X}_{R}'}\Big).
\end{align*}
Since $L^{2}(\B_{R})\subset L^{\frac{6}{5}}(\B_{R})\subset \mathcal{X}'_R$, we have
$$
\|v_{1}-v_{2}\|_{\mathcal{X}'_{R}}\leq \|v_{1}-v_{2}\|_{L^{2}(\B_{R})},\,\,\ \ \|x\cdot\nabla v_{1}-x\cdot\nabla v_{2}\|_{\mathcal{X}'_{R}}\leq C\|\nabla(v_{1}-v_{2})\|_{L^{2}(\B_{R})},
$$
and
\begin{align*}
\|F(v_{1})-F(v_{2})\|_{\mathcal{X}'_{R}}\leq &\|U_{0}\|_{L^{3}(\B_{R})} \|\nabla(v_{1}-v_{2})\|_{L^{2}(\B_{R})}+\|\nabla U_{0}\|_{L^{2}(\B_{R})}\|v_{1}-v_{2}\|_{L^{3}(\B_{R})}\\&
+\|\nabla v_1\|_{L^2(\B_R)}\|v_1-v_2\|_{L^3(\B_R)}+\|v_2\|_{L^3(\B_R)}\|\nabla(v_1-v_2)\|_{L^2(\B_R)}.
\end{align*}
So
$$
\|S(v_{1})-S(v_{2})\|_{\mathcal{X}_{R}}\leq C\big (1+\|v_{1}\|_{\mathcal{X}_R}+\|v_{2}\|_{\mathcal{X}_R}\big)\|v_{1}-v_{2}\|_{\mathcal{X}_{R}},
$$
which implies that $S$ is continuous.

Now we prove $S$ is compact, it suffices to show that: for any bounded sequence $v_{k}$, there exists a subsequence $v_{kl}$ such that
\begin{align}\label{E3.9}
\|S(v_{k_{l}})-S(v)\|_{\mathcal{X}_{R}}\to0 \ \ \ \ \mbox{as}\ \ l\to+\infty, \ \ \mbox{for some}\ \  v\in \mathcal{X}_{R}.
\end{align}
Indeed, if $\|v_{k}\|_{\mathcal{X}_{R}}<C$, then by the Sobolev embedding theorem, there exists a subsequence, still denoted by $v_{k}$, such that
$$
v_{k}\to v\ \ \ \mbox{in}\quad \ L^{q}(\B_{R})\ \ \ \mbox{for}\ \ 1\leq q<6,
$$
from which we immediately have for any vector $\psi\in C^{\infty}_0(\R^3)$,
 \begin{align*}
\langle x\cdot\nabla (v_{k}-v),\psi\rangle=-\langle x\cdot\nabla\psi,v_{k}-v\rangle-\langle v_{k}-v,\psi\rangle\lesssim \|v_{k}-v\|_{L^2}\|\psi\|_{\mathcal{X}_R},
\end{align*}
and
\begin{align*}
\big\langle F(v_{k})-F(v),\psi\big\rangle
=&-\langle U_0\cdot\nabla(v_{k}-v),\psi\rangle-\langle v_{k}\cdot\nabla(v_{k}-v),\psi\rangle\\
&-\langle (v_{k}-v)\cdot\nabla v,\psi\rangle-\langle (v_{k}-v)\cdot\nabla U_0,\psi\rangle\\
=&+\langle U_0\cdot\nabla\psi,v_{k}-v\rangle+\langle v_{k}\cdot\nabla\psi,v_{k}
-v\rangle\\&-\langle (v_{k}-v)\cdot\nabla v,\psi\rangle-\langle (v_{k}-v)\cdot\nabla U_0,\psi\rangle\\
\leq&C \|v_{k}-v\|_{L^3}\|\psi\|_{\mathcal{X}_R}.
\end{align*}
The above inequalities imply \eqref{E3.9}.
\end{proof}

From Lemma \ref{L3.2}, Lemma \ref{L3.3}, we conclude that the operator $S$ satisfies all the requirements of the Leray-Schauder principle, and so we have:  \begin{prop}[Existence in $\B_{R}$]\label{P3.4}
 The system \eqref{E3.4} has a solution $V_{R,\epsilon}\in \mathcal{X}_{R}$.
\end{prop}

Now, we wish to extend statements of Proposition \ref{P3.4} from $\B_R$ to $\R^{3}$. We will establish the following uniform bound independent of $\epsilon, R$:

\begin{lem}\label{L3.5}
Let $V_{R,\epsilon}$ be the solution of \eqref{E3.4}, we have the a priori bound
$$
\int_{\R^{3}}\big(\epsilon|\nabla V_{R,\epsilon}|^{2}+|(-\Delta)^{\frac{\alpha}{2}}V_{R,\epsilon}|^{2}+\frac{5-4\alpha}{4\alpha}|V_{R,\epsilon}|^{2}\big)\,{\rm d}x\leq C(U_{0}).
$$
\end{lem}
\begin{proof} Since  the a priori bound is independent of $\lambda$ by Lemma \ref{L3.2}, we suppose $\lambda=1$ at the moment. Now we proceed, as previously, by a contradiction argument. In fact, suppose that its assertion is not true. Then there exist sequences $\B_{k}\triangleq\B_{R_{k}},\,\, \epsilon_{k}$ and
$V_{k}\triangleq V_{R_{k},\epsilon_{k}}\in \mathcal{X}_{R_{k}}$ such that
$$
L_{k}^{2}\triangleq\int_{\R^{3}}\Big(\epsilon_{k}|\nabla V_{k}|^{2}+|(-\Delta)^{\frac{\alpha}{2}}V_{k}|^{2}+\frac{5-4\alpha}{4\alpha}|V_{k}|^{2}\Big)\,{\rm d}x\to+\infty,
$$
where $\epsilon_{k}\to \epsilon_{0}\in [0,1]$.

Multiplying the equation \eqref{E3.4} by $V_{k}$ and integrating by parts in $\B_{k}$, we have
\begin{equation}\label{E3.10}
L_{k}^{2}=\int_{\B_{k}}\Big(-U_{0}\cdot\nabla U_{0}
-V_{k}\cdot\nabla U_{0}\Big)V_{k}\,{\rm d}x,
\end{equation}
where we have used the fact that
$$
\int_{\B_{k}}\Big(\big(U_{0}+V_{k}\big)\nabla V_{k}\Big)V_{k}\,{\rm d}x =0.
$$
Now let
$$
\tilde{V}_{k}\triangleq\frac{1}{L_{k}}V_{k}\quad\text{and}\quad \tilde{P}_{k}\triangleq\frac{1}{L_{k}^{2}}P_{k},
$$
then they satisfy
\begin{equation}\label{E3.11}
\begin{split}
\tilde{V}_{k}\cdot\nabla \tilde{V}_{k}+\nabla \tilde{P}_{k}
=\frac{1}{L_{k}}\bigg(&\epsilon_k\Delta \tilde{V}_{k}+(-\Delta)^{\alpha}\tilde{V}_k+{2\alpha-1\over 2\alpha}\tilde{V}_k+\frac{1}{2\alpha}x\cdot\nabla \tilde{V}_{k}\\
&\,\,-\frac{1}{L_{k}}U_{0}\cdot\nabla U_{0}
-U_{0}\cdot\nabla \tilde{V}_{k}-\tilde{V}_{k}\cdot\nabla U_{0}\bigg)
\end{split}
\end{equation}
and
$$
\int_{\B_{k}}\Big(\epsilon_{k}|\nabla \tilde{V}_{k}|^{2}+|(-\Delta)^{\frac{\alpha}{2}}\tilde{V}_{k}|^{2}+\frac{5-4\alpha}{4\alpha}|\tilde{V}_{k}|^{2}\Big)\,{\rm d}x=1.
$$
Thus we could extract a subsequence still denoted by $\tilde{V}_{k}$ such that
\begin{align*}
&\tilde{V}_{k}\rightharpoonup V \ \ \ \ \mbox{in}\ \  H_{0}^{\alpha}(\B_{k});\\
&\tilde{V}_{k}\to V \ \ \ \ \mbox{in}\ \  L^{q}(\Omega')\ \ \mbox{for any bounded}\ \Omega'\subset\R^{3}\ \mbox{and}\ \  1\leq q<\frac{6}{3-2\alpha}.
\end{align*}
Thanks to Proposition \ref{prop-U1}, one  has that for $\alpha>\frac58,$
\begin{equation}\label{E3.11+}
\|U_0\cdot\nabla U_0\|_{L^2(\R^3)}\leq C\big\|\langle\cdot\rangle^{-(4\alpha-1)}\big\|_{L^2(\R^3)}<+\infty.
\end{equation}
Therefore we have by the Cauchy-Schwarz inequality that for $\alpha>\frac58,$
\begin{equation}\label{E3.12-}
\int_{\R^{3}}U_{0}\cdot\nabla U_{0}\cdot\tilde{V}_{k}\,{\rm d}x\leq \bigg(\int_{\R^{3}}|U_{0}\cdot\nabla U_{0}|^{2}\,{\rm d}x\bigg)^{\frac{1}{2}}\bigg(\int_{\R^{3}}|\tilde{V}_{k}|^{2}\,{\rm d}x\bigg)^{\frac{1}{2}}<+\infty.
\end{equation}
Multiplying \eqref{E3.10} by $\frac{1}{L_{k}^{2}}$ and taking a limit as $k\to+\infty$, we have by \eqref{E3.12-} that
\begin{equation}\label{E3.12}
1=\int_{\R^{3}}(-V\cdot\nabla U_{0})V\,{\rm d}x=-\int_{\R^{3}}(V_{i} V_{j})\partial_{i}U_{0,j} \,{\rm d}x.
\end{equation}
 From \eqref{E3.11+},  we have $U_{0}\in W^{1,q}(\R^{3})$ with some large enough $q$ and $ {\rm div}\ U_{0}=0$. By the density argument, there exists a sequence
$\varphi_{n}\in  C_{0,\sigma}^{\infty}(\R^{3})$ such that
$$
\big\|\nabla (\varphi_{n}-U_{0})\big\|_{L^{q}(\R^{3})}\to0.
$$
From this limitation  and the fact $V\in L^{m}(\R^{3})$ for $1\leq m\leq\frac{6}{3-2\alpha}$ , we immediately have
$$
0=\int_{\R^{3}}(V\otimes V)\cdot\nabla\varphi_{n} \,{\rm d}x\to \int_{\R^{3}}(V\otimes V)\cdot\nabla U_{0} \,{\rm d}x=-1,
$$
which is a contradiction.
\end{proof}

\begin{prop}[Existence in $\R^{3}$]\label{L3.16}
The system \eqref{E3.3} has a solution $V_{\epsilon}\in H_{\sigma}^{1}(\R^{3})\cap H_{\sigma}^{\alpha}(\R^{3})$ for $\lambda=1$.
\end{prop}
\begin{proof}  Let $\B_{R}$ be the ball in $\R^{3}$ with radius, then by Lemma \ref{L3.5} there exists a solution
$V_{R,\epsilon}\in \mathcal{X}_{R}$ of the system \eqref{E3.4} in $\B_{R}$, which satisfies
\begin{equation*}
\int_{\R^{3}}\Big(\epsilon|\nabla V_{R,\epsilon}|^{2}+|(-\Delta)^{\frac{\alpha}{2}}V_{R,\epsilon}|^{2}+\frac{5-4\alpha}{4\alpha}|V_{R,\epsilon}|^{2}\Big)\,{\rm d}x\leq C(U_{0}).
\end{equation*}
Due to the above uniform regularity estimate, there exists a converging subsequence $\{V_{R_{j},\epsilon}\}_{j=1}^{\infty}$ (where $R_{j}\uparrow +\infty$)  such that
\begin{align}\label{weak1}
&V_{R_{j},\epsilon}\rightharpoonup V_{\epsilon}\ \ \ \mbox{in}\ \  H_{\sigma}^{1}(\R^{3})\cap H_{\sigma}^{\alpha}(\R^{3});
\end{align}
\begin{align*}%\label{weak2}
&V_{R_{j},\epsilon}\rightharpoonup V_{\epsilon}\ \ \ \mbox{in}\ \ L^{q}(\R^{3})\ \ \mbox{with}\ \ 2\leq q\leq6,
\end{align*}
and for any $0<R<+\infty$
\begin{align}\label{strong1}
V_{R_{j},\epsilon}\to V_{\epsilon}\ \ \ \  \mbox{in}\ \  L^{q}(\B_{R})\ \ \mbox{for}\ \ 1\leq q<6,
\end{align}
where we have used the fact that $V_{R_{j},\epsilon}\equiv0$ in $\R^{3}\setminus\B_{R_{j}}$.

We first  show that $V_{\epsilon}$ satisfies \eqref{E3.3} in the sense of
distribution. For $\forall\, \varphi\in C^\infty_{0,\sigma}(\R^3)$, seeding $R_j\to +\infty$ we easily  derive
\begin{equation}\label{1}
\begin{split}
&\int_{\R^3}\epsilon  \nabla V_{R_j,\epsilon} \nabla\varphi \,{\rm d}x+\int_{\R^3}(-\Delta)^{\frac{\alpha}{2}}V_{R_j,\epsilon}(-\Delta)^{\frac{\alpha}{2}}\varphi \,{\rm d}x
+\int_{\R^3}x\cdot \nabla V_{R_j, \epsilon}\varphi \,{\rm d}x\\\to& \int_{\R^3}\epsilon  \nabla V_\epsilon \nabla\varphi\,{\rm d}x+\int_{\R^3}(-\Delta)^{\frac{\alpha}{2}}V_{\epsilon}(-\Delta)^{\frac{\alpha}{2}}\varphi \,{\rm d}x+\int_{\R^3}x\cdot \nabla V_{\epsilon}\varphi \,{\rm d}x.
\end{split}
\end{equation}
Since $\varphi\in C_0^\infty(\mathbb{R}^3)$, we may assume that
$\mathrm{supp}\,\varphi\subset \mathbb{B}_{R}.$
The simple calculation yields
\begin{equation}\label{eq.conv1}
\begin{split}
&\int_{\mathbb{R}^3}V_{R_j,\varepsilon}\cdot\nabla V_{R_j,\varepsilon} \varphi(x)\,{\rm d}x-\int_{\mathbb{R}^3}V_{\varepsilon}\cdot\nabla V_{\varepsilon} \varphi(x)\,{\rm d}x\\
=&\int_{\mathbb{R}^3}\left(V_{R_j,\varepsilon}\otimes\big(V_{R_j,\varepsilon}-V_{\varepsilon}\big)\right):\nabla\varphi(x)\,{\rm d}x
+\int_{\mathbb{R}^3}\left(\big(V_{R_j,\varepsilon}-V_{\varepsilon}\big)\otimes V_{R_j}\right):\nabla\varphi(x)\,{\rm d}x\\
\triangleq&I_1+I_2.
\end{split}
\end{equation}
We choose $R_j$ lager than $R$, by the H\"older inequality, we immediately obtain
\begin{equation*}
I_1=\int_{\mathbb{R}^3}\left(V_{R_j,\varepsilon}\otimes\big(V_{R_j,\varepsilon}-V_{\varepsilon}\big)\right):\nabla\varphi(x)\,{\rm d}x
\leq\big\|V_{R_j,\varepsilon}\big\|_{L^2(\mathbb{R}^3)}\big\|V_{R_j,\varepsilon}\big\|_{L^2(\mathbb{B}_R)}\big\|\nabla\varphi\big\|_{L^2(\mathbb{R}^3)}.
\end{equation*}
This estimate together with the strong convergence  \eqref{strong1} yields
\begin{equation}\label{eq.conv2}
I_1\to0\quad\text{as}\quad R_j\to+\infty.
\end{equation}
Similarly, we have
\begin{equation}\label{eq.conv3}
I_2\to0\quad\text{as}\quad R_j\to+\infty.
\end{equation}
Inserting \eqref{eq.conv2} and \eqref{eq.conv3} into \eqref{eq.conv1} leads to
\begin{equation}\label{eq.I1}
\int_{\mathbb{R}^3}V_{R_j,\varepsilon}\cdot\nabla V_{R_j,\varepsilon} \cdot\varphi(x)\,{\rm d}x\to\int_{\mathbb{R}^3}V_{\varepsilon}\cdot\nabla V_{\varepsilon} \cdot\varphi(x)\,{\rm d}x
\end{equation}
as $j$ goes to infinity.
\smallskip

We observe that
\begin{equation}\label{eq.conv4}
\int_{\mathbb{R}^3}U_0\cdot\nabla V_{R_j,\varepsilon}\cdot \varphi(x)\,{\rm d}x= -\int_{\mathbb{R}^3}V_{R_j,\varepsilon}\otimes U_0\cdot\nabla\varphi(x)\,{\rm d}x.
\end{equation}
Since $|U_0|\leq(1+|x|)^{1-2\alpha}$ and $\varphi\in C_0^\infty(\R^3)$, it easy to check that $U_0\cdot\nabla\varphi(x)\in L^2(\mathbb{R}^2)$. Thus the weak convergence \eqref{weak1}  enables us to infer that
\begin{equation}\label{eq.conv5}
\int_{\mathbb{R}^3}V_{R_j,\varepsilon}\otimes U_0\cdot\nabla\varphi(x)\,{\rm d}x\to \int_{\mathbb{R}^3}V_{\varepsilon}\otimes U_0\cdot\nabla\varphi(x)\,{\rm d}x \qquad\text{as}\,\,\, j\to+\infty.
\end{equation}
Plugging \eqref{eq.conv5} into \eqref{eq.conv4} gives
\begin{equation}\label{eq.I2}
\int_{\mathbb{R}^3}U_0\cdot\nabla V_{R_j,\varepsilon} \cdot\varphi(x)\,{\rm d}x= -\int_{\mathbb{R}^3}V_{\varepsilon}\otimes U_0\cdot\nabla\varphi(x)\,{\rm d}x \quad\text{as}\quad R_j\to+\infty.
\end{equation}
In the same fashion as used in \eqref{eq.I2}, one can conclude  that
\begin{equation}\label{eq.I3}
\int_{\mathbb{R}^3}V_{R_j,\varepsilon} \cdot\nabla U_0 \cdot\varphi(x)\,{\rm d}x= -\int_{\mathbb{R}^3}U_0\otimes V_{\varepsilon}\cdot\nabla\varphi(x)\,{\rm d}x \quad\text{as}\quad R_j\to+\infty.
\end{equation}
Collecting \eqref{eq.I1}, \eqref{eq.I2} and \eqref{eq.I3}, we readily get
\begin{equation*}
\int_{\mathbb{R}^3}F\big( V_{R_j,\varepsilon}\big) \varphi(x)\,{\rm d}x\to  \int_{\mathbb{R}^3}F\big( V_{\varepsilon}\big) \varphi(x)\,{\rm d}x \qquad\text{as}\,\,\, R_j\to+\infty.
\end{equation*}
Thus, we obtain that $V_{\epsilon}$ satisfies \eqref{E3.3} in the sense of distribution. By density argument, for $\forall\, \varphi\in H_{\sigma}^{1}(\R^{3})\cap H_{\sigma}^{\alpha}(\R^{3})$ satisfying
$
\big\||\cdot|\varphi(\cdot)\big\|_{L^{2}(\R^{3})}<+\infty,
$
we have
\begin{equation*}
\int_{\R^{3}}\left(\epsilon\nabla V_{\epsilon}: \nabla \varphi+(-\Delta)^{\frac{\alpha}{2}} V_{\epsilon}\cdot (-\Delta)^{\frac{\alpha}{2}} \varphi\right) \,{\rm d}x
=\int_{\R^3}\Big(\frac{2\alpha-1}{2\alpha}V_{\epsilon}+\frac{1}{2\alpha}x\cdot\nabla V_{\epsilon}+F(V_{\epsilon})\Big)\cdot \varphi \,{\rm d}x.
\end{equation*}
 Finally, according to the weak low semi-continuity of the norm, the
limit function $V_{\epsilon}$ satisfies
\begin{equation}\label{eq.I3-add}
\int_{\R^{3}}\Big(\epsilon|\nabla V_{\epsilon}|^{2}+|(-\Delta)^{\frac{\alpha}{2}}V_{\epsilon}|^{2}+\frac{5-4\alpha}{4\alpha}|V_{\epsilon}|^{2}\Big)\,{\rm d}x\leq C(U_{0}).
\end{equation}
This estimate implies the desired result.
\end{proof}
\medskip

Letting $\epsilon\to0$, we obtain the main result of this subsection by the classical diagonalization argument.

\begin{thm}\label{P3.7}
There is a function $V\in  H_{\sigma}^{\alpha}(\R^{3})$ satisfies system \eqref{E3.2} in the sense of distribution.%\footnote{I think the proof of this proposition can be deleted }
\end{thm}
\begin{rem}  i) When $\alpha\leq\frac{5}{8}$, we  do not know whether $V\in H^{\alpha}(\R^{3})$ or not via energy argument due to $U_{0}\cdot \nabla U_{0}\not\in L^{2}(\R^{n})$, and so  we can not construct solutions of \eqref{E1.4a}  by the blow-up argument for this case.

(ii) For $\alpha=1$, the distributional solution of system \eqref{E3.2} established in Theorem \ref{P3.7} that for any $\varphi\in H^{1}(\R^{3})$ with
$$\big\||\cdot|\varphi(\cdot)\big\|_{L^{2}(\R^{3})}<+\infty,$$ we have
\begin{equation}\label{WEN}
\int_{\R^{3}}\nabla V:\nabla \varphi \,{\rm d}x-\int_{\R^{3}}P\text{div}\,\varphi\,\mathrm{d}x
=\int_{\R^3}\Big(\frac{1}{2}V+\frac{1}{2}x\cdot\nabla V+F(V)\Big)\cdot \varphi \,{\rm d}x.
\end{equation}
 This is a starting  point for the study of decay estimate of the solution $V(x)$.
\end{rem}

\subsection{Improved regularity of  $V(x)$}

First of all, we review the following fractional Leibniz estimate which was shown in  \cite{Gr}:

Let
$$
\alpha>0,\, 1\leq p\leq+\infty,\,\, 1<p_{1}, p_{1}, q_{1}, q_{2}\leq+\infty, \,\,\frac{1}{p}=\frac{1}{p_{1}}+\frac{1}{q_{1}}=\frac{1}{p_{2}}+\frac{1}{q_{2}},
$$
then for $f,g \in C_{0}^{\infty}(\R^{n})$, we have
\begin{equation}\label{Leib}
\|(-\Delta)^{\frac{\alpha}{2}}(fg)\|_{L^{p}(\R^{n})}\leq C \|(-\Delta)^{\frac{\alpha}{2}}f\|_{L^{p_{1}}(\R^{n})}\|g\|_{L^{q_{1}}(\R^{n})}
+C\|(-\Delta)^{\frac{\alpha}{2}}g\|_{L^{p_{2}}(\R^{n})}\|f\|_{L^{q_{2}}(\R^{n})}.
\end{equation}

Next we consider some basic properties of  the following non-local Stokes operator:
\begin{equation}\label{S1}
\left\{\begin{aligned}
(-\Delta)^{\alpha}u+\lambda u+\nabla q&=f(x), \quad\ \ x\in \R^{n},\\
\qquad\text{\quad \quad div}\ u&=0, \ \ \ \ \ \ \ \ \ x\in \R^{n},
\end{aligned}\right.
\end{equation}
where $\lambda>0$.
\begin{lem}\label{L3.9-}
Let $f(x)\in C_{0}^{\infty}(\R^{n})$, then system \eqref{S1} admits a solution $(u,q)$ such that
\begin{equation}\label{elliptic estimate}
\|(-\Delta)^{\alpha}u\|_{L^{p}(\R^{n})}+\lambda^{\frac{1}{2}}\|(-\Delta)^{\frac{\alpha}{2}}u\|_{L^{p}(\R^{n})}+\lambda\|u\|_{L^{p}(\R^{n})}+\|\nabla q\|_{L^{p}(\R^{n})}
\leq C \|f\|_{L^{p}(\R^{n})}
\end{equation}
with $1<p<+\infty$.
\end{lem}
\begin{proof}  We introduce
\begin{align*}
q\triangleq-\partial_{i}\Gamma\ast f_{i};\ \ \ \ \   u=(u_{i})\triangleq\big(\{\delta_{ij}\mathcal{B}^{\lambda}_{\alpha}+\partial_{i}\partial_{j}\Gamma\ast\mathcal{B}^{\lambda}_{\alpha}\}\ast f_{j}\big)
\end{align*}
where $\Gamma$, $\mathcal{B}^{\lambda}_{\alpha}$ are the fundamental solutions of the operator $-\Delta$ and  $(-\Delta)^{\alpha}+\lambda I$ respectively, i.e.
$$
\Gamma(x)= \mathcal{F}^{-1}[|\xi|^{-2}]=\frac{1}{(n-2)\omega_{n}|x|^{n-2}}\ \ \mbox{for}\ \ n\geq 3
$$
and
\begin{align*}
\mathcal{B}^{\lambda}_{\alpha}(x)&=\mathcal{F}^{-1}\Big[\frac{1}{\lambda+|\xi|^{\alpha}}\Big]=\frac{1}{(2\pi)^{3}}\int_{\R^{3}}\frac{1}{\lambda+|\xi|^{\alpha}}e^{i\xi\cdot x}\,{\rm d}\xi\\&
=\int_{\R^{3}}\int_{0}^{+\infty}e^{-(\lambda+|\xi|^{\alpha})t}e^{i\xi\cdot x}\,{\rm d}t{\rm d}\xi=\int_{0}^{+\infty}e^{-\lambda t}G_{t}^{\alpha}(x)\,{\rm d}t\in C^{\infty}(\R^{n}\setminus \{0\})
\end{align*}
which is usually called as Besesel potentials. Now we define $u_{1}(x,t)=(G_{t}^{\alpha}\ast f)(x,t)$, then
$$u_{1}(x,t)\in C^{\infty}(\R^{n}\times[0,\infty))\cap L^{\infty}(\R^{n}\times[0,\infty))$$ and $\partial_{t}u_{1}+(-\Delta)^{\alpha}u_{1}=0,\ \ u_{1}(x,0)=f(x)$
for $(x,t)\in \R^{n}\times(0,+\infty)$, from which we immediately infer  $\hat{u}_{1}(x):=(\mathcal{B}^{\lambda}_{\alpha}\ast f)(x)$ is a smooth solution of
$$
(-\Delta)^{\alpha}u+\lambda u=f(x), \quad\ \ x\in \R^{n}.
$$
Thus, $(u,q)$ fulfills equation \eqref{S1}. In addition,  we have
$$
(-\Delta)^{\alpha}u_{j}=\mathcal{F}^{-1}\left[\frac{|\xi|^{2\alpha}}{\lambda+|\xi|^{2\alpha}}\left(\delta_{ij}+\frac{\xi_{i}\xi_{j}}{|\xi|^{2}}\right)\mathcal{F}f_{j}\right]
\triangleq \mathcal{F}^{-1}[m(\xi)\mathcal{F}f_{j}]
$$
with $m(\xi)\in L^{\infty}(\R^{n})$, then by the Calderon-Zygmund inequality, we derive
$$
\|(-\Delta)^{\alpha}u_{j}\|_{L^{p}(\R^{n})}\leq C \|f\|_{L^{p}(\R^{n})}.
$$
On the other hand, we note
$$
\lambda u_{j}=\mathcal{F}^{-1}\left[\frac{\lambda}{\lambda+|\xi|^{\alpha}}\left(\delta_{ij}+\frac{\xi_{i}\xi_{j}}{|\xi|^{2}}\right)\mathcal{F}f_{j}\right]
\triangleq \mathcal{F}^{-1}[m_{1}(\xi)\mathcal{F}f_{j}]
$$
with $m_{1}(\xi)\in L^{\infty}(\R^{n})$, and derive, by using the Calderon-Zygmund inequality again
$$
\|\lambda u_{j}\|_{L^{p}(\R^{n})}\leq C \|f\|_{L^{p}(\R^{n})}.
$$
Due to the the following interpolation equality
$$
\|u\|_{H^{\alpha}_{p}(\R^{n})}\leq C\|u\|^{\frac{1}{2}}_{H^{2\alpha}_{p}(\R^{n})}\|u\|^{\frac{1}{2}}_{L^{p}(\R^{n})},
$$
we obtain
$$
\lambda^{\frac{1}{2}}\|(-\Delta)^{\frac{\alpha}{2}}u\|_{L^{p}(\R^{n})}\leq C \|f\|_{L^{p}(\R^{n})}.
$$
Finally, by the elliptic estimates, we derive
$$
\|\nabla q\|_{L^{p}(\R^{n})}\leq C \|f\|_{L^{p}(\R^{n})}.
$$
Combining the above discussion, we complete the proof.
\end{proof}

Combining  Lemma \ref{L3.9-} with a density argument, we obtain immediately the following regularity result.
\begin{cor}\label{RC}
Let $f\in L^{p}(\R^{n})$ with $1<p<+\infty$, then equation \eqref{S1} admits a unique strong solution $(u,q)\in H^{2\alpha}_{p}(\R^{n})\times \dot{H}^{1}(\R^{n})$ satisfying \eqref{elliptic estimate}
\end{cor}

Following the argument of \cite{Kry}, we present the following regularity lemma of the nonlocal elliptic operator, which plays the  key role in improving
regularity of solution $V(x)$.

\begin{lem}\label{L3.9}
Let $p\in(1, +\infty)$, $g, f_{1},..., f_{n}\in L^{p}(\R^{n})$ and $\lambda>0$, the equation
\begin{equation}\label{R1}
(-\Delta)^{\alpha}u(x)+\lambda u+\nabla q=\partial_{i}f_{i}(x)+g(x)\ \ \ x\in\R^{n}
\end{equation}
has a  solution $u\in H_{p}^{2\alpha-1}(\R^{n})$ such that
\begin{equation}\label{R1.1}
\|(-\Delta)^{\frac{2\alpha-1}{2}}u\|_{L^{p}(\R^{n})}+\lambda^{\frac{2\alpha-1}{2\alpha}}\|u\|_{L^{p}(\R^{n})}\leq C \sum_{i=1}^{n}\|f_{i}\|_{L^{p}(\R^{n})}+\lambda^{-\frac{1}{2\alpha}}\|g\|_{L^{p}(\R^{n})}.
\end{equation}
Furthermore, this  equation can only have one solution in $L^{p}(\R^{n})$.
\end{lem}
\begin{proof}  We consider
\begin{equation*}%\label{R1}
(-\Delta)^{\alpha}v_{i}(x)+\lambda v_{i}(x)+\nabla q_{i}=f_{i}(x) \ \ \ \,\,\text{in}\,\,\,\in\R^{n},
\end{equation*}
and
\begin{equation*}%\label{R1}
(-\Delta)^{\alpha}v(x)+\lambda v(x)+\nabla q=g(x)\ \ \ \,\,\text{in}\,\,\,\in\R^{n}.
\end{equation*}
By the Corollary \ref{RC}, the above two equations have  solution $v_{i}$ and $v$ respectively, such that
\begin{equation}\label{R2}\left.\begin{aligned}
&\|(-\Delta)^{\alpha}v_{i}\|_{L^{p}(\R^{n})}+\lambda^{\frac{1}{2}}\|(-\Delta)^{\frac{\alpha}{2}}v_{i}\|_{L^{p}(\R^{n})}+\lambda\|v_i\|_{L^{p}(\R^{n})}\leq C  \|f_{i}\|_{L^{p}(\R^{n})},\\
\smallskip
&\|(-\Delta)^{\alpha}v\|_{L^{p}(\R^{n})}+\lambda^{\frac{1}{2}}\|(-\Delta)^{\frac{\alpha}{2}}v\|_{L^{p}(\R^{n})}+
\lambda\|v\|_{L^{p}(\R^{n})}\leq C\|g\|_{L^{p}(\R^{n})}.\end{aligned}\right.
\end{equation}
Letting $u\triangleq  \partial_{i}v_{i}+v$, we easily obtain \eqref{R1.1} by the interpolation theory.

To prove $u$ is a unique solution to \eqref{R1}, we only prove that for any  $w\in L^{p}(\R^{n})$ solves \eqref{R1} with $f_{i}=0,\, g=0$ in sense of distribution, then $w\equiv 0$. To do this, we introduce
$$
\psi(y)=\int_{\R^{n}}\varphi(y+x)w(x)\,{\rm d}x,\ \  \ \forall\,\varphi\in C_{0}^{\infty}(\R^{n}).
$$
Since $\varphi\in C_{0}^{\infty}(\R^{n})$ and $w\in L^{p}(\R^{n})$, the function $\psi(y)$ is infinitely smooth and tends to zero as $y\to+\infty$. By a simple calculation, we can derive
\begin{equation}\label{R4}
(-\Delta)^{\alpha}\psi(y)+\lambda \psi(y)=\int_{\R^{n}}w(x)\Big((-\Delta)^{\alpha}\varphi(y+x)+\lambda \varphi(y+x)\Big)\, {\rm d}x=0.
\end{equation}
In fact, by the Helmholtz decomposition, we see that for $\forall\ \varphi\in C_{0}^{\infty}(\R^{n})$
$$
\varphi=\varphi_{1}+\nabla \varphi_{2}
$$
with ${\rm div}\ \varphi_{1}=0$. And then
\begin{align}\label{R4}
\begin{split}
(-\Delta)^{\alpha}\psi(y)+\lambda \psi(y)&=\int_{\R^{n}}w(x)\Big((-\Delta)^{\alpha}\varphi(y+x)+\lambda \varphi(y+x)\Big)\, {\rm d}x\\&
=\int_{\R^{n}}\Big((-\Delta)^{\alpha}w(x)+\lambda w(x)\Big)\Big(\varphi_{1}(y+x)+\nabla \varphi_{2}(y+x)\Big)\, {\rm d}x\\&
=\int_{\R^{n}}\Big((-\Delta)^{\alpha}w(x)+\lambda w(x)\Big)\varphi_{1}(y+x)\, {\rm d}x=0.
\end{split}
\end{align}

Now we claim $\psi(y)\equiv0$. Indeed, we suppose by contradiction that $\sup_{y\in\R^{n}}\psi(y)>0$. Since
 $\psi(y)\to0 $ as $ |y|\to+\infty$,
we can find $y_{0}\in\R^{n}$ such that $u(y_{0})=\sup_{y\in\R^{n}}\psi(y)>0$.

We see that
$$
(-\Delta)^{\alpha}\psi(y_{0})+\lambda \psi(y_{0})=C_{n,\alpha}{\rm p.v.}\int_{\R^{n}}\frac{\psi(y_{0})-\psi(y)}{|y^{0}-y|^{n+\alpha}}\,{\rm d}y+\lambda \psi(y_{0})>0
$$
which contradicts  \eqref{R4}. Therefore we must have   $\sup_{y\in\R^{n}}\psi(y)\le0$. Similarly, we also infer $\inf_{y\in\R^{n}}\psi(y)\geq0$. Thus we have $\psi(y)\equiv0$.
 The arbitrariness of $\varphi$ with $\psi(0)=0$ leads to the conclusion that $w=0$. The lemma is thus proved.
\end{proof}
\begin{cor}\label{C3.11}
Let $f_{i}\ (i=1,2,...,n),\ g\in H_{p}^{m}(\R^{n})$, then \eqref{R1} has a  solution $u\in H_{p}^{2\alpha-1+m}(\R^{n})$, which is a unique solution in $L^{p}(\R^{n})$.
\end{cor}

\begin{defn}
Let $0<\alpha\leq1$, we say that $u\in D'(\R^{n})$ satisfies $(-\Delta)^{\alpha}u+u=0$ in an open set $\Omega$ if for every $\phi\in C_{0}^{\infty}(\Omega)$ such that
\begin{equation}\label{a-harmonic}
\big\langle u, (-\Delta)^{\alpha}\phi+\phi\big\rangle=0.
\end{equation}
\end{defn}
\begin{lem}[\cite{La,Silv}]\label{Regu of harmo}
Let $u\in L_{\rm loc}^{1}(\R^{n})$,  and it satisfies \eqref{a-harmonic} and
  $$\int_{\R^{n}}\frac{\big|u(x)\big|}{1+|x|^{n+2\alpha}}\,\mathrm{d}x<+\infty.$$ Then $u$ is smooth in $\R^n$.
\end{lem}
\begin{thm}\label{P3.12}
Let $\frac{5}{6}< \alpha\leq 1$, then the distributional solution $V(x)$ of system \eqref{E3.2} established in Theorem~\ref{P3.7} is smooth.
\end{thm}
\begin{proof}
 We now rewrite system \eqref{E3.2} as
$$
(-\Delta)^{\alpha}V+\frac{2-\alpha}{\alpha}V+\nabla P={\rm div} \,\big(G(U_{0},V)\big),\quad\text{div}\,V=0,
$$
where
$$
G(U_{0},V)=\frac{1}{2\alpha}x\otimes V-U_{0}\otimes U_{0}-U_{0}\otimes V-V\otimes U_{0}-V\otimes V.
$$
%and $\mathbb{P}$ is the Leray-Hopf projection.

{\bf Step 1}. Now we consider
\begin{equation}\label{Regulairty3.37}
(-\Delta)^{\alpha}V_{1}+\frac{2-\alpha}{\alpha}V_{1}+\nabla P_{1} ={\rm div} (G_{1}(U_{0},V)),\ \ \mbox{div}\ V_{1}=0,
\end{equation}
where
$$
G_{1}(U_{0},V)=\varphi_{1}\big(\frac{x}{2\alpha}\otimes V-U_{0}\otimes U_{0}-U_{0}\otimes V-V\otimes U_{0}-V\otimes V\big),
$$
and $\varphi_{1}\in C_{0}^{\infty}(\B_{2R}), \,\varphi_{1}\equiv 1$ in $\B_{(2-\frac{1}{2})R}$ with some fixed $R>0$.

 Due to $V\in H^{\alpha}(\R^{3})$, we derive by \eqref{Leib} that
$V\otimes V\in H^{\alpha}_{\frac{3}{3-\alpha}}(\R^{3})$, and then we have
$$
G_{1}(U_{0},V)\in H^{\alpha}_{\frac{3}{3-\alpha}}(\R^{3}).
$$ This inclusion  together with  Corollary \ref{C3.11} implies \eqref{Regulairty3.37} admits a unique solution $V_{1}\in  H^{3\alpha-1}_{\frac{3}{3-\alpha}}(\R^{3})$.

 Letting $w_{1}=V-V_{1}$, we obtain
$$
(-\Delta)^{\alpha}w_{1}+\frac{2-\alpha}{\alpha}w_{1}+\nabla \tilde{P}_{1}=\tilde{f}(x)\quad \ \ \mbox{in}\ \ \R^{3},
$$
where
$$
\tilde{f}(x)=0\ \ \  \mbox{for}\ \  x\in \B_{(2-\frac{1}{2})R}\quad\text{and}\quad  \tilde{P}_{1}=P-P_{1}.
$$
Since $\tilde{P}_1$ is harmonic in $\B_{(2-\frac{1}{2})R}$,  we have $\tilde{P}_1\in C^{\infty}(\B_{(2-\frac{1}{2})R})$.  Now let $\tilde{w}_{1}$
be the solution of
$$
(-\Delta)^{\alpha}\tilde{w}_{1}+\tilde{w}_{1}=\nabla \tilde{P}_1\eta_{1}(x)\in C_{0}^{\infty}(\R^{3}),
$$
where $\eta_{1}(x)\in C_{0}^{\infty}(\B_{(2-\frac{3}{4})R})$ and $\eta_{1}(x)\equiv1$ for $x\in \B_{(2-\frac{7}{8})R}$. By the classical theory, we know that $\tilde{w}_{1}\in C^{\infty}(\R^{3})\cap L^{\infty}(\R^{3})$. Again letting $\tilde{v}_{1}=w_{1}-\tilde{w}_{1}$, we see that
$$
(-\Delta)^{\alpha}\tilde{v}_{1}+\tilde{v}_{1}=0 \, \ \ \ \text{in}\ \ \B_{(2-\frac{7}{8})R}.
$$
In terms of Lemma \ref{Regu of harmo}, we derive $\tilde{v}_{1}\in C^{\infty}(\B_{(2-\frac{7}{8})R})$. Therefore we have $w_{1}\in C^{\infty}(\B_{(2-\frac{15}{16})R})$.

Since $V=V_{1}+w_{1}$, we infer from the fact $\alpha>\frac{5}{6}$ and the Sobolev embedding theorem that
$$
V\in H^{3\alpha-1}_{\frac{3}{3-\alpha}}(\B_{(2-\frac{15}{16})R})\subset H^{\alpha}(\B_{(2-\frac{15}{16})R}).
$$

{\bf Step 2}. {\em Bootstrapping Arguments.}  Again, consider
\begin{equation}\label{Regulairty3.38}
(-\Delta)^{\alpha}V_{2}+\frac{2-\alpha}{\alpha}V_{2}+\nabla P_{2} ={\rm div} \,\big(G_{2}(U_{0},V)\big),\ \ \mbox{div}\ V_{2}=0
\end{equation}
where
$$
G_{2}(U_{0},V)=\varphi_{2}\big(\frac{x}{2\alpha}\otimes V-U_{0}\otimes U_{0}-U_{0}\otimes V-V\otimes U_{0}-V\otimes V\big),
$$
and
$$
\varphi_{2}\in C_{0}^{\infty}\big(\B_{(2-\frac{15}{16})R}\big),\ \  \varphi_{2}\equiv 1\ \  \mbox{in}\ \ \B_{(2-\frac{32}{33})R}.
$$
The fact
$$
V\in H^{3\alpha-1}_{\frac{3}{3-\alpha}}(\B_{(2-\frac{15}{16})R})\hookrightarrow
\left \{
\begin{array}{ll}
L^{\frac{3}{4-4\alpha}}(\B_{(2-\frac{15}{16})R})\ \ & \text{if}\ \alpha<1,\\\\
L^{q}(\B_{(2-\frac{15}{16})R})\quad\text{for each}\,\,\, q>1 \ \ & \text{if}\ \alpha=1,
\end{array}
\right.
$$
and
$$
\nabla V\in H^{3\alpha-2}_{\frac{3}{3-\alpha}}(\B_{(2-\frac{15}{16})R})\hookrightarrow L^{\frac{3}{5-4\alpha}}(\B_{(2-\frac{15}{16})R})
$$
allow us to derive
$$
{\rm div}\, \big(G_{2}(U_{0},V)\big)\in L^{p}(\R^{3})\ \ \ \  \mbox{with}\ \ p=\frac{3}{9-8\alpha}>1,
$$
where we have used the fact that $\alpha>\frac{5}{6}$.

 By Corollary \ref{RC}, \eqref{Regulairty3.38} has a unique solution $V_{2}\in H_{p}^{2\alpha}(\R^{3})$ with $p=\frac{3}{9-8\alpha}$. Performing the same argument as Step 1, we have
$$
V\in H_{p}^{2\alpha}(\B_{(2-\frac{2^{7}-1}{2^{7}})R})\subset H^{3\alpha-1}_{\frac{3}{3-\alpha}}(\B_{(2-\frac{2^{7}-1}{2^{7}})R}),.
$$
By the Sobolev  embedding theorem, we see
$$
V\in\left \{
\begin{array}{ll}
L^{\frac{3}{9-10\alpha}}(\B_{(2-\frac{2^{7}-1}{2^{7}})R}) & \ \ \mbox{if}\ \ \alpha<\frac{9}{10},\\\\
L^{q}(\B_{(2-\frac{2^{7}-1}{2^{7}})R})\quad\text{for each}\,\,\, q>1  & \ \ \mbox{if}\ \ \alpha=\frac{9}{10},\\\\
L^{\infty}(\B_{(2-\frac{2^{7}-1}{2^{7}})R}) & \ \ \mbox{if}\ \ \alpha>\frac{9}{10}.
\end{array}
\right.
$$
Now repeating our previous argument a finite number of times,  we  get  $V\in L^{\infty}(\B_{(2-\frac{2^{k_{0}}-1}{2^{k_{0}}})R})$ with some integer $k_{0}>0$.
Once $V\in L^{\infty}(\B_{(2-\frac{2^{k_{0}}-1}{2^{k_{0}}})R})$,  we derive that $V\cdot\nabla V\in H_{p}^{2\alpha-1}(\B_{(2-\frac{2^{k_{0}}-1}{2^{k_{0}}})R})$ by \eqref{Leib} and immediately obtain by using the argument of Step 1 that
$$
V\in H_{p}^{2(2\alpha-1)}\Big(\B_{(2-\frac{2^{k_{0}+3}-1}{2^{k_{0}+3}})R}\Big).
$$
Furthermore, repeating the process in Step 1, we  can show by induction that
$$
V\in H_{p}^{m(2\alpha-1)}\big(\B_{R}\big)\ \ \ \ \mbox{for all}\ \ m>0.
$$
This implies $V$ is smooth.
\end{proof}

\subsection{Decay estimate for $V(x)$ when $\alpha=1$}

In this subsection, we will prove  a few decay estimates of the  weak solution  to equations \eqref{E3.2}, which is the key point in the proof of the case $\alpha=1$. \begin{thm}\label{thm-1-I}
Assume that  $V\in H^{1}(\R^{3})$ is the weak solution of problem \eqref{E3.2} established in Theorem~\ref{P3.7}. Then $V$ is smooth and there exists a constant $C>0$ such that
\[\big\|\langle\cdot\rangle\nabla P(\cdot)\big\|_{\dot{B}^0_{\infty,\infty}(\R^3)}<+\infty,\]
$$
\big|V\big|(x)\leq C\big(1+|x|\big)^{-3 }\log(1+|x|)\qquad\text{for all}\,\,\, x\in\R^3,
$$
and
$$
\big|D V\big|(x)\leq C\big(1+|x|\big)^{-3}\qquad\text{for all}\,\,\, x\in\R^3.
$$
\end{thm}
By the well-known theorem of De Rham (See for example [Proposition 1.1, \cite{Temam}]), there exists a pressure $P\in L^2(\R^3)$  governed by
\begin{equation*}\label{eq.PE}
P=\sum_{i,j=1}^3\frac{1}{4\pi}\partial^2_{x_i,x_j}\int_{\mathbb{R}^3}\frac{1}{|x-y|}\big(V^iV^j+U^i_0 V^j+V^iU^j_0+U^i_0 U^j_0\big)\,\mathrm{d}y
\end{equation*}
such that for all vector fields $\varphi\in H^1(\mathbb{R}^3)$ satisfying $\big\||\cdot |\varphi(\cdot)\big\|_{L^2(\mathbb{R}^3)}<+\infty,$
the couple $(V,P)$ fulfills
\begin{equation}\label{eq.weak}
\begin{split}
&\int_{\mathbb{R}^3}\nabla V:\nabla\varphi\,\mathrm{d}x-\frac12\int_{\mathbb{R}^3}x\cdot \nabla V\cdot\varphi\,\mathrm{d}x-\frac12\int_{\mathbb{R}^3} V\cdot\varphi\,\mathrm{d}x-\int_{\mathbb{R}^3}P\, \mathrm{div}\,\varphi\,\mathrm{d}x\\
=&\int_{\mathbb{R}^3}V\cdot \nabla \varphi\cdot V\,\mathrm{d}x-\int_{\mathbb{R}^3}U_0\cdot \nabla  V\cdot \varphi\,\mathrm{d}x-\int_{\mathbb{R}^3}V\cdot \nabla  U_0\cdot \varphi\,\mathrm{d}x-\int_{\mathbb{R}^3}U_0\cdot \nabla  U_0\cdot \varphi\,\mathrm{d}x.
\end{split}
\end{equation}
From Theorem \ref{P3.12}, it follows that $V\in C^{\infty}(\R^3)$.

Since $P$  satisfies
 \begin{equation}\label{eq.P}
-\Delta P=\mathrm{div}\,\big(V\cdot\nabla V+U_{0}\cdot\nabla V+V\cdot\nabla U_{0}+U_{0}\cdot\nabla U_{0}\big),
\end{equation}
  we have by the elliptic theory that  $P\in C^{\infty}(\R^3)$.

From Proposition  \ref{prop-U1}, we can conclude that $U_0\in L^p(\R^3)$ for all $p\in(3,+\infty)$, $\nabla U_0\in L^2(\R^3)$ and
 \begin{equation*}\
\sup_{x\in\mathbb{R}^3}\big|\langle x\rangle \nabla U_0\big|(x)+\sup_{x\in\mathbb{R}^3}\big|\langle x\rangle U_0\big|(x)<+\infty.
\end{equation*}
These estimates together with the properties of singular operator and the imbedding theorem impliy that for each $\frac32<p\leq3,$
\begin{equation*}\label{eq.Pe1}
\begin{split}
\|P\|_{L^p(\mathbb{R}^3)}\leq&C\|V\|^2_{L^{2p}(\mathbb{R}^3)}+ C\|V\|_{L^{2p}(\mathbb{R}^3)}\|U_{0}\|_{L^{2p}(\mathbb{R}^3)}+C\|U_{0}\|^2_{L^{2p}(\mathbb{R}^3)}\\
\leq& C\|V\|^2_{H^1(\mathbb{R}^3)} +C\|U_{0}\|^2_{L^{2p}(\mathbb{R}^3)}<+\infty.
\end{split}
\end{equation*}
To accomplish the decay estimate, we first prove the  $H^1(\R^3)$-estimate of $|x|V$, which is the key estimate in our proof.
\begin{prop}\label{prop-xv}
	Let the couple $(V,P)\in H^1(\R^3)\times L^2(\R^3)$  satisfying  \eqref{eq.weak}.
	Then we have
	\[W(x)\triangleq|x|V(x)\in H^1(\mathbb{R}^3).\]
\end{prop}
\begin{proof}
Denoting $h_\varepsilon(x)\triangleq\frac{|x|}{(1+\varepsilon|x|^2)^{\frac34}}$ with $\varepsilon>0$, it is easy to check that $h^2_{\varepsilon}(x) V(x)\in H^1(\mathbb{R}^2)$ and satisfies  $$\big\||\cdot |h^2_\varepsilon V(\cdot)\big\|_{L^2(\mathbb{R}^3)}\leq C(\varepsilon).$$
 Now, choosing $\varphi(x)\triangleq\frac{|x|^2}{(1+\varepsilon|x|^2)^{\frac32}}V(x)=h^2_\varepsilon(x) V(x)$ in equality \eqref{eq.weak}, we easily  find that the vector field $W_\varepsilon(x)\triangleq h_\varepsilon(x) V(x)$ fulfills
 \begin{equation}\label{eq.weakIII-1}
 \begin{split}
 &\int_{\mathbb{R}^3}\nabla V:\nabla\big( h_\varepsilon W_\varepsilon\big)\,\mathrm{d}x-\frac12\int_{\mathbb{R}^3}x\cdot \nabla V\cdot \big( h_\varepsilon W_\varepsilon\big)\,\mathrm{d}x-\frac12\int_{\mathbb{R}^3} W^2_\varepsilon\,\mathrm{d}x\\
 =&\int_{\mathbb{R}^3}P \,\text{div}\,\big( h_\varepsilon W_\varepsilon\big)\,\mathrm{d}x+\int_{\mathbb{R}^3}V\cdot \nabla \big( h_\varepsilon W_\varepsilon\big)\cdot V\,\mathrm{d}x-\int_{\mathbb{R}^3}U_0\cdot \nabla  V\cdot \big( h_\varepsilon W_\varepsilon\big)\,\mathrm{d}x\\
 &-\int_{\mathbb{R}^3}V\cdot \nabla  U_0\cdot \big( h_\varepsilon W_\varepsilon\big)\,\mathrm{d}x-\int_{\mathbb{R}^3}U_0\cdot \nabla  U_0\cdot \big( h_\varepsilon W_\varepsilon\big)\,\mathrm{d}x.
 \end{split}
 \end{equation}
 A simple calculation yields that
 \begin{equation}\label{eq.weakIII-2}
 \begin{split}
 \int_{\mathbb{R}^3}\nabla V:\nabla\big( h_\varepsilon W_\varepsilon\big)\,\mathrm{d}x
=&\int_{\mathbb{R}^3}\nabla V:\big(\nabla h_\varepsilon\otimes W_\varepsilon\big)\,\mathrm{d}x+\int_{\mathbb{R}^3}\big(h_\varepsilon\nabla V\big):\nabla W_\varepsilon\,\mathrm{d}x\\
 =&\|\nabla W_\varepsilon\|^2_{L^2(\mathbb{R}^3)}+\int_{\mathbb{R}^3}\nabla V:\big(\nabla h_\varepsilon\otimes W_\varepsilon\big)\,\mathrm{d}x-\int_{\mathbb{R}^3}\big(\nabla h_\varepsilon\otimes V\big):\nabla W_\varepsilon\,\mathrm{d}x.
 \end{split}
 \end{equation}
 Since
 \begin{align*}
 x\cdot\nabla\left(\frac{|x|}{(1+\varepsilon|x|^2)^{\frac34}} V\right)=&\frac{|x|}{(1+\varepsilon|x|^2)^{\frac34}} \,x\cdot\nabla V +\frac{|x|}{(1+\varepsilon|x|^2)^{\frac34}}  V -\frac32 \frac{\varepsilon|x|^{3}}{(1+\varepsilon|x|^2)^{\frac74}}  V\\
 =&h_\varepsilon\big( \,x\cdot\nabla V\big) +W_\varepsilon -\frac32\, \frac{\varepsilon|x|^{2}}{1+\varepsilon|x|^2}  W_\varepsilon,
 \end{align*}
 we have
 \begin{equation}\label{eq.weakIII-3}
 \begin{split}
&-\frac12\int_{\mathbb{R}^3}x\cdot \nabla V\cdot \big( h_\varepsilon W_\varepsilon\big)\,\mathrm{d}x-\frac12\int_{\mathbb{R}^3} W^2_\varepsilon\,\mathrm{d}x\\
=&-\frac12\int_{\mathbb{R}^3}x\cdot \nabla W_\varepsilon\cdot  W_\varepsilon\,\mathrm{d}x-\frac34\int_{\mathbb{R}^3}\frac{\varepsilon|x|^{2}}{1+\varepsilon|x|^2}  W_\varepsilon\cdot W_\varepsilon\,\mathrm{d}x\\
=&\frac34\int_{\mathbb{R}^3}W^2_\varepsilon\,\mathrm{d}x-\frac34\int_{\mathbb{R}^3}\frac{\varepsilon|x|^{2}}{1+\varepsilon|x|^2}  W^2_\varepsilon\,\mathrm{d}x=\frac34\int_{\mathbb{R}^3}\frac{1}{1+\varepsilon|x|^2}  W^2_\varepsilon\,\mathrm{d}x.
 \end{split}
 \end{equation}
 Here we used the fact that $x(W^\varepsilon)^2\in W^{1,1}(\R^3).$

  Setting $g_\varepsilon(x)\triangleq\frac{1}{\sqrt{1+\varepsilon |x|^2}}$ and plugging estimates \eqref{eq.weakIII-2} and \eqref{eq.weakIII-3} in \eqref{eq.weakIII-1}, we immediately have
  \begin{equation}\label{eq.weakIII-4}
\begin{split}
&\big\|\nabla W_\varepsilon\big\|^2_{L^2(\mathbb{R}^3)}+\frac34\big\| W_\varepsilon\big\|^2_{L^2(\mathbb{R}^3)}\\
=&\int_{\mathbb{R}^3}P \,\text{div}\,\big( h_\varepsilon W_\varepsilon\big)\,\mathrm{d}x-\int_{\mathbb{R}^3}V\cdot \nabla \big( h_\varepsilon W_\varepsilon\big)\cdot V\,\mathrm{d}x-\int_{\mathbb{R}^3}\big( h_\varepsilon U_0\big)\cdot \nabla  V\cdot W_\varepsilon\,\mathrm{d}x\\
&-\int_{\mathbb{R}^3}W_\varepsilon\cdot \nabla  U_0\cdot  W_\varepsilon\,\mathrm{d}x-\int_{\mathbb{R}^3} \big( h_\varepsilon U_0\big)\cdot \nabla  U_0\cdot W_\varepsilon\,\mathrm{d}x-\int_{\mathbb{R}^3}\nabla V:\big(\nabla h_\varepsilon\otimes W_\varepsilon\big)\,\mathrm{d}x\\&+\int_{\mathbb{R}^3}\big(\nabla h_\varepsilon\otimes V\big):\nabla W_\varepsilon\,\mathrm{d}x.
\end{split}
\end{equation}
Thanks to $\mathrm{div}\,V=0,$ we have that
\begin{align*}
&\int_{\mathbb{R}^3}V\cdot \nabla \big( h_\varepsilon W_\varepsilon\big)\cdot V\,\mathrm{d}x\\
=&\int_{\mathbb{R}^3}V\cdot \nabla  W_\varepsilon\cdot W_\varepsilon \,\mathrm{d}x+\int_{\mathbb{R}^3}V\cdot\big( \nabla  h_\varepsilon \otimes W_\varepsilon\big)\cdot V\,\mathrm{d}x\\
=&\int_{\mathbb{R}^3}V\cdot\big( \nabla  h_\varepsilon \otimes W_\varepsilon\big)\cdot V\,\mathrm{d}x\\
=&\int_{\mathbb{R}^3}V\cdot\left( \frac{x}{|x|(1+\varepsilon|x|^2)^{\frac34}}\otimes W_\varepsilon\right)\cdot V\,\mathrm{d}x-\frac32\int_{\mathbb{R}^3}V\cdot\left( \frac{\varepsilon |x|x}{(1+\varepsilon|x|^2)^{\frac74}}\otimes W_\varepsilon\right)\cdot V\,\mathrm{d}x.
\end{align*}
By the H\"older inequality and the Young inequality, one has
\begin{align*}
\int_{\mathbb{R}^3}V\cdot\left( \frac{x}{|x|(1+\varepsilon|x|^2)}\otimes W_\varepsilon\right)\cdot V\,\mathrm{d}x
\leq&\|V\|^2_{L^{\frac{12}{5}}(\mathbb{R}^3)}\|W_\varepsilon\|_{L^6(\mathbb{R}^3)}\\
\leq&C\|V\|^4_{L^{\frac{12}{5}}(\mathbb{R}^3)}+\frac{1}{128}\|\nabla W_\varepsilon\|^2_{L^2(\mathbb{R}^3)}
\end{align*}
and
\begin{align*}
-\frac32\int_{\mathbb{R}^3}V\cdot\left( \frac{\varepsilon |x|x}{(1+\varepsilon|x|^2)^2}\otimes W_\varepsilon\right)\cdot V\,\mathrm{d}x
\leq&2\|V\|^2_{L^{\frac{12}{5}}(\mathbb{R}^3)}\|W_\varepsilon\|_{L^6(\mathbb{R}^3)}\\
\leq&C\|V\|^4_{L^{\frac{12}{5}}(\mathbb{R}^3)}+\frac{1}{128}\|\nabla W_\varepsilon\|^2_{L^2(\mathbb{R}^3)}.
\end{align*}
Thus, we have
\begin{equation*}
\int_{\mathbb{R}^3}V\cdot \nabla \big( h_\varepsilon W_\varepsilon\big)\cdot V\,\mathrm{d}x\leq C\|V\|^4_{L^{\frac{12}{5}}(\mathbb{R}^3)}+\frac{1}{64}\|\nabla W_\varepsilon\|^2_{L^2(\mathbb{R}^3)}.
\end{equation*}
Next, we obtain by some computations that
\begin{align*}
&-\int_{\mathbb{R}^3}\nabla V:\big(\nabla h_\varepsilon\otimes W_\varepsilon\big)\,\mathrm{d}x\\
=&-\int_{\mathbb{R}^3}\nabla V:\left(\frac{x}{|x|(1+\varepsilon|x|^2)^{\frac{3}{4}}}\otimes W_\varepsilon\right)\,\mathrm{d}x+\frac32\int_{\mathbb{R}^3}\nabla V:\left(\frac{\varepsilon|x|x}{(1+\varepsilon|x|^2)^{\frac74}}\otimes W_\varepsilon\right)\,\mathrm{d}x.
\end{align*}
By the H\"older inequality and the Young inequality, one has
\begin{align*}
-\int_{\mathbb{R}^3}\nabla V:\left(\frac{x}{|x|(1+\varepsilon|x|^2)^{\frac{3}{4}}}\otimes W_\varepsilon\right)\,\mathrm{d}x\leq&\|\nabla V\|_{L^2(\mathbb{R}^3)}\|g_\varepsilon W_\varepsilon\|_{L^2(\mathbb{R}^3)}\\
\leq&C\|\nabla V\|^2_{L^2(\mathbb{R}^3)}+\frac{1}{128}\|g_\varepsilon W_\varepsilon\|^2_{L^2(\mathbb{R}^3)}
\end{align*}
and
\begin{align*}
\frac32\int_{\mathbb{R}^3}\nabla V:\left(\frac{\varepsilon|x|x}{(1+\varepsilon|x|^2)^{\frac{7}{4}}}\otimes W_\varepsilon\right)\,\mathrm{d}x\leq&2\|\nabla V\|_{L^2(\mathbb{R}^3)}\|g_\varepsilon W_\varepsilon\|_{L^2(\mathbb{R}^3)}\\
\leq&C\|\nabla V\|^2_{L^2(\mathbb{R}^3)}+\frac{1}{128}\|g_\varepsilon W_\varepsilon\|^2_{L^2(\mathbb{R}^3)}.
\end{align*}
Therefore we have
\begin{equation*}
-\int_{\mathbb{R}^3}\nabla V:\big(\nabla h_\varepsilon\otimes W_\varepsilon\big)\,\mathrm{d}x\leq C\|\nabla V\|^2_{L^2(\mathbb{R}^3)}+\frac{1}{64}\|g_\varepsilon W_\varepsilon\|^2_{L^2(\mathbb{R}^3)}.
\end{equation*}
 Note that
 \begin{align*}
 &\int_{\mathbb{R}^3}\big(\nabla h_\varepsilon\otimes V\big):\nabla W_\varepsilon\,\mathrm{d}x\\
 =&\int_{\mathbb{R}^3}\left(\frac{|x|}{|x|(1+\varepsilon|x|^2)^{\frac34}}\otimes V\right):\nabla W_\varepsilon\,\mathrm{d}x-\frac32\int_{\mathbb{R}^3}\left(\frac{\varepsilon |x|x}{(1+\varepsilon|x|^2)^{\frac74}}\otimes V\right):\nabla W_\varepsilon\,\mathrm{d}x,
 \end{align*}
 we have by the H\"older inequality and the Young inequality that
 \begin{align*}
 \int_{\mathbb{R}^3}\left(\frac{x}{|x|(1+\varepsilon|x|^2)^{\frac34}}\otimes V\right):\nabla W_\varepsilon\,\mathrm{d}x\leq&\| V\|_{L^2(\mathbb{R}^3)}\|\nabla W_\varepsilon\|_{L^2(\mathbb{R}^3)}\\
 \leq&C\| V\|^2_{L^2(\mathbb{R}^3)}+\frac{1}{128}\|\nabla W_\varepsilon\|^2_{L^2(\mathbb{R}^3)}
 \end{align*}
 and
 \begin{align*}
-\frac32\int_{\mathbb{R}^3}\left(\frac{\varepsilon |x|x}{(1+\varepsilon|x|^2)^{\frac{7}{4}}}\otimes V\right):\nabla W_\varepsilon\,\mathrm{d}x\leq&2\| V\|_{L^2(\mathbb{R}^3)}\|\nabla W_\varepsilon\|_{L^2(\mathbb{R}^3)}\\
 \leq&C\| V\|^2_{L^2(\mathbb{R}^3)}+\frac{1}{128}\|\nabla W_\varepsilon\|^2_{L^2(\mathbb{R}^3)}.
 \end{align*}
 So we have
 \begin{equation*}
 \int_{\mathbb{R}^3}\big(\nabla h_\varepsilon\otimes V\big):\nabla W_\varepsilon\,\mathrm{d}x\leq C\| V\|^2_{L^2(\mathbb{R}^3)}+\frac{1}{64}\|\nabla W_\varepsilon\|^2_{L^2(\mathbb{R}^3)}.
 \end{equation*}
By the H\"older inequality and the Young inequality again, one has that
\begin{align*}
-\int_{\mathbb{R}^3}V\cdot\nabla U_0\cdot\big(h_\varepsilon W_\varepsilon\big) \,\mathrm{d}x=&-\int_{\mathbb{R}^3}\frac{|x|}{(1+\varepsilon|x|^2)^{\frac34}}V\cdot\nabla U_0 \cdot W_\varepsilon \,\mathrm{d}x\\
\leq&\sup_{x\in\mathbb{R}^3}\big||x|\nabla U_0\big|(x)\,\|V\|_{L^2(\mathbb{R}^3)}
\|g_\varepsilon W_\varepsilon\|_{L^2(\mathbb{R}^3)  }\\
\leq&C\Big(\sup_{x\in\mathbb{R}^3}\big|\langle x\rangle \nabla U_0\big|(x)\Big)^2\|V\|^2_{L^2(\mathbb{R}^3)}+\frac{1}{128}
\|g_\varepsilon W_\varepsilon\|^2_{L^2(\mathbb{R}^3)  }
\end{align*}
and
\begin{align*}
-\int_{\mathbb{R}^3}U_0\cdot\nabla U_0\cdot \big(h_\varepsilon W_\varepsilon\big) \,\mathrm{d}x=&-\int_{\mathbb{R}^3}\frac{|x|}{(1+\varepsilon|x|^2)^{\frac34}}U_0\cdot\nabla U_0 \cdot W_\varepsilon \,\mathrm{d}x\\
\leq &\sup_{x\in\mathbb{R}^3}\big||x| U_0\big|(x)\,\|\nabla U_0\|_{L^2(\mathbb{R}^3)}\|g_\varepsilon W_\varepsilon\|_{L^2(\mathbb{R}^3)}\\
\leq &C\Big(\sup_{x\in\mathbb{R}^3}\big|\langle x \rangle U_0\big|(x)\Big)^2\|\nabla U_0\|^2_{L^2(\mathbb{R}^3)}+\frac{1}{128}\|g_\varepsilon W_\varepsilon\|^2_{L^2(\mathbb{R}^3)}.
\end{align*}
 Similarly, one has
\begin{align*}
-\int_{\mathbb{R}^3}U_0\cdot \nabla  V\cdot \big(h_\varepsilon W_\varepsilon\big)\,\mathrm{d}x=&-\int_{\mathbb{R}^3}\frac{|x|}{(1+\varepsilon|x|^2)^{\frac34}}U_0\cdot\nabla V\cdot W_\varepsilon \,\mathrm{d}x\\
\leq &\sup_{x\in\mathbb{R}^3}\big||x| U_0\big|(x)\,\|\nabla V\|_{L^2(\mathbb{R}^3)}\|g_\varepsilon W_\varepsilon\|_{L^2(\mathbb{R}^3)}\\
\leq &16\Big(\sup_{x\in\mathbb{R}^3}\big|\langle x \rangle U_0\big|(x)\Big)^2\|\nabla V\|^2_{L^2(\mathbb{R}^3)}+\frac{1}{128}\|g_\varepsilon W_\varepsilon\|^2_{L^2(\mathbb{R}^3)}.
\end{align*}
Using $\text{div} V=0$, we have
\begin{equation*}\label{ep}
\int_{\mathbb{R}^3}P \,\text{div}\,\big( h_\varepsilon W_\varepsilon\big)\,\mathrm{d}x=\int_{\mathbb{R}^3}\frac{x}{|x|(1+\varepsilon|x|^2)^{\frac34}}\cdot W_{\varepsilon}P\,\mathrm{d}x-\frac32 \int_{\mathbb{R}^3}\frac{\varepsilon|x|x}{(1+\varepsilon|x|^2)^{\frac74}}\cdot W_{\varepsilon}P\,\mathrm{d}x.
\end{equation*}
Since $P\in L^2(\mathbb{R}^3)$, in terms of the H\"older inequality, we have
\begin{align*}
\int_{\mathbb{R}^3}\frac{x}{|x|(1+\varepsilon|x|^2)^{\frac34}}\cdot W_{\varepsilon}P\,\mathrm{d}x\leq&\|P\|_{L^2(\mathbb{R}^3)}\|g_\varepsilon W_\varepsilon \|_{L^2(\mathbb{R}^3)}\\
\leq&C\|P\|_{L^2(\mathbb{R}^3)}^2+\frac{1}{32}\|g_\varepsilon W_\varepsilon\|^2_{L^2(\mathbb{R}^3)}
\end{align*}
and
\begin{align*}
-\frac32\ \int_{\mathbb{R}^3}\frac{\varepsilon |x|x}{(1+\varepsilon|x|^2)^{\frac74}}\cdot W_{\varepsilon}P\,\mathrm{d}x\leq&2\|P\|_{L^2(\mathbb{R}^3)}\|g_\varepsilon W_\varepsilon \|_{L^2(\mathbb{R}^3)}\\
\leq&C\|P\|_{L^2(\mathbb{R}^3)}^2+\frac{1}{32}\|g_\varepsilon W_\varepsilon\|^2_{L^2(\mathbb{R}^3)}.
\end{align*}
Plugging both estimates above in \eqref{eq.weakIII-4}, we obtain
\begin{equation*}\label{epI}
\int_{\mathbb{R}^3}P \,\text{div}\,\big( h_\varepsilon W_\varepsilon\big)\,\mathrm{d}x\leq C\|P\|_{L^2(\mathbb{R}^3)}^2+\frac{1}{16}\|g_\varepsilon W_\varepsilon\|^2_{L^2(\mathbb{R}^3)}.
\end{equation*}
Collecting  all these estimates and using the fact that $(V,P)\in H^1(\R^3)\times L^2(\R^3)$, we eventually have
\begin{equation*}
\|g_\varepsilon W_\varepsilon\|^2_{L^2(\mathbb{R}^3)}+\|\nabla W_\varepsilon\|^2_{L^2(\mathbb{R}^3)}\leq C\big(U_0\big).
\end{equation*}
By the Lebesgue dominated convergence theorem, we get by taking $\varepsilon\to0+$ in the above inequality that
\begin{equation*}
\|W\|^2_{L^2(\mathbb{R}^3)}+\|\nabla W\|^2_{L^2(\mathbb{R}^3)}\leq C\big(U_0\big).
\end{equation*}
This implies the desired result in Proposition \ref{prop-xv}.
\end{proof}
With this weighted $H^1$-estimate,  we are going to improve the regularity of solution. Before doing this, we need to establish the following regularity estimate.

\begin{lem}\label{RNS}
	Let $f\in L^2(\mathbb{R}^3)$ and the divergence-free vector field $\overline{V}\in H^1(\mathbb{R}^3)$.  Assume that  $(V,P )$ is a  weak solution  of the following problem
	 \begin{equation*}\label{eq.NS}
	\left\{\begin{aligned}
	&-\Delta V+\overline{V}\cdot\nabla V-\frac12\big(x\cdot\nabla V+V\big)+\nabla P=f \\
	&\mathrm{div}\,V=0
	\end{aligned}\right. \;\quad\;\text{\rm in}\quad\mathbb{R}^3
	\end{equation*}
	that is,   $V\in L^2(\R^3),\,|x|V(x)\in L^2(\R^3)$, $P\in L^2(\mathbb{R}^3)$, and  for all vector fields $\varphi\in H^1(\mathbb{R}^3),$
	  \begin{equation}\label{eq.weakLNS}
	 \int_{\mathbb{R}^3}\nabla V:\nabla\varphi\,\mathrm{d}x-\frac12\int_{\mathbb{R}^3}\big(x\cdot \nabla V+V\big)\cdot\varphi\,\mathrm{d}x
	= \int_{\mathbb{R}^3}P\, \mathrm{div}\,\varphi\,\mathrm{d}x-\int_{\mathbb{R}^3}\overline{V}\cdot \nabla \varphi\cdot V\,\mathrm{d}x-\int_{\mathbb{R}^3}f\cdot \varphi\,\mathrm{d}x.
	\end{equation}
	Then there exists a positive constant  $C$ such that
	\begin{equation*}
	\|V\|_{H^2(\mathbb{R}^3)}\leq C\left(\big\|\overline{V}\big\|^2_{H^1(\mathbb{R}^3)}	\|V\|_{H^1(\mathbb{R}^3)}+\|f\|_{L^2(\mathbb{R}^3)}\right).
	\end{equation*}
\end{lem}
\begin{proof}
	We denoted by $D^h_ku$   the difference quotient
	$$D_k^hu=\frac{u(x+h\mathbf{e}_k)-u(x)}{h}\qquad \text{with}\quad h\in\mathbb{R}\backslash\{0\}.$$
	We take
$\varphi(x)\triangleq-D_k^{-h}\big(D^h_kV\big),$ we compute
	\begin{align*}
-	\int_{\mathbb{R}^3}\nabla V:\nabla D_k^{-h}\big(D^h_k V\big)\,\mathrm{d}x
=&-	\int_{\mathbb{R}^3}\nabla V: D_k^{-h}\big(D^h_k \nabla V\big)\,\mathrm{d}x\\
=&\int_{\mathbb{R}^3}\big(D^h_k \nabla V\big): \big(D^h_k \nabla V\big)\,\mathrm{d}x=\big\|D^h_k \nabla V\big\|^2_{L^2(\mathbb{R}^3)}.
	\end{align*}
	Integrating by parts, we get
	\begin{align*}
&	\frac12\int_{\mathbb{R}^3}x\cdot \nabla V\cdot D_k^{-h}\big(D^h_kV\big)\,\mathrm{d}x+\frac12\int_{\mathbb{R}^3} V\cdot D_k^{-h}\big(D^h_kV\big)\,\mathrm{d}x\\
=&-	\frac12\int_{\mathbb{R}^3}D_k^{h}\big(x\cdot \nabla V\big)\cdot D^h_kV\,\mathrm{d}x-\frac12\int_{\mathbb{R}^3}D^h_kV\cdot D^h_kV\,\mathrm{d}x\\
=&-	\frac12\int_{\mathbb{R}^3}(x+h\mathbf{e}_k)\cdot \nabla D_k^{h} V\cdot D^h_kV\,\mathrm{d}x-	\frac12\int_{\mathbb{R}^3}\big(D_k^{h}x\big)\cdot \nabla V\cdot D^h_kV\,\mathrm{d}x-\frac12\int_{\mathbb{R}^3}D^h_kV\cdot D^h_kV\,\mathrm{d}x\\
=&-	\frac12\int_{\mathbb{R}^3}x\cdot \nabla D_k^{h} V\cdot D^h_kV\,\mathrm{d}x-	\frac12\int_{\mathbb{R}^3} \partial_{x_k}V\cdot D^h_kV\,\mathrm{d}x-\frac12\int_{\mathbb{R}^3}D^h_kV\cdot D^h_kV\,\mathrm{d}x\\
=&-\frac12\int_{\mathbb{R}^3} \partial_{x_k}V\cdot D^h_kV\,\mathrm{d}x+\big\|D^h_k  V\big\|^2_{L^2(\mathbb{R}^3)}.
	\end{align*}
Next, we turn to deal with the term involving pressure. Integrating by parts gives
\begin{align*}
-\int_{\mathbb{R}^3}P \text{div}\,D_k^{-h}\big(D^h_kV\big)\,\mathrm{d}x=&
\int_{\mathbb{R}^3}D^h_kP\, \text{div}\,D^h_kV\,\mathrm{d}x=0.
\end{align*}	
 According to the fact that $\text{div}\,V=0$, we  obtain
\begin{align*}
-\int_{\mathbb{R}^3}\overline{V}\cdot \nabla D_k^{-h}\big(D^h_kV\big)\cdot V\,\mathrm{d}x=&\int_{\mathbb{R}^3}\overline V(x+h\mathbf{e}_k)\cdot \nabla D^h_kV(x)\cdot D^h_kV(x)\,\mathrm{d}x\\
&+\int_{\mathbb{R}^3}D^h_k\overline{V}(x)\cdot \nabla D^h_kV(x)\cdot V(x)\,\mathrm{d}x\\
=&\int_{\mathbb{R}^3}D^h_k\overline{V}(x)\cdot \nabla D^h_kV(x)\cdot V(x)\,\mathrm{d}x.
\end{align*}
 So, we have from \eqref{eq.weakLNS} that
\begin{equation*}
\begin{split}
\|D^h_k\nabla V\|_{L^2(\mathbb{R}^3)}^2+\frac12\|D^h_kV\|_{L^2(\mathbb{R}^3)}^2=&\int_{\mathbb{R}^3}f(x)D^{-h}_kD^h_kV(x)\,\mathrm{d}x+\frac12\int_{\mathbb{R}^3} \partial_{x_k}V\cdot D^h_kV\,\mathrm{d}x\\
&+\int_{\mathbb{R}^3}D^h_k\overline{V}(x)\cdot \nabla D^h_kV(x)\cdot V(x)\,\mathrm{d}x.
\end{split}
\end{equation*}	
Also, we see that
\begin{align*}
\int_{\mathbb{R}^3}D_k^{-h}\big(D^h_kV\big)\cdot D_k^{-h}\big(D^h_kV\big)\,\mathrm{d}x=&\int_{\mathbb{R}^3}D_k^{h}\big(D^h_kV\big)\cdot D_k^{h}\big(D^h_kV\big)\,\mathrm{d}x
\leq  C\|\nabla D^h_kV\|^2_{L^2(\mathbb{R}^3)}.
\end{align*}	
This estimate enables us to conclude that
\begin{align*}
\int_{\mathbb{R}^3}f(x)D^{-h}_kD^h_kV(x)\,\mathrm{d}x\leq &\|f\|_{L^2(\mathbb{R}^3)}\big\|D^{-h}_kD^h_kV\big\|_{L^2(\mathbb{R}^3)}\\
\leq &C\|f\|_{L^2(\mathbb{R}^3)}\big\|\nabla D^h_kV\big\|_{L^2(\mathbb{R}^3)}\\
\leq &C\|f\|^2_{L^2(\mathbb{R}^3)}+\frac{1}{16}\big\|\nabla D^h_kV\big\|^2_{L^2(\mathbb{R}^3)}.
\end{align*}
By the H\"older inequality, one has
\begin{equation*}
\frac12\int_{\mathbb{R}^3} \partial_{x_k}V\cdot D^h_kV\,\mathrm{d}x\leq\|\partial_{x_k}V\|_{L^2(\mathbb{R}^3)}\|D^h_kV\|_{L^2(\mathbb{R}^3)}\leq C\|\nabla V\|^2_{L^2(\mathbb{R}^3)}.
\end{equation*}
By the interpolation inequality, we see that
\begin{align*}
\int_{\mathbb{R}^3}D^h_k\overline{V}(x)\cdot \nabla D^h_kV(x)\cdot V(x)\,\mathrm{d}x
\leq&\|\overline{V}\|_{L^6(\mathbb{R}^3)}\|D^h_kV\|_{L^3(\mathbb{R}^3)}\|D^h_k\nabla V\|_{L^2(\mathbb{R}^3)}\\
\leq&C\|\overline{V}\|_{L^6(\mathbb{R}^3)}\|D^h_kV\|^\frac12_{L^2(\mathbb{R}^3)}\|D^h_k\nabla V\|^{\frac{3}{2}}_{L^2(\mathbb{R}^3)}\\
\leq&C\|\nabla \overline{V}\|^4_{L^2(\mathbb{R}^3)}\|\nabla  V \|^2_{L^2(\mathbb{R}^3)}+\frac{1}{16}\|D^h_k\nabla V\|^2_{L^2(\mathbb{R}^3)}.
\end{align*}
Collecting all estimates yields
\[\|D^h_k\nabla V\|_{L^2(\mathbb{R}^3)}^2+\|D^h_kV\|_{L^2(\mathbb{R}^3)}^2\leq C\big\|\nabla \overline{V}\big\|^4_{L^2(\mathbb{R}^3)}\|\nabla  V \|^2_{L^2(\mathbb{R}^3)}+C\|f\|^2_{L^2(\mathbb{R}^3)}.\]
Taking $h\to0$ in the above inequality, we readily have
\[\|\nabla^2 V\|_{L^2(\mathbb{R}^3)}^2+\|\nabla V\|_{L^2(\mathbb{R}^3)}^2\leq C\big\|\nabla \overline{V}\big\|^4_{L^2(\mathbb{R}^3)}\|\nabla  V \|^2_{L^2(\mathbb{R}^3)}+C\|f\|^2_{L^2(\mathbb{R}^3)}.\]
This completes the proof of the lemma.
\end{proof}
According to Lemma \ref{RNS}, we will show the $H^2(\R^3)$-estimate for $V$ and the $\dot{H}^1(\mathbb{R}^3)$-estimate for $|x|P$.
\begin{prop}\label{coro-h2}
	Let the couple $(V,P)\in H^1(\R^3)\times L^2(\R^3)$  satisfying  \eqref{eq.weak}.
	Then we have
	\[V(x)\in H^2(\mathbb{R}^3)\]
	and
	\begin{equation*}\label{eq.EQ}
	Q(x)\triangleq|x|P\in \dot{H}^1(\mathbb{R}^3).
	\end{equation*}
\end{prop}
\begin{proof}
By Proposition \ref{prop-xv}, we have that $W\in H^1(\R^3)$, moreover, we obtain  by Lemma \ref{RNS} that $W\in H^2(\R^2)$.

Thanks to \eqref{eq.P}, we see that $Q=|x|P$ satisfies
\begin{align*}
-\Delta Q=|x|\mathrm{div}\,\big(V\cdot\nabla V+U_{0}\cdot\nabla V+V\cdot\nabla U_{0}+U_{0}\cdot\nabla U_{0}\big)-2\frac{x}{|x|}\cdot\nabla P-\frac{2}{|x|}P.
\end{align*}
By the Cauchy-Schwarz inequality and the Hardy inequality, we get that for any $\varphi\in C_{0}^{\infty}(\R^{3})$
\begin{align*}
\Big\langle\frac{2}{|x|}P,\varphi\Big\rangle&\leq C \|P\|_{L^{2}(\R^{3})}\|\nabla \varphi\|_{L^{2}(\R^{3})}
\end{align*}
and
\begin{align*}
\Big\langle\frac{x}{|x|}\cdot\nabla P,\varphi\Big\rangle&=-\int_{\R^{3}}P\partial_{i}\Big(\frac{x_{i}}{|x|}\varphi\Big)\,{\rm d}x\\
&\leq C \|P\|_{L^{2}(\R^{3})}\|\nabla \varphi\|_{L^{2}(\R^{3})}
+\|P\|_{L^{2}(\R^{3})}\Big\|\frac{\varphi}{|x|}\Big\|_{L^{2}(\R^{3})}\leq C \|P\|_{L^{2}(\R^{3})}\|\nabla \varphi\|_{L^{2}(\R^{3})}.
\end{align*}
Similarly, one has that  for any  $\varphi\in C_{0}^{\infty}(\R^{3})$,
\begin{align*}
&\left\langle|x|\mathrm{div}\,\big(V\cdot\nabla V+U_{0}\cdot\nabla V+V\cdot\nabla U_{0}+U_{0}\cdot\nabla U_{0}\big), \varphi\right\rangle\\=&
 -\int_{\R^{3}}\big(V\cdot\nabla V+U_{0}\cdot\nabla V+V\cdot\nabla U_{0}+U_{0}\cdot\nabla U_{0}\big)\nabla(|x|\varphi) \,{\rm d}x\\
 \leq &\int_{\R^{3}}\Big(\frac{|\varphi|}{|x|}+|\nabla\varphi|\Big)\Big\{|x||V\cdot\nabla V|
 +|x||U_{0}\cdot\nabla V|+|x||V\cdot\nabla U_{0}|+|x||U_{0}\cdot\nabla U_{0}|\Big\} \,{\rm d}x\\
 \leq&C\Big(\|W\|_{L^4(\R^3)}\big(\|\nabla V\|_{L^4(\R^3)}+\|\nabla U_0\|_{L^4(\R^3)}\big)+\big\||\cdot|U_0(\cdot)\big\|_{L^\infty(\R^3)}\big(\|\nabla V\|_{L^2(\R^3)}\\
 &\,\,\,\,\,\,\,+\|\nabla U_0\|_{L^2(\R^3)}\big)\Big)
  \leq  C(V, U_{0})\|\nabla\varphi\|_{L^{2}(\R^{3})}.
\end{align*}
Combining these results and using the density argument yield the required estimate in Proposition \ref{RNS}.
\end{proof}
With these regularity estimates in hand, we are going to show  $H^2$-estimate for $|x|V$ which implies that $|x|V(x)$ is bounded.
\begin{prop}\label{prop-w-2}
Let the couple $(V,P)\in H^1(\R^3)\times L^2(\R^3)$  satisfying \eqref{eq.weak}.
Then we have
\[W(x)=|x|V(x)\in H^2(\mathbb{R}^3).\]
Moreover,
\[\big\|\langle \cdot\rangle V(\cdot)\big\|_{\dot{B}^0_{\infty,1}(\mathbb{R}^3)}<+\infty.\]
\end{prop}
\begin{proof}
Taking $\varphi(x)\triangleq-D^{-h}_kh_\varepsilon^2D^h_kV$  in equality \eqref{eq.weak}, we immediately have
\begin{equation*}\label{eq.wWI}
\begin{split}
&\int_{\mathbb{R}^3}\nabla V:\nabla\big(-D^{-h}_kh_\varepsilon^2D^h_k V\big)\,\mathrm{d}x-\frac12\int_{\mathbb{R}^3}x\cdot \nabla V\cdot\big(-D^{-h}_kh_\varepsilon^2D^h_k V\big)\,\mathrm{d}x\\
&-\frac12\int_{\mathbb{R}^3} V\cdot\big(-D^{-h}_kh_\varepsilon^2D^h_k V\big)\,\mathrm{d}x-\int_{\mathbb{R}^3}P\, \mathrm{div}\,\big(-D^{-h}_kh_\varepsilon^2D^h_kV\big)\,\mathrm{d}x\\
=&\int_{\mathbb{R}^3}V\cdot \nabla \big(-D^{-h}_kh_\varepsilon^2D^h_k V\big)\cdot V\,\mathrm{d}x-\int_{\mathbb{R}^3}U_0\cdot \nabla V\cdot \big(-D^{-h}_kh_\varepsilon^2D^h_k V\big)\,\mathrm{d}x\\
&-\int_{\mathbb{R}^3}V\cdot \nabla U_0\cdot \big(-D^{-h}_kh_\varepsilon^2D^h_k V\big)\,\mathrm{d}x-\int_{\mathbb{R}^3}U_0\cdot \nabla U_0\cdot  \big(-D^{-h}_kh_\varepsilon^2D^h_k V\big) \,\mathrm{d}x.
\end{split}
\end{equation*}
Some calculations yield
\begin{align*}
&\int_{\mathbb{R}^3}\nabla V:\nabla\big(-D^{-h}_kh_\varepsilon^2D^h_k V\big)\,\mathrm{d}x\\
=&\int_{\mathbb{R}^3}D^h_k\nabla V:\nabla\big(h_\varepsilon^2D^h_k V\big)\,\mathrm{d}x\\
=&\big\|h_\varepsilon D^h_k\nabla V\big\|_{L^2(\mathbb{R}^3)}^2+2\int_{\mathbb{R}^3}h_\varepsilon D^h_k\nabla V:\bigg(\frac{x}{|x|(1+\varepsilon|x|^2)^{\frac34}} \otimes D^h_k V\bigg)\,\mathrm{d}x\\
&-3\int_{\mathbb{R}^3}h_\varepsilon D^h_k\nabla V:\bigg(\frac{\varepsilon |x|x}{(1+\varepsilon|x|^2)^{\frac74}} \otimes D^h_k V\bigg)\,\mathrm{d}x.
\end{align*}
By the H\"older inequality and the Young inequality, one has
\begin{align*}
&-2\int_{\mathbb{R}^3}h_\varepsilon D^h_k\nabla V:\bigg(\frac{x}{|x|(1+\varepsilon|x|^2)^{\frac34}} \otimes D^h_k V\bigg)\,\mathrm{d}x
\\&+3\int_{\mathbb{R}^3}h_\varepsilon D^h_k\nabla V:\bigg(\frac{\varepsilon |x|x}{(1+\varepsilon|x|^2)^\frac{7}{4}} \otimes D^h_k V\bigg)\,\mathrm{d}x\\
\leq&C\|\nabla V\|_{L^2(\mathbb{R}^3)}\big\|h_\varepsilon D^h_k\nabla V\big\|_{L^2(\mathbb{R}^3)}\leq C\|\nabla V\|^2_{L^2(\mathbb{R}^3)}+\frac{1}{64}\big\|h_\varepsilon D^h_k\nabla V\big\|_{L^2(\mathbb{R}^3)}^2.
\end{align*}
By the fact that $\mathrm{div}\,V=0,$ we obtain
\begin{align*}
&-\frac12\int_{\mathbb{R}^3}x\cdot \nabla V\cdot\big(-D^{-h}_kh_\varepsilon^2D^h_k V\big)\,\mathrm{d}x-\frac12\int_{\mathbb{R}^3} V\cdot\big(-D^{-h}_kh_\varepsilon^2D^h_k V\big)\,\mathrm{d}x\\
=&-\frac12\int_{\mathbb{R}^3}\big(x+h\mathbf{e}_k\big) \nabla D^h_kV h_\varepsilon^2D^h_k V\,\mathrm{d}x-\frac12\int_{\mathbb{R}^3} \partial_{x_k} V\cdot\big(h_\varepsilon^2D^h_k V\big)\,\mathrm{d}x-\frac12\int_{\mathbb{R}^3} h_\varepsilon^2D^h_kV D^h_kV\,\mathrm{d}x\\
=&\frac12\int_{\mathbb{R}^3}(x\cdot \nabla h_\varepsilon)\,D^h_kV  \big(h_\varepsilon D^h_kV\big)\,\mathrm{d}x-\frac{h}{2}\int_{\mathbb{R}^3}  \partial_{x_k}D^h_kV h_\varepsilon^2D^h_kV\,\mathrm{d}x-\frac12\int_{\mathbb{R}^3} \partial_{x_k}D^h_k V \big(h_\varepsilon^2D^h_k V\big)\,\mathrm{d}x\\
&+\frac14\int_{\mathbb{R}^3} h_\varepsilon^2D^h_kV\cdot D^h_k V\,\mathrm{d}x\\
=&\frac12\int_{\mathbb{R}^3}(x\cdot \nabla h_\varepsilon)\,D^h_kV \big(h_\varepsilon D^h_kV\big)\,\mathrm{d}x+\frac{(1+h)}{4}\int_{\mathbb{R}^3}  \big(\partial_{x_k}h_\varepsilon^2\big)D^h_kV D^h_k V\,\mathrm{d}x+\frac14\int_{\mathbb{R}^3} h_\varepsilon^2D^h_kV D^h_k V\,\mathrm{d}x.
\end{align*}
We see that
\begin{align*}
&\frac12\int_{\mathbb{R}^3}(x\cdot \nabla h_\varepsilon)\,D^h_kV\cdot \big(h_\varepsilon D^h_kV\big)\,\mathrm{d}x\\
=&\frac12\int_{\mathbb{R}^3} \frac{|x|}{(1+\varepsilon|x|^2)^{\frac34} }\,(h_\varepsilon D^h_kV)\cdot D^h_kV\,\mathrm{d}x-\frac34\int_{\mathbb{R}^3} \frac{\varepsilon|x|^3}{(1+\varepsilon|x|^2)^{\frac74} }\,(h_\varepsilon D^h_kV)\cdot D^h_kV\,\mathrm{d}x\\
=&\frac12\int_{\mathbb{R}^3} h^2_\varepsilon D^h_kV\cdot D^h_kV\,\mathrm{d}x-\frac34\int_{\mathbb{R}^3} \frac{\varepsilon|x|^2}{1+\varepsilon|x|^2 }h^2_\varepsilon \,D^h_kV\cdot D^h_kV\,\mathrm{d}x.
\end{align*}
Thus, we have
\begin{align*}
&-\frac12\int_{\mathbb{R}^3}x\cdot \nabla V\cdot\big(-D^{-h}_kh_\varepsilon^2D^h_k V\big)\,\mathrm{d}x-\frac12\int_{\mathbb{R}^3} V\cdot\big(-D^{-h}_kh_\varepsilon^2D^h_k V\big)\,\mathrm{d}x\\
=&\frac34\int_{\mathbb{R}^3} \frac{\varepsilon|x|^2}{1+\varepsilon|x|^2 }h^2_\varepsilon \,D^h_kV\cdot D^h_kV\,\mathrm{d}x+\frac{(1+h)}{4}\int_{\mathbb{R}^3}  \big(\partial_{x_k}h_\varepsilon^2\big)D^h_kV\cdot D^h_k V\,\mathrm{d}x.
\end{align*}
We calculate
\begin{align*}
\frac{1+h}{4}\int_{\mathbb{R}^3}  \big(\partial_{x_k}h_\varepsilon^2\big)D^h_kV\cdot D^h_k V\,\mathrm{d}x=&\frac{1+h}{2}\int_{\mathbb{R}^3}  \frac{x_k}{|x|(1+\varepsilon|x|^2)^{\frac34}}\big(h_\varepsilon D^h_kV\big)\cdot D^h_k V\,\mathrm{d}x\\
&-\frac{3(1+h)}{4}\int_{\mathbb{R}^3}  \frac{\varepsilon|x|x_k}{(1+\varepsilon|x|^2)^{\frac74}}\big(h_\varepsilon D^h_kV\big)\cdot D^h_k V\,\mathrm{d}x.
\end{align*}
Since $|h|\leq1$, we have by the H\"older inequality and the Young inequality that
\begin{align*}
&-\frac{(1+h)}{2}\int_{\mathbb{R}^3}  \frac{x_k}{|x|(1+\varepsilon|x|^2)^{\frac34}}\big(h_\varepsilon D^h_kV\big)\cdot D^h_k V\,\mathrm{d}x\\
\leq&C\|\nabla V\|_{L^2(\mathbb{R}^3)}\big\|g_\varepsilon\big(h_\varepsilon D^h_kV\big)\big\|_{L^2(\mathbb{R}^3)}
\leq C\|\nabla V\|^2_{L^2(\mathbb{R}^3)}+\frac{1}{64}\big\|g_\varepsilon\big(h_\varepsilon D^h_kV\big)\big\|^2_{L^2(\mathbb{R}^3)}
\end{align*}
and
\begin{align*}
&\frac{3(1+h)}{4}\int_{\mathbb{R}^3}  \frac{\varepsilon|x|x_k}{(1+\varepsilon|x|^2)^{\frac74}}\big(h_\varepsilon D^h_kV\big)\cdot D^h_k V\,\mathrm{d}x\\
\leq&C\|\nabla V\|_{L^2(\mathbb{R}^3)}\big\|g_\varepsilon\big(h_\varepsilon D^h_kV\big)\big\|_{L^2(\mathbb{R}^3)}
\leq  C\|\nabla V\|^2_{L^2(\mathbb{R}^3)}+\frac{1}{64}\big\|g_\varepsilon\big(h_\varepsilon D^h_kV\big)\big\|^2_{L^2(\mathbb{R}^3)}.
\end{align*}
Therefore, we have
\begin{align*}
 -\frac{(1+h)}{4}\int_{\mathbb{R}^3}  \big(\partial_{x_k}h_\varepsilon^2\big)D^h_kV\cdot D^h_k V\,\mathrm{d}x
\leq C\|\nabla V\|^2_{L^2(\mathbb{R}^3)}+\frac{1}{32}\big\|g_\varepsilon\big(h_\varepsilon D^h_kV\big)\big\|^2_{L^2(\mathbb{R}^3)}.
\end{align*}
A simple calculation yields
\begin{align*}
\int_{\mathbb{R}^3}P\, \mathrm{div}\,\big(-D^{-h}_kh_\varepsilon^2D^h_kV\big)\,\mathrm{d}x=&\int_{\mathbb{R}^3}D^{h}_kP\, \mathrm{div}\,\big(h_\varepsilon^2D^h_kV\big)\,\mathrm{d}x\\
=&\int_{\mathbb{R}^3}D^{h}_kP\, \nabla h^2_\varepsilon\cdot\, D^h_kV\,\mathrm{d}x\\
=&2\int_{\mathbb{R}^3}D^{h}_kP\, \frac{x}{|x|(1+\varepsilon|x|^2)^{\frac34}}\cdot\,\big(h_\varepsilon D^h_kV\big)\,\mathrm{d}x\\
=&2\int_{\mathbb{R}^3}D^{h}_kP\, \frac{x}{|x|(1+\varepsilon|x|^2)^{\frac34}}\cdot\,\big(h_\varepsilon D^h_kV\big)\,\mathrm{d}x\\
&-3\int_{\mathbb{R}^3}D^{h}_kP\, \frac{\varepsilon|x|x}{(1+\varepsilon|x|^2)^{\frac74}}\cdot\,\big(h_\varepsilon D^h_kV\big)\,\mathrm{d}x.
\end{align*}
By the H\"older inequality and the Young inequality, one has
\begin{align*}
&2\int_{\mathbb{R}^3}D^{h}_kP\, \frac{x}{|x|(1+\varepsilon|x|^2)^{\frac34}}\cdot\,\big(h_\varepsilon D^h_kV\big)\,\mathrm{d}x-3\int_{\mathbb{R}^3}D^{h}_kP\, \frac{\varepsilon|x|x}{(1+\varepsilon|x|^2)^{\frac74}}\cdot\,\big(h_\varepsilon D^h_kV\big)\,\mathrm{d}x\\
\leq&2\big\||\cdot|D^h_kP\big\|_{L^2(\mathbb{R}^3)}\|D^h_kV\|_{L^2(\mathbb{R}^3)}+3\big\||\cdot|D^h_kP\big\|_{L^2(\mathbb{R}^3)}\|D^h_kV\|_{L^2(\mathbb{R}^3)}
\leq C\|Q\|_{L^2(\mathbb{R}^3)}\|\nabla V\|_{L^2(\mathbb{R}^3)}.
\end{align*}
For the convection term, we get by integration by parts that
\begin{align*}
&\int_{\mathbb{R}^3}V\cdot \nabla \big(-D^{-h}_kh_\varepsilon^2D^h_k V\big)\cdot V\,\mathrm{d}x\\
=&\int_{\mathbb{R}^3}V(x+h\mathbf{e}_k)\cdot \nabla \big(h_\varepsilon^2D^h_k V\big)\cdot D^h_kV\,\mathrm{d}x+\int_{\mathbb{R}^3}D^h_kV\cdot \nabla \big(h_\varepsilon^2D^h_k V\big)\cdot V\,\mathrm{d}x\\
=&\int_{\mathbb{R}^3}V(x+h\mathbf{e}_k)\cdot \nabla h_\varepsilon\, \big(h_\varepsilon D^h_k V\big)\cdot D^h_kV\,\mathrm{d}x+\int_{\mathbb{R}^3}D^h_kV\cdot \nabla h_\varepsilon\, \big(h_\varepsilon D^h_k V\big)\cdot V\,\mathrm{d}x\\
&+\int_{\mathbb{R}^3}h_\varepsilon D^h_kV\cdot \nabla \big(h_\varepsilon D^h_k V\big)\cdot V\,\mathrm{d}x.
\end{align*}
By the H\"older inequality, we find that
\begin{align*}
&\int_{\mathbb{R}^3}V(x+h\mathbf{e}_k)\cdot \nabla h_\varepsilon\, \big(h_\varepsilon D^h_k V\big)\cdot D^h_kV\,\mathrm{d}x\\
=&\int_{\mathbb{R}^3}V(x+h\mathbf{e}_k)\cdot \frac{x}{|x|(1+\varepsilon|x|^2)^{\frac34}}\, \big(h_\varepsilon D^h_k V\big)\cdot D^h_kV\,\mathrm{d}x\\
&-\frac{3}{2}\int_{\mathbb{R}^3}V(x+h\mathbf{e}_k)\cdot \frac{\varepsilon|x|x}{(1+\varepsilon|x|^2)^{\frac74}}\, \big(h_\varepsilon D^h_k V\big)\cdot D^h_kV\,\mathrm{d}x\\
\leq&C\|V\|_{L^\infty(\mathbb{R}^3)}\|\nabla V\|_{L^2(\mathbb{R}^3)}\big\|g_\varepsilon\big(h_\varepsilon D^h_k V\big)\big\|_{L^2(\mathbb{R}^3)}\\
\leq&C\|V\|^2_{L^\infty(\mathbb{R}^3)}\|\nabla V\|^2_{L^2(\mathbb{R}^3)}+\frac{1}{36}\big\|g_\varepsilon\big(h_\varepsilon D^h_k V\big)\big\|^2_{L^2(\mathbb{R}^3)}.
\end{align*}
Similarly, we have
\begin{align*}
\int_{\mathbb{R}^3}D^h_kV\cdot \nabla h_\varepsilon\, \big(h_\varepsilon D^h_k V\big)\cdot V\,\mathrm{d}x
\leq&C\|V\|^2_{L^\infty(\mathbb{R}^3)}\|\nabla V\|^2_{L^2(\mathbb{R}^3)}+\frac{1}{36}\big\|g_\varepsilon\big(h_\varepsilon D^h_k V\big)\big\|^2_{L^2(\mathbb{R}^3)}.
\end{align*}
A simple calculation yields
\begin{align*}
&\int_{\mathbb{R}^3}h_\varepsilon D^h_kV\cdot \nabla \big(h_\varepsilon D^h_k V\big)\cdot V\,\mathrm{d}x\\
=&\int_{\mathbb{R}^3}h^2_\varepsilon D^h_kV\cdot \nabla \big( D^h_k V\big)\cdot V\,\mathrm{d}x+\int_{\mathbb{R}^3}h_\varepsilon D^h_kV\cdot \nabla \big(h_\varepsilon \big)\,D^h_k V\cdot V\,\mathrm{d}x.
\end{align*}
On one hand,
\begin{align*}
\int_{\mathbb{R}^3}h^2_\varepsilon D^h_kV\cdot \nabla \big( D^h_k V\big)\cdot V\,\mathrm{d}x\leq &\|V\|_{L^\infty(\mathbb{R}^3)}\|\nabla V\|_{L^2(\mathbb{R}^3)}\big\|h_\varepsilon \nabla D^h_k V\big\|_{L^2(\mathbb{R}^3)}\\
\leq &\|V\|^2_{L^\infty(\mathbb{R}^3)}\|\nabla V\|^2_{L^2(\mathbb{R}^3)}+\frac{1}{36}\big\|h_\varepsilon \nabla D^h_k V\big\|^2_{L^2(\mathbb{R}^3)}.
\end{align*}
On the other hand,
\begin{align*}
&\int_{\mathbb{R}^3}h_\varepsilon D^h_kV\cdot \nabla \big(h_\varepsilon \big)\,D^h_k V\cdot V\,\mathrm{d}x\\
=&\int_{\mathbb{R}^3}\big(h_\varepsilon D^h_kV\big)\cdot  \frac{x}{|x|(1+\varepsilon|x|^2)^{\frac34}}\,D^h_k V\cdot V\,\mathrm{d}x-\frac32\int_{\mathbb{R}^3}\big(h_\varepsilon D^h_kV\big)\cdot  \frac{\varepsilon|x|x}{(1+\varepsilon|x|^2)^{\frac74}}\,D^h_k V\cdot V\,\mathrm{d}x\\
\leq &C\|V\|_{L^\infty(\mathbb{R}^3)}\|\nabla V\|_{L^2(\mathbb{R}^3)}\big\|g_\varepsilon\big(h_\varepsilon D^h_k V\big)\big\|_{L^2(\mathbb{R}^3)}\\
\leq &C\|V\|^2_{L^\infty(\mathbb{R}^3)}\|\nabla V\|^2_{L^2(\mathbb{R}^3)}+\frac{1}{36}\big\|g_\varepsilon\big(h_\varepsilon D^h_k V\big)\big\|^2_{L^2(\mathbb{R}^3)}.
\end{align*}
By the H\"older inequality and the Young inequality, one has
\begin{align*}
&-\int_{\mathbb{R}^3}U_0\cdot \nabla V\cdot \big(-D^{-h}_kh_\varepsilon^2D^h_k V\big)\,\mathrm{d}x\\
=&-\int_{\mathbb{R}^3}D^h_kU_0\cdot \nabla V(x+h\mathbf{e}_k)\cdot \big(h_\varepsilon^2D^h_k V\big)\,\mathrm{d}x-\int_{\mathbb{R}^3}U_0\cdot \nabla D^h_kV\cdot \big(h_\varepsilon^2D^h_k V\big)\,\mathrm{d}x\\
\leq &C\left(\big\||\cdot|\nabla U_0(\cdot)\big\|_{L^\infty(\mathbb{R}^3)}\|\nabla V\|_{L^2(\mathbb{R}^3)}+\big\||\cdot| U_0(\cdot)\big\|_{L^\infty(\mathbb{R}^3)}\|\nabla^2 V\|_{L^2(\mathbb{R}^3)}\right)\big\|g_\varepsilon\big(h_\varepsilon D^h_k V\big)\big\|_{L^2(\mathbb{R}^3)}\\
\leq &C\big\|\langle\cdot\rangle\nabla U_0(\cdot)\big\|^2_{L^\infty(\mathbb{R}^3)}\|\nabla V\|^2_{L^2(\mathbb{R}^3)}+C\big\|\langle\cdot\rangle U_0(\cdot)\big\|^2_{L^\infty(\mathbb{R}^3)}\|\nabla^2 V\|^2_{L^2(\mathbb{R}^3)}+\frac{1}{36}\big\|g_\varepsilon\big(h_\varepsilon D^h_k V\big)\big\|^2_{L^2(\mathbb{R}^3)}.
\end{align*}
We observe that
\begin{align*}
&-\int_{\mathbb{R}^3}V\cdot \nabla U_0\cdot \big(-D^{-h}_kh_\varepsilon^2D^h_k V\big)\,\mathrm{d}x\\
=&-\int_{\mathbb{R}^3}V(x+h\mathbf{e}_k)\cdot \nabla D^h_kU_0\cdot \big(h_\varepsilon^2D^h_k V\big)\,\mathrm{d}x-\int_{\mathbb{R}^3}D^h_kV\cdot \nabla U_0\cdot \big(h_\varepsilon^2D^h_k V\big)\,\mathrm{d}x\\
\leq &C\big\|\langle\cdot\rangle^2\nabla^2 U_0(\cdot)\big\|_{L^\infty(\mathbb{R}^3)}\|\nabla V\|_{L^2(\mathbb{R}^3)}\big\|V\big\|_{L^2(\mathbb{R}^3)}+C\big\|\langle\cdot\rangle^2 \nabla U_0(\cdot)\big\|_{L^\infty(\mathbb{R}^3)}\|\nabla V\|^2_{L^2(\mathbb{R}^3)}.
\end{align*}
Similarly, we can show that
\begin{align*}
&-\int_{\mathbb{R}^3}U_0\cdot \nabla U_0\cdot \big(-D^{-h}_kh_\varepsilon^2D^h_k V\big)\,\mathrm{d}x\\
=&-\int_{\mathbb{R}^3}U_0(x+h\mathbf{e}_k)\cdot \nabla D^h_kU_0\cdot \big(h_\varepsilon^2D^h_k V\big)\,\mathrm{d}x-\int_{\mathbb{R}^3}D^h_kU_0\cdot \nabla U_0\cdot \big(h_\varepsilon^2D^h_k V\big)\,\mathrm{d}x\\
\leq &C\big\|\langle\cdot\rangle^2\nabla^2 U_0(\cdot)\big\|_{L^\infty(\mathbb{R}^3)}\big(\|\nabla V\|_{L^2(\mathbb{R}^3)}\big\|U_0\big\|_{L^2(\mathbb{R}^3)}+ \|\nabla U_0\|_{L^2(\mathbb{R}^3)}\|\nabla V\|_{L^2(\mathbb{R}^3)}\big).
\end{align*}
Collecting all estimates implies
\begin{align*}
\big\|h_\varepsilon\nabla D^h_kV \big\|^2_{L^2(\mathbb{R}^3)}+\big\|g_\varepsilon\big(h_\varepsilon D^h_kV\big) \big\|^2_{L^2(\mathbb{R}^3)}\leq C\big(U_0\big).
\end{align*}
Taking $h\to0$ entails
\[\|\Delta W\|^2_{L^2(\mathbb{R}^3)}+\|W\|^2_{L^2(\mathbb{R}^3)}<+\infty.\]
This estimate together with the embedding theorem that $H^2(\mathbb{R}^3)\hookrightarrow \dot{B}^0_{\infty,1}( \mathbb{R}^3)$ entails
\[\big\||\cdot|V(\cdot)\big\|_{\dot{B}^0_{\infty,1}(\mathbb{R}^3)}<+\infty.\]
This combined with the fact that  $\|V\|_{\dot{B}^0_{\infty,1}(\mathbb{R}^3)}<+\infty$ enables us to conclude the desired result in the proposition.
\end{proof}
Next, we will further improve the regularity for the couple $(V,\,P)$ by using the bootstrapping argument.
\begin{prop}\label{prop-nv}
Let the couple $(V,P)\in H^1(\R^3)\times L^2(\R^3)$  satisfying  \eqref{eq.weak}.
Then we have $V\in H^3(\R^3)$ and $E\triangleq|x|\nabla V\in H^2(\R^3)$.
\end{prop}
\begin{proof}
By Theorem \ref{P3.12}, Proposition \ref{coro-h2} and Proposition \ref{prop-w-2}, we know that $V$ solves
\[-\Delta V=\frac{1}{2}x\cdot \nabla V+\frac12V-\mathbb{P}\big(V\cdot\nabla V-U_0\cdot\nabla V-V\cdot\nabla U_0-U_0\cdot\nabla U_0\big).\]
By the H\"older inequality, one has
\begin{equation*}
\big\||\cdot|\nabla V\big\|_{\dot{H}^1(R^3)}\leq\|\nabla W\|_{\dot{H}^1(\R^3)}+\|V\|_{\dot{H}^1(\R^3)}\leq\| W\|_{\dot{H}^2(\R^3)}+\|V\|_{\dot{H}^1(\R^3)}.
\end{equation*}
Similarly, we have
\begin{equation*}
\|V\cdot\nabla V\|_{\dot{H}^1(\R^3)}\leq\|\nabla V\|^2_{L^4(\R^3)}+C\|V\|_{L^\infty(\R^3)}\|V\|_{\dot{H}^2(\R^3)}
\end{equation*}
and
\begin{equation*}
\begin{split}
&\|V\cdot\nabla U_0\|_{\dot{H}^1(\R^3)}+\|U_0\cdot\nabla V\|_{\dot{H}^1(\R^3)}\\
\leq&2\|\nabla V\|_{L^4(\R^3)}\|\nabla U_0\|_{L^4(\R^3)}+C\|V\|_{L^\infty(\R^3)}\|U_0\|_{\dot{H}^2(\R^3)}+C\|U_0\|_{L^\infty(\R^3)}\|V\|_{\dot{H}^2(\R^3)}.
\end{split}
\end{equation*}
By  the H\"older inequality and the Young inequality again, we obtain
\begin{equation*}
\|U_0\cdot\nabla U_0\|_{\dot{H}^1(\R^3)}\leq\|\nabla U_0\|^2_{L^4(\R^3)}+C\|U_0\|_{L^\infty(\R^3)}\|U_0\|_{\dot{H}^2(\R^3)}.
\end{equation*}
By the elliptic regularity theory, we immediately obtain
\begin{equation*}
\|V\|_{\dot{H}^3(\R^3)}<+\infty.
\end{equation*}
Setting $E_k\triangleq|x|\partial_{x_k}V$ and $P_k\triangleq|x|\partial_{x_k}P,$ we immediately find that
\begin{equation}
\begin{split}
&-\Delta E_k-V\cdot\nabla E_k-\frac12x\cdot\nabla E_k-\frac12 E_k+\nabla P_k\\
=&-|x|\partial_{x_k}\big(V\cdot\nabla U_0+U_0\cdot\nabla V+U_0\cdot\nabla U_0\big)-E_k\cdot\nabla V-V\cdot \frac{x}{|x|}\partial_{x_k}V-E_k\\&-2\frac{x}{|x|}\cdot\nabla \partial_{x_k}V-\frac{2}{|x|}\partial_{x_k}V+\frac{x}{|x|}\partial_{x_k}P.
\end{split}
\end{equation}
We see that
\begin{equation*}
\|E_k\|_{L^2(\R^3)}\leq\|W\|_{L^2(\R^3)}+\|V\|_{L^2(\R^3)}
\end{equation*}
and
\begin{equation*}
\Big\|\frac{x}{|x|}\cdot\nabla \partial_{x_k}V\Big\|_{L^2(\R^3)}\leq C\|V\|_{\dot{H}^2(\R^3)}.
\end{equation*}
By the Hardy inequality, one has
\begin{equation*}
\left\|\big({1}/{|\cdot|}\big)\partial_{x_k}V\right\|_{L^2(\R^3)}\leq C\|V\|_{\dot{H}^2(\R^3)}.
\end{equation*}
With the help of the H\"older inequality, we obtain
\begin{equation*}
\big\|V\cdot\frac{x}{|x|}\partial_{x_k}V\big\|_{L^2(\R^3)}\leq C \|W\|_{L^\infty(\R^3)}\|V\|_{\dot H^2(\R^3)}
\end{equation*}
and
\begin{align*}
&\big\||\cdot|\partial_{x_k}\big(V\cdot\nabla U_0+U_0\cdot\nabla V+U_0\cdot\nabla U_0\big)\big\|_{L^2(\R^3)}\\
\leq&C\big\||\cdot|V\big\|_{L^\infty(\R^3)}\|U_0\|_{\dot H^2(\R^3)}
+C\big\||\cdot|\nabla U_0\big\|_{L^\infty(\R^3)}\|V\|_{\dot{H}^1(\R^3)}\\&+C\big\||\cdot|U_0\big\|_{L^\infty(\R^3)}\|V\|_{\dot H^2(\R^3)}+C\big\||\cdot|\nabla U_0\big\|_{L^\infty(\R^3)}\|V\|_{\dot{H}^1(\R^3)}\\&+C\big\||\cdot|U_0\big\|_{L^\infty(\R^3)}\|U_0\|_{\dot H^2(\R^3)}
+C\big\||\cdot|\nabla U_0\big\|_{L^\infty(\R^3)}\|U_0\|_{\dot{H}^1(\R^3)}.
\end{align*}
By resorting to Lemma \ref{RNS}, we know that
\begin{equation*}
\big\||\cdot|\partial_{x_k} V\big\|_{H^2(\R^3)}<+\infty.
\end{equation*}
This estimate together with the embedding theorem leads to
\[\sup_{x\in\R^3}|x||\nabla V|(x)<+\infty.\]
We finish the proof of the proposition.
\end{proof}

\begin{prop}\label{prop-2p}
	Let the couple $(V,P)\in H^1(\R^3)\times L^2(\R^3)$  satisfy  \eqref{eq.weak}.
	Then we have
		\begin{equation*}
	\sup_{x\in\R^3}|x|^2|\nabla V|(x)+\sup_{x\in\R^3}|x|^2|\nabla P|(x)\leq C(U_0).
	\end{equation*}
\end{prop}
\begin{proof}

	With the help of Proposition \ref{prop-E}, we write
	\[V^\varepsilon(x)=\int_0^1\int_{\R^3}\Phi(x-y,1-s) s^{-\frac32} F^\varepsilon\big(y/\sqrt{s}\big)\,\mathrm{d}y\mathrm{d}s,\]
	where
	\[F^\varepsilon(x)\triangleq\frac{1}{1+\varepsilon|x|^2}\left(\nabla P+V\cdot\nabla V+U_0\cdot\nabla V+V\cdot\nabla U_0+U_0\cdot\nabla U_0\right).\]
	From Proposition \ref{prop-nv},  there exists a constant $C>0$ such that
	\begin{equation}
	\left|\nabla P+V\cdot\nabla V+U_0\cdot\nabla V+V\cdot\nabla U_0+U_0\cdot\nabla U_0\right|\leq C\big(1+ |x|\big)^{-2}.
	\end{equation}
	Therefore,
	\begin{equation*}
	\begin{split}
	|x|^2\left|V^\varepsilon\right|(x)\leq&2\int_0^1\int_{\R^3}\Phi(x-y,1-s) s^{-\frac32}\left| |y|^2F^\varepsilon\big(y/\sqrt{s}\big)\right|\,\mathrm{d}y\mathrm{d}s\\
	&+2\int_0^1\int_{\R^3}|x-y|^2\Phi(x-y,1-s) s^{-\frac32} \left|F^\varepsilon\big(y/\sqrt{s}\big)\right|\,\mathrm{d}y\mathrm{d}s.
	\end{split}
	\end{equation*}
	On one hand,
	\begin{equation*}
	\int_0^1\int_{\R^3}\Phi(x-y,1-s) s^{-\frac32}\left| |y|^2F^\varepsilon\big(y/\sqrt{s}\big)\right|\,\mathrm{d}y\mathrm{d}s
	\leq \sup_{x\in\R^3}|x|^2\big|F^\varepsilon\big|(x)\int_0^1s^{-\frac12}\,\mathrm{d}s<+\infty.
	\end{equation*}
	On the other hand,
	\begin{equation*}
	\begin{split}
	&\int_0^1\int_{\R^3}|x-y|^2\Phi(x-y,1-s) s^{-\frac32} \left|F^\varepsilon\big(y/\sqrt{s}\big)\right|\,\mathrm{d}y\mathrm{d}s\\
	\leq&4\int_0^1(1-s)\int_{\R^3}\frac{|x-y|^2}{4(1-s)}\Phi(x-y,1-s) s^{-\frac32} \left|F^\varepsilon\big(y/\sqrt{s}\big)\right|\,\mathrm{d}y\mathrm{d}s\\
	 \leq&\int_0^1(1-s)s^{-\frac32}\left\|\frac{|\cdot|^2}{4(1-s)}\Phi(\cdot,1-s)\right\|_{L^2(\R^3)}\left\|F^\varepsilon\big(y/\sqrt{s}\big)\right\|_{L^2(\R^3)}\,\mathrm{d}s\\
	\leq&C\left\| |\cdot|^2 \Phi(\cdot,1 )\right\|_{L^2(\R^3)}\left\|F \right\|_{L^2(\R^3)}\int_0^1(1-s)^{\frac14}s^{-\frac34}\,\mathrm{d}s,
	\end{split}
	\end{equation*}
	where
	\[F\triangleq \nabla P+V\cdot\nabla V+U_0\cdot\nabla V+V\cdot\nabla U_0+U_0\cdot\nabla U_0.\]
	By the H\"older inequality, we find that
	\begin{align*}
	\|F\|_{L^2(\R^3)}\leq &\|\nabla P\|_{L^2(\R^3)}+\|V\|_{L^\infty(\R^3)}\|\nabla V\|_{L^2(\R^3)}+\|U_0\|_{L^\infty(\R^3)}\|\nabla V\|_{L^2(\R^3)}\\&+\|V\|_{L^\infty(\R^3)}\|\nabla U_0\|_{L^2(\R^3)}+\|U_0\|_{L^\infty(\R^3)}\|\nabla U_0\|_{L^2(\R^3)}\\
	\leq&C\|V\|_{L^\infty(\R^3)}\|\nabla V\|_{L^2(\R^3)}+C\|U_0\|_{L^\infty(\R^3)}\|\nabla V\|_{L^2(\R^3)}\\&+C\|V\|_{L^\infty(\R^3)}\|\nabla U_0\|_{L^2(\R^3)}+C\|U_0\|_{L^\infty(\R^3)}\|\nabla U_0\|_{L^2(\R^3)}.
	\end{align*}
	Combining all these estimates, we finally obtain
	\begin{equation*}
	\sup_{x\in\R^3}|x|^2\left|V^\varepsilon\right|(x)\leq C(U_0),
	\end{equation*}
	and then we get by taking $\varepsilon\to0+$ that the limit $\overline V=\int_0^1\int_{\R^3}\Phi(x-y,1-s) s^{-\frac32} F^0\big(y/\sqrt{s}\big)\,\mathrm{d}y\mathrm{d}s$ satisfies \begin{equation*}
	\sup_{x\in\R^3}|x|^2\left|\overline V\right|(x)\leq C(U_0).
	\end{equation*}
First, we consider that $W\in H^1(\R^3)$  with $|x|W\in H^1(\R^3)$ solves the following linear equations with $f\in L^2(\R^3)$:
\[-\Delta W-\frac12x\cdot\nabla W-\frac12W=f.\]
It is obvious that $W$ is unique in the space of $H^1(\R^3)$. Indeed, suppose that $\overline W\in H^1(\R^3)$  with $|x|\overline W\in H^1(\R^3)$ is
another solution of the above linear equations. Then the difference $\delta W\triangleq W-\overline W$ satisfies
\[-\Delta\delta W-\frac12x\cdot\nabla \delta W-\frac12\delta W=0\]
This implies $\|\delta W \|_{ H^1(\R^3)}=0$, and then we get the uniqueness. This yields hat $V=\overline V$, and we have

	\begin{equation*}
	\sup_{x\in\R^3}|x|^2\left|V\right|(x)\leq C(U_0).
	\end{equation*}
	In the same fashion as in  proving the above estimates, we can show that
	\begin{equation*}
	\sup_{x\in\R^3}|x|^2\left|\nabla V\right|(x)\leq C(U_0).
	\end{equation*}
	Next we show the decay estimate for the pressure $P$. Recall that
	\begin{equation*}
	-\Delta P=\text{div\,div}\,\big( V\otimes V+V\otimes U_0+U_0\otimes V+U_0\otimes U_0\big),
	\end{equation*}
one writes
	\begin{equation*}
	\begin{split}
	P=&\frac{1}{4\pi}\int_{\R^3}\frac{1}{|x-y|}\text{div\,div}\,\big( V\otimes V+V\otimes U_0+U_0\otimes V+U_0\otimes U_0\big)(y)\,\mathrm{d}y\\
	=&\sum_{i,j=1}^3\frac{1}{4\pi}\int_{\R^3}\frac{1}{|x-y|}\left( \partial_{x_j}V^i\partial_{x_i} V^j+\partial_{x_j}V^i\partial_{x_i} U_0^j+\partial_{x_j}U_0^i\partial_{x_i} V^j+\partial_{x_j}U_0^i\partial_{x_i} U_0^j\right)(y)\,\mathrm{d}y.
	\end{split}
	\end{equation*}
	Furthermore, we have
	\begin{equation*}
	\begin{split}
	\partial_{x_k}P=&\sum_{i,j=1}^3\int_{\R^3}K_{i,k}(x,y)\left(\partial_{x_j} V^i V^j+\partial_{x_j}V^i U_0^j+\partial_{x_j}U_0^iV^j+\partial_{x_j}U_0^iU_0^j\right)(y)\,\mathrm{d}y.
	\end{split}
	\end{equation*}
	Multiplying the above equality by $|x|^2$, we readily have
	\begin{equation*}
	\begin{split}
	|&x|^2\partial_{x_{k}}P(x)\\
	\leq&C\sum_{i,j=1}^3\Big(\int_{\R^3}\frac{1}{|x-y|^2}|y|^2\left|\partial_{x_j} V^i \partial_{x_i}V^j+\partial_{x_j}V^i \partial_{x_i}U_0^j+\partial_{x_j}U_0^i\partial_{x_i}V^j+\partial_{x_j}U_0^i\partial_{x_i}U_0^j\right|(y)\,\mathrm{d}y\\
	&+\int_{\R^3}\left| \partial_{x_j} V^i \partial_{x_i}V^j+\partial_{x_j}V^i \partial_{x_i}U_0^j+\partial_{x_j}U_0^i\partial_{x_i}V^j+\partial_{x_j}U_0^i\partial_{x_i}U_0^j\right|(y)\,\mathrm{d}y\Big).
	\end{split}
	\end{equation*}
	By the H\"older inequality, we see that
	\begin{align*}
	 \int_{\R^3}\left| \partial_{x_j} V^i \partial_{x_i}V^j+\partial_{x_j}V^i \partial_{x_i}U_0^j+\partial_{x_j}U_0^i\partial_{x_i}V^j+\partial_{x_j}U_0^i\partial_{x_i}U_0^j\right|\,\mathrm{d}y
	\leq C\|V\|^2_{\dot H^1(\R^3)}+C\|U_0\|^2_{\dot H^1(\R^3)}.
	\end{align*}
	By the generalized Young inequality in Lemma \ref{GHY}, we get
	\begin{equation*}
	\begin{split}
	&\int_{\R^3}\frac{1}{|x-y|^2}|y|^2\left| \partial_{x_j} V^i \partial_{x_i}V^j+\partial_{x_j}V^i \partial_{x_i}U_0^j+\partial_{x_j}U_0^i\partial_{x_i}V^j+\partial_{x_j}U_0^i\partial_{x_i}U_0^j\right|(y)\,\mathrm{d}y\\
	\leq &C\left\||\cdot|^2\partial_{x_j} V^i \partial_{x_i}V^j+\partial_{x_j}V^i \partial_{x_i}U_0^j+\partial_{x_j}U_0^i\partial_{x_i}V^j+\partial_{x_j}U_0^i\partial_{x_i}U_0^j(y)\right\|_{L^{3,1}(\R^3)}\\
	\leq&C\big\||\cdot|\nabla V\big\|_{L^{6,2}(\R^3)}^2+C\big\||\cdot|^2U_0\big\|_{L^{\infty}(\R^3)}\left(\big\||\nabla U_0\big\|_{L^{3,1}(\R^3)}
	+\big\|\nabla V\big\|_{L^{3,1}(\R^3)}\right)\\
	\leq& C\|E\|_{H^2(\R^3)}^2+C\big\||\cdot|^2U_0\big\|_{L^{\infty}(\R^3)}\left(\big\||U_0\big\|_{H^2(\R^3)}
	+\big\|V\big\|_{H^2(\R^3)}\right).
	\end{split}
	\end{equation*}
	Combining both estimates, we get
	\begin{equation*}
	\sup_{x\in\R^3}|x|^2|\nabla P|(x)\leq C(U_0).
	\end{equation*}
	So we complete the proof the proposition.
\end{proof}
\begin{proof}[Proof of Theorem \ref{thm-1-I}]
	We calculate
\begin{equation}
\begin{split}
-\Delta\big(|x|^3\partial_{x_k}P\big) =&-12|x|\partial_{x_{k}}P-6|x|x\cdot\nabla\partial_{x_{k}}P-|x|^3\Delta\partial_{x_k}P\\
=&12|x|\partial_{x_{k}}P-6\text{div}\,\big(|x|x\partial_{x_{k}}P\big)-\partial_{x_k}\big(|x|^3\Delta P\big)+3|x|x_k\Delta P.
\end{split}
\end{equation}
Thus, we have
\begin{equation*}
\begin{split}
&\big\|\big(|\cdot|^3\partial_{x_k}P\big)\big\|_{\dot{B}^0_{\infty,\infty}(\R^3)}\\
\leq& C\big\||\cdot|\partial_{x_{k}}P\big\|_{\dot{B}^{-2}_{\infty,\infty}(\R^3)}+C\big\||\cdot|^2\partial_{x_{k}}P\big\|_{\dot{B}^{-1}_{\infty,\infty}(\R^3)}+\big\||\cdot|^3\Delta P\big\|_{\dot{B}^{-1}_{\infty,\infty}(\R^3)}+\big\||\cdot|^2\Delta P\big\|_{\dot{B}^{-2}_{\infty,\infty}(\R^3)}\\
\leq&C\big\||\cdot|\partial_{x_{k}}P\big\|_{L^{\frac32,\infty}(\R^3)}+C\big\||\cdot|^2\partial_{x_{k}}P\big\|_{L^{3,\infty}(\R^3)}+\big\||\cdot|^3\Delta P\big\|_{\dot{B}^{-1}_{\infty,\infty}(\R^3)}+\big\||\cdot|^2\Delta P\big\|_{\dot{B}^{-2}_{\infty,\infty}(\R^3)}\\
\leq&C\big\||\cdot|^2\partial_{x_{k}}P\big\|_{L^{3,\infty}(\R^3)}+\big\||\cdot|^3\Delta P\big\|_{\dot{B}^{-1}_{\infty,\infty}(\R^3)}+\big\||\cdot|^2\Delta P\big\|_{\dot{B}^{-2}_{\infty,\infty}(\R^3)}.
\end{split}
\end{equation*}
Since
\[-\Delta P=\mathrm{div}\left(V\cdot\nabla V+U_0\cdot\nabla V+V\cdot\nabla U_0+U_0\cdot\nabla U_0\right),\]
 we readily have by Proposition \ref{prop-2p} that
\begin{equation*}
\big|\Delta P\big|(x)\leq C|x|^{-4}.
\end{equation*}
Furthermore, we get from Lemma \ref{lemma-young} that
\begin{align*}
\big\||\cdot|^3\Delta P\big\|_{\dot{B}^{-1}_{\infty,\infty}(\R^3)}\leq&C\big\||\cdot|^3\Delta P\big\|_{L^{3,\infty}(\R^3)}\\
\leq&C\big\||\cdot|^4\Delta P\big\|_{L^{3,\infty}(\R^3)}\big\|1/|\cdot|\big\|_{L^{3,\infty}(\R^3)}.
\end{align*}
Similarly, we have
\begin{align*}
\big\||\cdot|^2\Delta P\big\|_{\dot{B}^{-2}_{\infty,\infty}(\R^3)}\leq&C\big\||\cdot|^3\Delta P\big\|_{L^{\frac32,\infty}(\R^3)}\\
\leq&C\big\||\cdot|^4\Delta P\big\|_{L^{3,\infty}(\R^3)}\left\|1/|\cdot|^2\right\|_{L^{\frac32,\infty}(\R^3)}.
\end{align*}
We see that
\begin{equation*}
\begin{split}
|x|^2\partial_{x_{k}}P(x)
=&\sum_{i,j=1}^3\Big(\int_{\R^3}K_{i,k}(x,y)|y|^2\left(\partial_{x_j} V^i V^j+\partial_{x_j}V^iU_0^j+\partial_{x_j}U_0^iV^j+\partial_{x_j}U_0^iU_0^j\right)\,\mathrm{d}y\\
&+\int_{\R^{3}}K_{i,k}(x,y)(|x|-|y|)^2\left(\partial_{x_j} V^i V^j+\partial_{x_j}V^iU_0^j+\partial_{x_j}U_0^iV^j+\partial_{x_j}U_0^iU_0^j\right)\,\mathrm{d}y\\
&+2\int_{\R^{3}}K_{i,k}(x,y)(|x|-|y|)|y|\left(\partial_{x_j} V^i V^j+\partial_{x_j}V^iU_0^j+\partial_{x_j}U_0^iV^j+\partial_{x_j}U_0^iU_0^j\right)\,\mathrm{d}y\Big)\\
\triangleq&I_1+I_2+I_3.
\end{split}
\end{equation*}
Properties of Calder\'on-Zygmund singular operator enable us to conclude
\begin{align*}
\|I_1\|_{L^{3,\infty}(\R^3)}\leq& C\left\||\cdot|^2\left(\partial_{x_j} V^i V^j+\partial_{x_j}V^iU_0^j+\partial_{x_j}U_0^iV^j+\partial_{x_j}U_0^iU_0^j\right)\right\|_{L^{3,\infty}(\R^3)}\\
\leq&C\left\|\left(|\cdot|^2\partial_{x_j} V^i\right)\right\|_{L^\infty(\R^3)}\left(\left\|V^j \right\|_{L^{3,\infty}(\R^3)}+\left\|U_0^j \right\|_{L^{3,\infty}(\R^3)}\right)\\
&+C\left\|\left(|\cdot|^2\partial_{x_j} U_0^i\right)\right\|_{L^\infty(\R^3)}\left(\left\|V^j \right\|_{L^{3,\infty}(\R^3)}+\left\|U_0^j\right\|_{L^{3,\infty}(\R^3)}\right).
\end{align*}
Integrating by parts leads to
\begin{equation*}
\begin{split}
 I_2 =&-\int_{\R^{3}}\partial_{x_j}\left(K_{i,k}(x,y)(|x|-|y|)^2\right)\left( V^i V^j+V^iU_0^j+U_0^iV^j+U_0^iU_0^j\right)\,\mathrm{d}y\\
 \leq& C\int_{\R^{3}}\frac{1}{|x-y|^2}\left|V^i V^j+V^iU_0^j+U_0^iV^j+U_0^iU_0^j\right|\,\mathrm{d}y.
 \end{split}
\end{equation*}
By the generalized Young inequality in Lemma \ref{GHY}, we obtain
\begin{align*}
\big\|I_2\big\|_{L^{3,\infty}(\R^3)}\leq &C\left\|V^i V^j+V^iU_0^j+U_0^iV^j+U_0^iU_0^j\right\|_{L^{\frac32,\infty}(\R^3)}\\
\leq&C\|V\|^2_{L^{3,\infty}(\R^3)}+C\|U_0\|^2_{L^{3,\infty}(\R^3)}.
\end{align*}
Since
\begin{equation*}
\big|I_3\big|\leq \int_{\R^{3}}\frac{1}{|x-y|^2}|y|\left|\partial_{x_j} V^i V^j+\partial_{x_j}V^iU_0^j+\partial_{x_j}U_0^iV^j+\partial_{x_j}U_0^iU_0^j\right|\,\mathrm{d}y,
\end{equation*}
we have
\begin{align*}
\big\|I_3\big\|_{L^{3,\infty}(\R^3)}\leq& C\left\||\cdot|^2\left(\partial_{x_j} V^i V^j+\partial_{x_j}V^iU_0^j+\partial_{x_j}U_0^iV^j+\partial_{x_j}U_0^iU_0^j\right)\right\|_{L^{\frac32,\infty}(\R^3)}\\
\leq&C\left\|\left(|\cdot| V^j\right)\right\|_{L^\infty(\R^3)}\left(\left\|\partial_{x_j}V^i \right\|_{L^{\frac32,\infty}(\R^3)}+\left\|\partial_{x_j}U_0^i \right\|_{L^{\frac32,\infty}(\R^3)}\right)\\
&+C\left\||\cdot| U_0^j\right\|_{L^\infty(\R^3)}\left(\left\|\partial_{x_j}V^i \right\|_{L^{\frac32,\infty}(\R^3)}+\left\|\partial_{x_j}U_0^i \right\|_{L^{\frac32,\infty}(\R^3)}\right).
\end{align*}
Collecting all these estimates, we eventually obtain that
\begin{equation*}
\big\||\cdot|^2\nabla P\big\|_{L^{3,\infty}(\R^3)}<+\infty.
\end{equation*}
From  Proposition \ref{prop-2p}, it follows that
\[\sup_{x\in\R^3}\langle x\rangle^2\big(|V|+|\nabla V|\big)(x)<+\infty.\]
Moreover, by \cite[Equality (4.11)]{JS}, we get
 $$
| V|(x)\leq C(1+|x|)^{-3}\log(1+|x|)\qquad\text{for all}\,\,\,x\in\R^3$$
and  \begin{equation}\label{eq.HI}
|D V|(x)\leq C(1+|x|)^{-3}\qquad\text{for all}\,\,\,x\in\R^3.
\end{equation}
We finish the proof of Theorem \ref{thm-1-I}.
\end{proof}
%%%%%%%%%%%%%%%%%%%%%%%%%%%%%%%%%%%%%%%%%%%%%%%
\section{Proof of Theorem \ref{T1.1}}
%%%%%%%%%%%%%%%%%%%%%%%%%%%%%%%%%%%%%%%%%%%%%%%
In this section, we give a  proof of Theorem \ref{T1.1}. First of all, we focus on the existence of the forward self-similar solution. Letting $$u_{\rm L}(x,t)\triangleq t^{\frac{1-2\alpha}{2\alpha}}U_0\big(x/t^{\frac{1}{2\alpha}}\big), \qquad v(x,t)\triangleq t^{\frac{1-2\alpha}{2\alpha}}V\big(x/t^{\frac{1}{2\alpha}}\big)$$ and the pressure $p(x,t)\triangleq t^{\frac{1-2\alpha}{\alpha}}P\big(x/t^{\frac{1}{2\alpha}}\big)$, where $V$ and $P$ was constructed in \eqref{eq.weak}. We compute
\begin{equation*}
\big(\partial_t+(-\Delta)^\alpha\big) u_{\text{L} }(x,t)=\frac{1}{t}\left((-\Delta)^\alpha U_0-\frac{1}{2\alpha}\left(\frac{x}{ t^{\frac{1}{2\alpha}}}\cdot U_0-\big(2\alpha-1\big) U_0\right)\right)\Big(\frac{x}{t^{\frac{1}{2\alpha}}}\Big)=0,
\end{equation*}
and
\begin{align*}
\big(\partial_t+(-\Delta)^\alpha\big) v(x,t)=&\frac{1}{t}\left((-\Delta)^\alpha V-\frac{1}{2\alpha}\left(\frac{x}{ t^{\frac{1}{2\alpha}}}\cdot V-\big(2\alpha-1\big) V\right)\right)\Big(\frac{x}{t^{\frac{1}{2\alpha}}}\Big)\\
=&-\frac{1}{t}\big(\nabla P+V\cdot\nabla V+U_0\cdot\nabla V+V\cdot\nabla U_0+U_0\cdot\nabla U_0\big)\Big(\frac{x}{t^{\frac{1}{2\alpha}}}\Big)\\
=&-\big(\nabla P+v\cdot \nabla v+v_{\rm L}\cdot\nabla v+v\cdot\nabla u_{\rm L}+u_{\rm L}\cdot\nabla u_{\rm L}\big)\Big(\frac{x}{t^{\frac{1}{2\alpha}}}\Big).
\end{align*}
Now let us denote $u(x,t)\triangleq u_{\rm L}(x,t)+v(x,t)$, it is easy to verify that $u(x,t)$ is the solution of the fractional Navier-Stokes equations
\begin{equation*}
\partial_tu+u\cdot\nabla u+(-\Delta)^{\alpha}u+\nabla p=0\qquad \text{in}\quad \R^3\times(0,+\infty).
\end{equation*}
Next, we want to show that $u(x,t)\in \text{BC}_{\rm w}\big([0,+\infty);\,L^{(\frac{3}{2\alpha-1},\infty)}(\R^3)\big)$. Since $V\in H^{\alpha}(\R^3)$, we know by the embedding theorem that
\begin{equation*}
\|v(t)\|_{L^{\frac{3}{2\alpha-1}}(\R^3)}=\|V\|_{L^{\frac{3}{2\alpha-1}}(\R^3)}\leq C\|V\big\|_{H^\alpha(\R^3)}<+\infty
\end{equation*}
as long as $\alpha\in(5/6,\,1].$ On the other hand,  by Proposition \ref{prop-U1}, we have

\[\big\|u_{\rm L}(t)\big\|_{L^{(\frac{3}{2\alpha-1},\infty)}(\R^3)}\leq C \|u_0\|_{L^{(\frac{3}{2\alpha-1},\infty)}(\R^3)}.\]

Combining both estimates yields that $\|u\|_{L^\infty\big([0,+\infty);\,L^{(\frac{3}{2\alpha-1},\infty)}(\R^3)\big)}$ is  bounded. Now, we begin to show the weak weak continuity with respect to time $t$. For the linear part $u_{\rm L}(x,t)$, it is obvious that $u_{\rm L}\in C_{\rm w}([0,+\infty);\,L^{(\frac{3}{2\alpha-1},\infty)}(\R^3)) $ and $u_{\rm L}\to u_0$ in the weak sense of $L^{(\frac{3}{2\alpha-1},\infty)}(\R^3)$ as $t$ goes to $0+$. So we  need to show that $v\in C_{\rm w}([0,+\infty);\, L^{(\frac{3}{2\alpha-1},\infty)}(\R^3))$. Since $\alpha\in(5/6,\,1]$ and $V\in H^\alpha(\R^3)$, we know that for each $2\leq p\leq \frac{6}{3-2\alpha},$
\[\|V\|_{L^p(\R^3)}\leq C\|V\|_{H^\alpha(\R^3)}.\]
This estimate together with the fact $\big\|v(t)\big\|_{L^p}=t^{\frac{1}{2\alpha}(1+\frac{3}{p})-1}\|V\|_{L^{p}(\R^{3})}$ means that for every $2\leq p\leq \frac{6}{3-2\alpha},$
\[\|v(t)\|_{L^p(\R^3)}\leq Ct^{\frac{1}{2\alpha}(1+\frac{3}{p})-1}.\]
It follows  that for $2\leq p<\frac{3}{2\alpha-1},$
\[\|v(t)\|_{L^p(\R^3)}\to 0\qquad\text{as}\,\,\, t\to0+.\]
With this property in hand, we can infer that
$v(x,t)\rightharpoonup 0$ in $ L^{\frac{3}{2\alpha-1}}(\R^{3})$ as $t\to0+,$
which implies that  $u(x,t)\to u_0(x)$ as $t\to 0+$ in the weak sense of $L^{\frac{3}{2\alpha-1}}(\R^{3})$.

Finally, we  show the higher decay estimate of solution $v$ for the case $\alpha=1$. From Theorem~\ref{thm-1-I}, we know that
\[ \big(1+|x|^3\big)|V|(x)\leq C\log\big(1+|x|\big)\qquad\text{for all}\quad x\in\R^3.\]
Combining this estimate with the relation that $v(x,t)=\sqrt{t}^{-1}V\big(x/\sqrt{t}\big)$ enables us to conclude that
\begin{equation*}
\sqrt{t}\left(1+\frac{|x|^3}{(\sqrt{t})^3}\right)v(x,t)=\left(1+\frac{|x|^3}{(\sqrt{t})^3}\right)V\left(\frac{x}{\sqrt{t}}\right)\leq C\log\bigg(1+\frac{|x| }{\sqrt{t}}\bigg).
\end{equation*}
 Thus we have
 \begin{equation*}
 |v(x,t)|\leq C\frac{t}{t^{\frac32}+|x|^3}\log\bigg(1+\frac{|x| }{\sqrt{t}}\bigg)\qquad\text{for all}\quad (x,t)\in\R^3\times(0,+\infty).
 \end{equation*}
Thanks to Proposition \ref{prop-U1}, we know that
$
\sup_{x\in\R^3}\big(1+|x|\big)|U_{0}(x)|<+\infty.$
Since \[u_{\rm L}(x,t)=\frac{1}{\sqrt{t}}U_0\bigg(\frac{x}{\sqrt{t}}\bigg),\] we readily have  that  \begin{equation*}
|u_{\rm L}(x,t)|\leq C\frac{\sqrt t}{\sqrt t+|x|}\qquad\text{for all}\quad (x,t)\in\R^3\times(0,+\infty).
\end{equation*}
So we finish the proof of Theorem \ref{T1.1}.

\section*{Acknowledgments} The authors thank the  referee and the associated editor for their
invaluable comments  which helped improve the paper greatly.   This work is supported in part by the National Natural Science Foundation of China
 under grant  No.11671047, No.11871087,  No.11831004, No.11771423 and No.11826005.

\end{document}